\documentclass{article}

\title{Faithfulness of highest-weight modules for Iwasawa algebras in type D}
\author{Stephen Mann}

\usepackage{amsmath, amsthm, amssymb, amsfonts, mathtools,stmaryrd,hyperref,tikz-cd,setspace,etoolbox,expl3,pgfkeys,pgfopts,tikz,xparse,xstring,dynkin-diagrams}
\theoremstyle{definition}
\newtheorem{mydef}{Definition}[subsection]

\theoremstyle{plain}
\newtheorem{mythm}[mydef]{Theorem}
\newtheorem{mylem}[mydef]{Lemma}
\newtheorem{mycor}[mydef]{Corollary}
\newtheorem{myprop}[mydef]{Proposition}
\newtheorem{Theorem}{Theorem}

\theoremstyle{remark}
\newtheorem{myrem}[mydef]{Remark}

\DeclareMathOperator{\Ann}{Ann}
\DeclareMathOperator{\End}{End}
\newcommand{\Ug}{\ensuremath{\widehat{U(\mathfrak{g})_{n,K}}}}

\newcommand{\ug}{\ensuremath{U(\mathfrak{g}_{K})}}

\newcommand{\Up}{\ensuremath{\widehat{U(\mathfrak{p})_{n,K}}}}

\newcommand{\up}{\ensuremath{U(\mathfrak{p}_{K})}}

\newcommand{\Ul}{\ensuremath{\widehat{U(\mathfrak{l})_{n,K}}}}

\newcommand{\ul}{\ensuremath{U(\mathfrak{l}_{K})}}

\newcommand{\upp}{\ensuremath{U(\mathfrak{p}'_{K})}}

\newcommand{\Ugg}{\ensuremath{\widehat{U(\mathfrak{g}')_{n,K}}}}
\newcommand{\ugg}{\ensuremath{U(\mathfrak{g}'_{K})}}

\newcommand{\Un}{\ensuremath{\widehat{U(\mathfrak{n})_{n,K}}}}

\newcommand{\un}{\ensuremath{U(\mathfrak{n}_{K})}}

\newcommand{\Of}{\mathcal{O}_{F}}

\begin{document}

\maketitle

\begin{abstract}
\noindent We prove that infinite-dimensional highest-weight modules are faithful for Iwasawa algebras corresponding to a simple Lie algebra of type D. We use this to prove that all non-zero two-sided ideals of the Iwasawa algebra have finite codimension in this case.
\end{abstract}

\section{Introduction}

Let $p$ be an odd prime, $F$ a finite field extension of $\mathbb{Q}_{p}$, and $K$ a finite field extension of $F$. Write $\Of$ for the ring of integers of $F$, $R$ for the ring of integers of $K$, and let $\pi$ be a uniformiser of $K$.

Now suppose $G$ is an $F$-uniform group as defined in \cite[\S 2.4]{me}, and suppose the corresponding Lie algebra $L_{G}$ has the form $L_{G} = p^{n+1} \mathfrak{g}$, where $\mathfrak{g}$ is an $\Of$-lattice of a split simple $F$-Lie algebra, and $\mathfrak{g}$ is spanned by a Chevalley basis, with a Cartan subalgebra $\mathfrak{h}$.

The key example for this paper will be when $\mathfrak{g}$ has type $D_{m}$, in which case $G$ is the subgroup $\ker \left( SO_{2m}(\Of) \rightarrow SO_{2m}(\Of / p^{n+1}\Of) \right)$ of $SO_{2m}(\Of)$.

We can define the Iwasawa algebras corresponding to $G$ over $R$ and $K$ as follows:
\begin{mydef}
\begin{itemize}
\item $RG = \underset{N \unlhd_{o} G}{\varprojlim} R[G/N]$, 

\item $KG = K \otimes_{R} RG$.

\end{itemize}

\end{mydef}

\noindent The motivation behind studying Iwasawa algebras is as follows. If we wish to study the representations of $G$ over $K$, the class of all representations will turn out to be too large to reasonably work with. As such, we will want to restrict to studying representations with some continuity conditions. Then whereas modules over the classical group algebra $K[G]$ precisely correspond to $K$-representations of $G$, the finitely generated modules over $KG$ will instead correspond to representations of $G$ with certain continuity conditions.

More precisely, there is a correspondence between finitely generated $KG$-modules and admissible $K$-Banach space representations of $G$, see \cite[Theorem 3.5]{Schneider2}. 

There is interest in studying the algebraic properties of Iwasawa algebras. One of the main questions is if we can classify the prime ideals of Iwasawa algebras, which is an ongoing project, see for instance \cite{nilpotent} and \cite{verma}. This paper is a part of this project.

As we are working in the setting of semisimple groups, we will be able to use the theory of highest-weight modules for Lie algebras.

We can define an algebra $\Ug =\underset{i>0}{\varprojlim} U(p^{n}\mathfrak{g}_{R}) / \pi^{i} U(p^{n} \mathfrak{g}_{R})$, which is a completion of the usual universal enveloping algebra $\ug$. Then $KG$ can be considered as a subalgebra of $\Ug$, by \cite[Corollary 2.5.5]{me}.

If $M$ is a highest-weight $\ug$-module, then we can define a completion $\widehat{M} = \Ug \otimes_{\ug} M$.

\begin{mydef}
An \emph{affinoid highest-weight module} for $\Ug$ is a module of the form $\hat{M} = \Ug \otimes_{\ug} M$, where $M$ is a highest-weight module with highest weight $\lambda$ satisfying $\lambda(p^{n} \mathfrak{h}_{R}) \subseteq R$.
\end{mydef}

The condition that $\lambda(p^{n} \mathfrak{h}_{R}) \subseteq R$ is imposed because otherwise we will have $\hat{M} = 0$.

In \cite{verma}, two questions are asked.

\bigskip

\noindent \emph{Question A}: does every primitive ideal of $\Ug$ with $K$-rational central character arise as the annihilator of a simple affinoid highest-weight module?

\bigskip

\noindent \emph{Question B}: is every affinoid highest-weight module that is infinite-dimensional over $K$ faithful as a $KG$-module?

\bigskip

\noindent If both questions are answered positively, we can prove (in the case where $F = \mathbb{Q}_{p}$ and $\pi = p$) that every non-zero two-sided ideal of $KG$ has finite codimension. Question A has already been answered positively (see \cite{Ioan}), so it is our goal to answer question B.

Question B has already been answered positively for the case of generalised Verma modules in \cite{gvmcase}, and also in the case where $\mathfrak{g}$ has type A in \cite{me}.

In this paper we will give a positive answer to Question B in the case where $\mathfrak{g}$ has type D.

\begin{Theorem}
\label{maintheorem}
Suppose $G = \ker \left( SO_{2m}(\Of) \rightarrow SO_{2m}(\Of / p^{n+1} \Of) \right)$ for some integers $n \geq 0$ and $m \geq 4$, and suppose $p \geq 5$. Then any infinite-dimensional affinoid highest-weight module for $KG$ is a faithful $KG$-module.
\end{Theorem}
  
\begin{Theorem}
\label{maintheoremapplication}
Suppose $G$ is a compact open subgroup of \\ $\ker \left( SO_{2m}(\mathbb{Z}_{p}) \rightarrow SO_{2m}(\mathbb{Z}_{p} /  p\mathbb{Z}_{p}) \right)$ for some integer $m \geq 4$, and suppose $p \geq 5$. Also assume $R$ has uniformiser $p$.

\begin{itemize}
\item Every two-sided ideal of $KG$ is either $0$ or has finite $K$-codimension in $KG$.

\item Every non-zero prime ideal of $KG$ is the annihilator of a finite-dimensional simple module.
\end{itemize}
\end{Theorem}

Given a weight $\lambda$, there is a unique irreducible highest-weight module $L(\lambda)$ for $\ug$ with highest-weight $\lambda$. Moreover $L(\lambda)$ is infinite-dimensional precisely when $\lambda$ is not a dominant integral weight. We will see later that any infinite-dimensional affinoid highest-weight module has an infinite-dimensional subquotient of the form $\widehat{L(\lambda)}$ for some $\lambda$, and so the following theorem is enough to imply Theorem 1.

\begin{mythm}
\label{maintheoremvariant}
Suppose $G = \ker \left( SO_{2m}(\Of) \rightarrow SO_{2m}(\Of / p^{n+1} \Of) \right)$ for some integers $n \geq 0$ and $m \geq 4$, and suppose $p \geq 5$. Let $\lambda \in \mathfrak{h}_{K}^{*}$ be a weight that is not dominant integral, and such that $\lambda(\mathfrak{h}_{R}) \subseteq p^{n} R$. 

Then $\widehat{L(\lambda)}$ is a faithful $KG$-module.
\end{mythm}

\subsection{Main approach}

Suppose $G$ is an $F$-uniform group with corresponding Lie algebra $L_{G} = p^{n+1} \mathfrak{g}$, where $\mathfrak{g}$ is an $\Of$-lattice of a split semisimple $F$-Lie algebra. Write $\Phi$ for the root system of $\mathfrak{g}$, and fix a set of simple roots $\Delta$. 

Assume $\mathfrak{g}$ is spanned over $\Of$ by a Chevalley basis $\{e_{\alpha},f_{\alpha} \mid \alpha \in \Phi^{+} \} \cup \{h_{\alpha} \mid \alpha \in \Delta \}$ (in the sense of Definition \ref{chevalley}). Then we have a corresponding triangular decomposition $\mathfrak{g} = \mathfrak{n}^{+} \oplus \mathfrak{h} \oplus \mathfrak{n}^{-}$.

Let $\lambda : \mathfrak{h}_{K} \rightarrow K$ be a weight. Then we define a subset $\Delta_{\lambda} = \{ \alpha \in \Delta \mid \lambda(h_{\alpha}) \in \mathbb{N}_{0} \}$, a corresponding nilpotent subalgebra $\mathfrak{n}_{\Delta_{\lambda}} = \Of\{ f_{\alpha} \mid \alpha \in \Phi^{+}, \alpha \text{ has a non-zero coefficient of some element of } \Delta_{\lambda} \}$, and a corresponding parabolic subalgebra $\mathfrak{p}_{\Delta_{\lambda}}$ with $\mathfrak{g} = \mathfrak{p}_{\Delta_{\lambda}} \oplus \mathfrak{n}_{\Delta_{\lambda}}$.

We can then define a generalised Verma module $M_{\Delta_{\lambda}}(\lambda)$, which is a highest-weight module with highest weight $\lambda$. $M_{\Delta_{\lambda}}(\lambda)$ has the form $\ug \otimes_{U(\mathfrak{p}_{\Delta_{\lambda},K})} V_{\lambda}$, where $V_{\lambda}$ is a finite-dimensional, irreducible $\mathfrak{p}_{\Delta_{\lambda}}$-module.

Suppose $P_{\lambda}$ is the $F$-uniform subgroup of $G$ with  $L_{P_{\lambda}}=p^{n+1}\mathfrak{p}_{\Delta_{\lambda}}$, and $N_{\lambda}$ is the $F$-uniform subgroup with $L_{N_{\lambda}} = p^{n+1}\mathfrak{n}_{\Delta_{\lambda}}$. Then $V_{\lambda}$ can be considered as a $KP_{\lambda}$-module (see \cite[Lemma 2.2.17]{me}). 

So we can consider a $KG$-module $KG \otimes_{KP_{\lambda}} V_{\lambda}$.

Now there is a composition of $KG$-module homomorphisms:

\begin{equation*}
KG \otimes_{KP_{\lambda}} V_{\lambda} \hookrightarrow \widehat{M_{\Delta_{\lambda}}(\lambda)} \twoheadrightarrow \widehat{L(\lambda)}.
\end{equation*}

We will later consider a topology on $\widehat{M_{\Delta_{\lambda}}(\lambda)}$ which makes $KG \otimes_{KP_{\lambda}} V_{\lambda}$ a dense $KG$-submodule.

By \cite[Theorem 1]{me}, we know that $\widehat{M_{\Delta_{\lambda}}(\lambda)}$ is faithful over $KG$ (as long as $p$ is sufficiently large). The density then implies that $KG \otimes_{KP_{\lambda}} V_{\lambda}$ is a faithful $KG$-module.

So if we can show that the mapping $KG \otimes_{KP_{\lambda}} V_{\lambda} \rightarrow \widehat{L(\lambda)}$ is injective, it will follow that $\widehat{L(\lambda)}$ is a faithful $KG$-module. This will be our main approach for proving faithfulness of $\widehat{L(\lambda)}$.

We can reformulate the condition that the composition $KG \otimes_{KP_{\lambda}} V_{\lambda} \rightarrow \widehat{L(\lambda)}$ is injective by using the fact that $KG \otimes_{KP_{\lambda}} V_{\lambda}$ is linearly isomorphic to $KN_{\lambda} \otimes_{K} V_{\lambda}$, and using the fact that $V_{\lambda}$ embeds in $L(\lambda)$ and hence $\widehat{L(\lambda)}$.

We now state the following two conditions on $\lambda$.

\bigskip

\noindent \textbf{Condition (1)}: the multiplication map $KN_{\lambda} \otimes_{K} V_{\lambda} \rightarrow \widehat{L(\lambda)}$ is injective.

\bigskip

\noindent \textbf{Condition (2)}: the multiplication map $KN_{\lambda} \otimes_{K} L(\lambda) \rightarrow \widehat{L(\lambda)}$ is injective.

\bigskip

\noindent Note that Condition (2) is a stronger form of Condition (1). Either of these conditions is sufficient to prove that $\widehat{L(\lambda)}$ is faithful over $KG$. The vast majority of the time we will want to prove the weaker form, Condition (1).

\subsection{Notation}

We use the convention that $\mathbb{N}$ does not include $0$, and write $\mathbb{N}_{0}$ for the set of non-negative integers.

Throughout, $p$ is a fixed odd prime. We will always take $K / F$ to be finite field extensions of $\mathbb{Q}_{p}$, with rings of integers $R$ and $\mathcal{O}_{F}$ respectively, and $\pi$ a uniformiser of $K$.

We will take $\mathfrak{g}$ to be an $\Of$-lattice of a split simple $F$-Lie algebra with a root system $\Phi$, and a fixed set of simple roots $\Delta$. We will assume $\mathfrak{g}$ is spanned over $\Of$ by a Chevalley basis $\{ e_{\alpha},f_{\alpha} \mid \alpha \in \Phi^{+} \} \cup \{h_{\alpha} \mid \alpha \in \Delta \}$. 

Then $\mathfrak{g}$ has a corresponding triangular decomposition $\mathfrak{g} = \mathfrak{n}^{+} \oplus \mathfrak{h} \oplus \mathfrak{n}^{-}$. 

We will take $G$ to be the $F$-uniform group with corresponding Lie algebra $L_{G} = p^{n+1} \mathfrak{g}$, with the choice of $n \in \mathbb{N}_{0}$ fixed.

The Weyl group of $\Phi$ will be denoted $W$, and we write  $\rho = \frac{1}{2} \underset{\alpha \in \Phi^{+}}{\sum} \alpha$.

Suppose $\mathcal{L}$ is an $A$-Lie algebra, where $A$ is a commutative ring, and suppose $B$ is an $A$-algebra. Then we write $\mathcal{L}_{B} = \mathcal{L} \otimes_{A} B$, which is a $B$-Lie algebra.

Suppose $\lambda \in \mathfrak{h}_{K}^{*}$ is a weight for $\mathfrak{g}_{K}$. We will write:

\begin{itemize}
 \item $\Phi_{\lambda} = \{ \alpha \in \Phi \mid \langle \lambda + \rho, \alpha^{\lor} \rangle \in \mathbb{Z} \}$,

\item $\Delta_{\lambda} = \{ \alpha \in \Delta \mid \lambda(h_{\alpha}) \in \mathbb{N}_{0} \}$.
\end{itemize}

\begin{itemize}
\item The Verma module with highest weight $\lambda$ is denoted $M(\lambda)$, and the irreducible hightest-weight module with highest weight $\lambda$ is denoted $L(\lambda)$.

\item We will write $w_{\lambda}$ for the highest-weight vector generating $M_{\Delta_{\lambda}}(\lambda)$ (defined in \S \ref{gvm}) and $v_{\lambda}$ for the highest-weight vector generating $L(\lambda)$.
\end{itemize}

\noindent Suppose also $\alpha_{1} \in \Delta$.

\begin{itemize}

\item Given $\alpha \in \Phi^{+}$ and $\alpha_{1} \in \Delta$, we will use the notation $\alpha_{1} \in \alpha$ to mean that $\alpha$ has a non-zero $\alpha_{1}$-coefficient, and $\alpha_{1} \notin \alpha$ otherwise.

\item Given a highest-weight module $M$ of highest-weight $\lambda$ and $x \in M$ with weight decomposition $x = \sum x_{\mu}$, and given $\alpha_{1} \in \Delta$, we say the $\alpha_{1}$-free component of $x$ is the sum of all $x_{\mu}$ with $\mu = \lambda - \underset{\alpha \in \Delta}{\sum} c_{\alpha} \alpha$ such that $c_{\alpha_{1}} = 0$. If $x$ equals its $\alpha_{1}$-free component, we say $x$ is $\alpha_{1}$-free.
\end{itemize}

\noindent The dot action of $W$ on $\mathfrak{h}_{K}^{*}$ is defined by $w \cdot \lambda = w(\lambda + \rho) - \rho$. for $w \in W, \lambda \in \mathfrak{h}_{K}^{*}$.

\bigskip

\noindent Finally, if $S$ is some index set and $t \in \mathbb{N}_{0}^{S}$, we will write $|t| = \underset{s \in S}{\sum} t_{s}$.

\bigskip

\noindent \textbf{Acknowledgements:} I would like to thank Simon Wadsley for his continuing advice and help.

\section{Background}

For more detailed background we refer to \cite[Chapter 2]{me}.

\subsection{Preliminaries about Lie algebras}
\label{Liealgsec}

Assume that $\mathfrak{g}$ is an $\Of$-lattice (in the sense of \cite[Definition 2.1.1]{me}) of a split simple $F$-Lie algebra, with root system $\Phi$ and a fixed set of simple roots $\Delta$.

\begin{mydef}
\label{chevalley}
We say that an $\Of$-basis $\{h_{\alpha} \mid \alpha \in \Delta \} \cup \{e_{\alpha},f_{\alpha} \mid \alpha \in \Phi^{+} \}$ for $\mathfrak{g}$
is a \emph{Chevalley basis} if the following relations are satisfied:

\begin{itemize}
\item $[h_{\alpha}, h_{\beta}] = 0$ for all $\alpha, \beta \in \Delta$,

\item $[h_{\alpha}, e_{\beta}] = \langle \beta, \alpha^{\lor} \rangle e_{\beta}$ for all $\alpha \in \Delta, \beta \in \Phi^{+}$,

\item $[h_{\alpha}, f_{\beta}] = -\langle \beta, \alpha^{\lor} \rangle f_{\beta}$ for all $\alpha \in \Delta, \beta \in \Phi^{+}$,

\item $[e_{\alpha},f_{\alpha}] = h_{\alpha}$ for all $\alpha \in \Delta$.

\item Suppose we write $e_{-\alpha} = f_{\alpha}$ whenever $\alpha \in \Phi^{+}$. Then for $\alpha, \beta \in \Phi$ with $\alpha \neq - \beta$, we have $[e_{\alpha},e_{\beta}] = C_{\alpha,\beta}e_{\alpha + \beta}$ for some $C_{\alpha,\beta} \in  \mathbb{Z}$, where $C_{\alpha, \beta} \neq 0 $ if and only if $\alpha + \beta \in \Phi$, and where $C_{\alpha,\beta} =  C_{-\beta, -\alpha}$.

\end{itemize}
\end{mydef}

\noindent This is analogous to the concept of a Chevalley basis for a semisimple Lie algebra over an algebraically closed field of characteristic $0$, which is known to always exist. In our setting we will assume that a Chevalley basis exists, for example in our main case of $\mathfrak{g} = \mathfrak{so}_{2m}(\Of)$ it does exist.

From now on, we will assume that $\mathfrak{g}$ does have a Chevalley basis $\{h_{\alpha} \mid \alpha \in \Delta \} \cup \{e_{\alpha},f_{\alpha} \mid \alpha \in \Phi^{+} \}$.

We now look at some subalgebras of $\mathfrak{g}$. Given a subset $I \subseteq \Delta$, we can then define a corresponding root subsystem $\Phi_{I} = \mathbb{Z}I \cap \Phi$ with positive roots $\Phi_{I}^{+} = \Phi_{I} \cap \Phi^{+}$. Then we have the following subalgebras of $\mathfrak{g}$:

\begin{mydef}
\begin{itemize}
\label{subalgebras}
\item $\mathfrak{h}$ will denote the Cartan subalgebra spanned by $\{h_{\alpha} \mid \alpha \in \Delta \}$,

\item $\mathfrak{n}^{+}$ is the subalgebra spanned by all $e_{\alpha}$ for $\alpha \in \Phi^{+}$,

\item $\mathfrak{n}^{-}$ is the subalgebra spanned by all $f_{\alpha}$ for $\alpha \in \Phi^{+}$,

\item $\mathfrak{p}_{I}$ is the parabolic subalgebra spanned by $\{h_{\alpha} \mid \alpha \in \Delta \} \cup \{e_{\alpha} \mid \alpha \in \Phi^{+} \} \cup \{f_{\alpha} \mid \alpha \in \Phi_{I}^{+} \}$,

$\mathfrak{n}_{I}$ is the nilpotent subalgebra spanned by $\{ f_{\alpha} \mid \alpha \in \Phi^{+} \setminus \Phi_{I}^{+} \}$,

\item $\mathfrak{l}_{I}$ is the Levi subalgebra spanned by $\{h_{\alpha} \mid \alpha \in \Delta \} \cup \{e_{\alpha},f_{\alpha} \mid \alpha \in \Phi_{I}^{+} \}$,

\item $\mathfrak{g}_{I}$ is the semisimple subalgebra spanned by $\{h_{\alpha} \mid \alpha \in I \} \cup \{e_{\alpha},f_{\alpha} \mid \alpha \in \Phi_{I}^{+} \}$.

\end{itemize}

\end{mydef}

\noindent In particular $\mathfrak{g} = \mathfrak{n}^{+} \oplus \mathfrak{h} \oplus \mathfrak{n}^{-}$ and $\mathfrak{g} = \mathfrak{p}_{I} \oplus \mathfrak{n}_{I}$.

\bigskip

\noindent Finally, we will give an explicit description of the root system $D_{m}$, where $m \geq 4$. Take $\Delta = \{\alpha_{1},\dots,\alpha_{m}\}$, with the corresponding Dynkin diagram as follows.

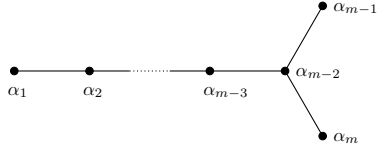
\begin{dynkinDiagram}[
edge length=1cm,
labels={1,2,m-3,m-2,m-1,m},
label directions={,,,right,,},
label macro/.code={\alpha_{\drlap{#1}}}
]D{}
\end{dynkinDiagram}

We can then describe the positive roots in $\Phi$ as follows:

\begin{itemize}
\item $\alpha_{i,j}:=\alpha_{i}+\dots+\alpha_{j}$ for $1 \leq i \leq j \leq m$,

\item $\beta_{i}:=\alpha_{i}+\dots+\alpha_{m-2}+\alpha_{m}$ for $1 \leq i \leq m-2$,

\item $\gamma_{i,j}:=\alpha_{i}+\dots+\alpha_{j-1}+2\alpha_{j}+\dots+2\alpha_{m-2}+\alpha_{m-1}+\alpha_{m}$ for $1 \leq i < j \leq m-2$.
\end{itemize}

\begin{proof}
The standard description of the root system $D_{m}$ (e.g. \cite[12.1]{otherHumphreys}) is as follows: let $e_{1},\dots,e_{m}$ be the standard basis for $\mathbb{R}^{m}$. Then our simple roots are $\alpha_{i} = e_{i} - e_{i+1}$ for $1 \leq i \leq m$ and $\alpha_{m} = e_{m-1}+e_{m}$. The positive roots are then $e_{i} - e_{j},e_{i}+e_{j}$ for $i < j$.

We then describe these positive roots in terms of the simple roots, using explicit calculation.

\begin{itemize}
\item For $1 \leq i < j \leq m$, $e_{i} - e_{j} = \alpha_{i}+\dots+\alpha_{j-1} = \alpha_{i,j-1}$.

\item $e_{m-1}+e_{m} = \alpha_{m}$, and for $1 \leq i \leq m-2$, $e_{i} + e_{m} = \alpha_{i} + \dots + \alpha_{m-2} + \alpha_{m} = \beta_{i}$.

\item For $1 \leq i \leq m-2$, $e_{i} + e_{m-1} = \alpha_{1} + \dots + \alpha_{m-1} + \alpha_{m} = \alpha_{i,m}$.

\item For $1 \leq i < j \leq m-2$, $e_{i} + e_{j}  = \alpha_{i} + \dots + \alpha_{j-1} + 2 \alpha_{j} + \dots + \alpha_{m-2} + \alpha_{m-1} + \alpha_{m} = \gamma_{i,j}$. 
\end{itemize}

\end{proof}

\subsection{Completions}
\label{completions}

For now, we take $\mathfrak{g}$ to be any $\Of$-Lie algebra which is free of finite rank as an $\Of$-module. We use a fixed choice of the deformation parameter $n \in \mathbb{N}_{0}$, and recall we write $\pi$ for the uniformiser of $R$.

\begin{mydef}
We fix a descending filtration of $\ug$ as follows: first take $\Gamma_{0}\ug = U(p^{n} \mathfrak{g}_{R})$. Then take $\Gamma_{i}\ug = \pi^{i} \Gamma_{0} \ug$ for each $i \in \mathbb{Z}$.

Define $\widehat{U(\mathfrak{g})_{n,R}}$ to be the completion of $\Gamma_{0}\ug$ with respect to the corresponding positive filtration. Then define $\Ug = K \otimes_{R} \widehat{U(\mathfrak{g})_{n,R}}$, which is a filtered ring with the filtration $\Gamma_{i} \Ug = \pi^{i} \widehat{U(\mathfrak{g})_{n,R}}$.
\end{mydef}

\begin{mydef}
For any finitely generated $\ug$-module $M$, we define the completion of $M$ as $\hat{M} = \Ug \otimes_{\ug} M$. 
\end{mydef}

\noindent As in \cite{me}, we will restrict ourselves to filtered modules in the following category.

\begin{mydef}
The \emph{category} $\mathcal{M}_{n}(\mathfrak{g})$ is the category of all finitely generated filtered $\ug$-modules $(M, \Gamma)$ such that $\Gamma$ is a good filtration, and such that $\Gamma_{0}M$ is free as an $R$-module. The morphisms are filtered module homomorphisms. We will write $M \in \mathcal{M}_{n}(\mathfrak{g})$ if there exists a filtration for $M$ such that $(M,\Gamma) \in \mathcal{M}_{n}(\mathfrak{g})$.

\end{mydef}

\noindent  Let $(M,\Gamma) \in \mathcal{M}_{n}(\mathfrak{g})$. In \cite[Proposition 2.2.12]{me}, it is shown that for any $(M, \Gamma) \in \mathcal{M}_{n}(\mathfrak{g})$, $\hat{M} = \Ug \otimes_{\ug} M$ is isomorphic to the completion of $M$ with respect to $\Gamma$.

Given a free $R$-basis $\mathcal{B}$ for $\Gamma_{0}M$, we can describe $\hat{M}$ as consisting precisely of the infinite sums $\underset{v \in \mathcal{B}}{\sum} C_{v} v$ which are convergent in the sense that for any $i$, $C_{v} \in \pi^{i}R$ for all but finitely many $v$. We can then describe $M \subseteq \hat{M}$ as the subset of all the finite sums.

We will generally want to restrict further to modules in the category $\mathcal{O}_{n}(\mathfrak{g})$, as defined in \cite[Definition 2.2.20]{me}. By \cite[Corollary 2.2.23]{me}, any $M \in \mathcal{O}_{n}(\mathfrak{g})$ has a good filtration $\Gamma$ such that $(M,\Gamma) \in \mathcal{M}_{n}(\mathfrak{g})$.

\noindent By \cite[Lemma 2.2.27]{me}, given $M \in \mathcal{O}_{n}(\mathfrak{g})$, it makes sense to discuss weight vectors in $\hat{M}$. It also makes sense to describe weight components of elements of $v \in \hat{M}$. We will choose our good filtration $\Gamma$ of $M$ such that $\Gamma_{0}M$ has an $R$-basis $\mathcal{B}$ consisting of weight vectors. 

Then we know any element of $\hat{M}$ can be written as a convergent sum of the form $\underset{v \in \mathcal{B}}{\sum} a_{v} v$, which means in particular that every element of $\hat{M}$ is a convergent sum of weight vectors. This means that every element of $\hat{M}$ has a weight decomposition: if $x = \underset{v \in \mathcal{B}}{\sum} a_{v} v \in \hat{M}$, then the $\lambda$-weight component of $x$ consists of the sum of all $a_{v} v$ such that $v$ has weight $\lambda$.

\begin{mylem}
\label{closurelem}
Suppose $M \in \mathcal{O}_{n}(\mathfrak{g})$, and $\Gamma$ is a good filtration of $M$ such that $\Gamma_{0}M$ has an $R$-basis $\mathcal{B}$ consisting of weight vectors. 

Suppose $W$ is a vector subspace of $M$ spanned by weight vectors. Let $\bar{W}$ be the closure of $W$ in $\hat{M}$ (using the topology corresponding to $\Gamma$). Then if $x \in \bar{W}$, all weight components of $x$ lie in $W$.

\end{mylem} 

\begin{proof}
We can describe the topology on $\hat{M}$ as being given by the norm $||x|| = p^{-r}$, where $r$ is maximal such that $x \in \Gamma_{r} \hat{M}$. We can also describe $\Gamma_{r} \hat{M} = \{ \underset{\mathcal{B}}{\sum} a_{v} v \in \hat{M} \mid a_{v} \in \pi^{r} R \text{ for all } v \}$.

For each weight $\lambda$, write $\mathcal{B}_{\lambda}$ for the set of elements of $\mathcal{B}$ with weight $\lambda$. Then if $x = \underset{\mathcal{B}}{\sum}a_{v}v \in \hat{M}$, we can describe the weight components of $x$ as $x_{\lambda} = \underset{\mathcal{B}_{\lambda}}{\sum}a_{v}v$. Now we can see that if $x \in \Gamma_{r} \hat{M}$, then all $x_{\lambda} \in \Gamma_{r} \hat{M}$. In particular, $||x_{\lambda}|| \leq ||x||$ for each $\lambda$.

Suppose $x \in \bar{W}$. Then say $x_{i}$ is a sequence in $W$ that converges to $x$ in the norm $|| \cdot ||$. Let $\lambda$ be a weight, and $x_{i,\lambda}$, $x_{\lambda}$ the $\lambda$-weight components of $x_{i}$, $x$ respectively. Then we have $x_{i,\lambda} \rightarrow x_{\lambda}$ in $|| \cdot ||$, and also $x_{i,\lambda} \in W$ for each $i$ since $W$ is spanned by weight vectors.

Let $W_{\lambda}$ be the $\lambda$-weight component of $W$. Then $W_{\lambda}$ is finite-dimensional, since $M \in \mathcal{O}_{n}(\mathfrak{g})$ implies $M$ has finite-dimensional weight spaces. By \cite[Proposition 2.2.18]{me}, finite-dimensional subspaces of $\hat{M}$ are closed, so $x_{\lambda} \in W$.
\end{proof}

\noindent Recall given $\lambda \in \mathfrak{h}_{K}^{*}$, we write $L(\lambda)$ for the unique irreducible highest-weight $\ug$-module with highest weight $\lambda$. It is a standard fact that $L(\lambda)$ is infinite-dimensional precisely when $\lambda$ is not dominant integral.

\begin{mylem}
\label{irreducibleslem}
Suppose $M \in \mathcal{O}_{n}(\mathfrak{g})$ is infinite-dimensional over $K$. Then $\hat{M}$ has a subquotient of the form $\widehat{L(\lambda)}$, where $\lambda \in \mathfrak{h}_{K}^{*}$ is not dominant integral and $\lambda(p^{n} \mathfrak{h}_{R}) \subseteq R$.
\end{mylem}

\begin{proof}
By \cite[Theorem 1.11]{Humphreys}, every element of $\mathcal{O}(\mathfrak{g})$ has finite composition length, and also by \cite[Theorem 1.3]{Humphreys} the irreducible elements of $\mathcal{O}(\mathfrak{g})$ are precisely the $L(\lambda)$. So $M$ has a composition series with composition factors of the form $L(\lambda)$, which must further have $\lambda(p^{n} \mathfrak{h}_{R}) \subseteq R$ since $M \in \mathcal{O}_{n}(\mathfrak{g})$.

Since $M$ is infinite-dimensional, some composition factor of $M$ must be infinite-dimensional. So $M$ has a subquotient of the form $L(\lambda)$ where $\lambda(p^{n} \mathfrak{h}_{R}) \subseteq R$ and also $\lambda$ is not dominant integral. Then since the completion functor is flat (\cite[Proposition 2.2.12]{me}), we can see that $\widehat{L(\lambda)}$ is a subquotient of $\hat{M}$.
\end{proof}

\subsection{Induced modules}
\label{gvm}

Assume $\mathfrak{g}$ is an $\Of$-lattice of a split semisimple $F$-Lie algebra with root system $\Phi$ and a set of simple roots $\Delta$, and assume $\mathfrak{g}$ has a Chevalley basis $\{h_{\alpha} \mid \alpha \in \Delta \} \cup \{e_{\alpha},f_{\alpha} \mid \alpha \in \Phi^{+} \}$. Then we have corresponding subalgebras $\mathfrak{g}_{I}$, $\mathfrak{p}_{I}, \mathfrak{l}_{I}$ as in Definition \ref{subalgebras}.

Suppose $\lambda: \mathfrak{h}_{K} \to K$ is a weight such that $\lambda(h_{\alpha}) \in \mathbb{N}_{0}$ for all $\alpha \in I$. Then let $V_{\lambda}$ be the unique irreducible highest-weight $U(\mathfrak{g}_{I,K})$-module of highest weight $\lambda |_{\mathfrak{h}_{I,K}}$, which is finite-dimensional over $K$. Moreover, $V_{\lambda}$ can be made into a $\up$-module, as in \cite[\S 9.4]{Humphreys}.

\begin{mydef}
The \emph{generalised Verma module} corresponding to $I$ and $\lambda$ is $M_{I}(\lambda) = U(\mathfrak{g}_{K}) \otimes_{U(\mathfrak{p}_{I,K})} V_{\lambda}$.
\end{mydef}

Now assume further that $\lambda(p^{n} \mathfrak{h}_{R}) \subseteq R$. Then $V_{\lambda}$ lies in the category $\mathcal{O}_{n}(\mathfrak{p})$, being a highest-weight $\up$-module with highest-weight $\lambda$. So we can choose a good filtration $\Gamma_{i}V_{\lambda}$ such that $\Gamma_{0} V_{\lambda}$ is free as an $R$-module. Fix such a filtration $\Gamma$, and let $\mathcal{B}$ be an $R$-basis for $\Gamma_{0}V_{\lambda}$ consisting of weight vectors.
 
Let $x_{1},\dots,x_{r}$ be the elements of $\{f_{\alpha} \mid \alpha \in \Phi^{+} \setminus \Phi^{+}_{I} \}$. Let $x_{r+1},\dots,x_{s}$ be an $\Of$-basis for $\mathfrak{a}$. Then by the PBW theorem (see \cite[2.1.12]{Dixmier}), $M_{I}(\lambda)$ has basis $\{x^{t}v\mid t \in \mathbb{N}_{0}^{s}, v \in \mathcal{B} \}$. Here, $x^{t}$ denotes $x_{1}^{t_{1}} \dots x_{s}^{t_{s}}$. 

Then we can define a filtration $\Gamma_{i}M_{I}(\lambda) = (\Gamma_{i}\ug) \cdot \Gamma_{0}V_{\lambda}$, which gives $(M_{I}(\lambda),\Gamma) \in \mathcal{M}_{n}(\mathfrak{g})$. 

We can then describe $\widehat{M_{I}(\lambda)}$ as consisting of all convergent sums of the form $\underset{v \in \mathcal{B}}{\sum}  \underset{t \in \mathbb{N}_{0}^{r}}{\sum} C_{v,t} (p^{n}x)^{t}v$, where $(p^{n}x)^{t}$ denotes $(p^{n}x_{1})^{t_{1}} \dots (p^{n} x_{r})^{t_{r}}$, and $C_{v,t} \in K$ with $C_{v,t} \rightarrow 0 $ as $|t| \rightarrow \infty$. 

In particular, $\widehat{M_{I}(\lambda)}$ is isomorphic as a $\Un$-module to $\Un \otimes_{K} V_{\lambda}$.

Finally, we define the concept of a scalar generalised Verma module, which is a special case of particular interest to us.

\begin{mydef}
We say $M_{I}(\lambda)$ is a \emph{scalar} generalised Verma module if $\lambda(h_{\alpha}) = 0$ for all $\alpha \in I$. In this case, $V_{\lambda}$ is one-dimensional over $K$.
\end{mydef}

\section{Proving Theorem \ref{maintheorem}}

\subsection{Outline}
\label{outline}

From now on, we will always assume that $\mathfrak{g}$ is an $\Of$-lattice of a split simple $F$-Lie algebra with root system $\Phi$. Take a fixed set of simple roots $\Delta$ of $\Phi$, with corresponding positive system $\Phi^{+}$, and we will assume $\mathfrak{g}$ has a Chevalley basis $\{h_{\alpha} \mid \alpha \in \Delta \} \cup \{e_{\alpha},f_{\alpha} \mid \alpha \in \Phi^{+}\}$. Then $\mathfrak{g}$ has a corresponding triangular decomposition $\mathfrak{g} = \mathfrak{n}^{+} \oplus \mathfrak{h} \oplus \mathfrak{n}^{-}$. 

We will fix a choice of the deformation parameter $n \in \mathbb{N}_{0}$. Then we will take $G$ to be the $F$-uniform group with $L_{G} = p^{n+1} \mathfrak{g}$.

Suppose $\lambda \in \mathfrak{h}_{K}^{*}$ is a weight with $\lambda(p^{n}\mathfrak{h}_{R}) \subseteq R$, and assume $\lambda$ is not a dominant integral weight. Recall that we define a subset $\Delta_{\lambda} = \{ \alpha \in \Delta \mid \lambda(h_{\alpha}) \in \mathbb{N}_{0}$. We then take $\mathfrak{p} = \mathfrak{p}_{\lambda} = \mathfrak{p}_{\Delta_{\lambda}}$ and $\mathfrak{n} = \mathfrak{n}_{\lambda} =  \mathfrak{n}_{\Delta_{\lambda}}$ in the sense of Definition \ref{subalgebras}, and also write $P,N$ for the $F$-uniform groups with $L_{P}=p^{n+1} \mathfrak{p}$, $L_{N} = p^{n+1} \mathfrak{n}$.

Then using the construction of \S \ref{gvm}, we can define a generalised Verma module $M_{\Delta_{\lambda}}(\lambda)$. This module has the form $\ug \otimes_{U(\mathfrak{p}_{K})} V_{\lambda}$, where $V_{\lambda}$ is the irreducible highest-weight $U(\mathfrak{p}_{K})$-module with highest-weight $\lambda$, and $V_{\lambda}$ is finite-dimensional.

Then we know that $M_{\Delta_{\lambda}}(\lambda)$ is isomorphic as a $\un$-module to $\un \otimes_{K} V_{\lambda}$, while $\widehat{M}_{\Delta_{\lambda}}(\lambda)$ is isomorphic as a $\Un$-module to $\Un \otimes_{K} V_{\lambda}$.

We also know that we can choose a filtration $\Gamma$ on $M_{\Delta_{\lambda}}(\lambda)$ with $(M_{\Delta_{\lambda}}(\lambda) , \Gamma) \in \mathcal{M}_{n}(\mathfrak{g})$ (as in \S \ref{gvm}), by first choosing the filtration $\Gamma$ on $V_{\lambda}$ and then extending to $M_{\Delta_{\lambda}}(\lambda)$. Then $\Gamma$ makes $\widehat{M_{\Delta_{\lambda}}(\lambda)}$ into a filtered $\Ug$-module, which also gives $\widehat{M_{\Delta_{\lambda}}(\lambda)}$ a corresponding topology. We will assume throughout that such a topology on $\widehat{M_{\Delta_{\lambda}}(\lambda)}$ is fixed.

\begin{mylem}
\label{densitylem}
Suppose $M$ is a highest-weight $\ug$-module with a highest-weight $\lambda$ such that $\lambda(p^{n} \mathfrak{h}_{R}) \subseteq R$, and suppose $\Gamma_{i}M$ is a good filtration on $M$ such that $(M, \Gamma) \in \mathcal{M}_{n}(\mathfrak{g})$. Let $v$ be the highest-weight vector generating $M$.

Then $KGv$ is a dense subset of $\hat{M}$ in the topology given by $\Gamma$.
\end{mylem}

\begin{proof}

Using our description $\hat{M} = \Ug \otimes_{\ug} M$ we can see that $\hat{M}$ is generated by $v$ over $\Ug$. Also by \cite[Lemma 2.5.6]{me}, $KG$ is dense in $\Ug$. Together, this implies that $KGv$ is dense in $\Ug v = \hat{M}$.
\end{proof}

Now note that by \cite[Lemma 2.2.17]{me}, we can consider the finite-dimensional $\up$-module $V_{\lambda}$ as a $\Up$-module and hence a $KP$-module. We then want to consider the $KG$-module $KG \otimes_{KP} V_{\lambda}$.

\begin{mylem}
$KG \otimes_{KP} V_{\lambda}$ embeds as a dense $KG$-submodule of $\widehat{M_{\Delta_{\lambda}}(\lambda)}$. Moreover $KG \otimes_{KP} V_{\lambda}$ can be identified with $KN V_{\lambda}$.
\end{mylem}

\begin{proof}
From our definition of completions in terms of a tensor product with $\Ug$, we can describe $\widehat{M_{\Delta_{\lambda}}(\lambda)}$ as $\Ug \otimes_{\up} V_{\lambda}$. Then since $V_{\lambda}$ is a $\Up$-module, we can say $\widehat{M_{\Delta_{\lambda}}(\lambda)} = \Ug \otimes_{\Up} V_{\lambda}$.

Now we know $KG \subseteq \Ug$ and $KP = KG \cap \Up$. Using the universal property of tensor products, we can define a map $KG \otimes_{KP} V_{\lambda} \rightarrow \Ug \otimes_{\Up} V_{\lambda}$ by $x \otimes v \mapsto x \otimes v$ for $x \in KG$, $v \in V_{\lambda}$, and this map is injective.

Moreover, if $w_\lambda$ is the highest-weight vector generating $M_{\Delta_{\lambda}}(\lambda)$, then the image of $KG \otimes_{KP} V_{\lambda}$ in $\widehat{M}_{\Delta_{\lambda}}(\lambda)$ contains $KG w_{\lambda}$, and so is dense by Lemma \ref{densitylem}.
\end{proof}

\noindent The following lemma implies that there is an embedding of $\up$-modules of $V_{\lambda}$ in $L(\lambda)$.

\begin{mylem}
\label{Vembedding}
The projection map $M_{\Delta_{\lambda}}(\lambda) \rightarrow L(\lambda)$ is injective on $V_{\lambda}$. In particular, $V_{\lambda}$ embeds in $L(\lambda)$ as the $\up$-submodule generated by the highest-weight vector $v_{\lambda}$.
\end{mylem}

\begin{proof}
Write $w_{\lambda}$ for the highest-weight vector generating $M_{\Delta_{\lambda}}(\lambda)$. Suppose $0 \neq v \in V_{\lambda} \subseteq M_{\Delta_{\lambda}}(\lambda)$. Then since $V_{\lambda}$ is an irreducible $\up$-module, $w_{\lambda} \in \up v$ Thus the image of $v$ in $L(\lambda)$ is non-zero.

We can describe $V_{\lambda}$ as the $\up$-submodule $\up w_{\lambda}$ of $M_{\Delta_{\lambda}}(\lambda)$, which maps to $\up v_{\lambda}$ in $L(\lambda)$.
\end{proof}

\noindent Now consider the composition map:

\begin{equation*}
KG \otimes_{KP} V_{\lambda} \hookrightarrow \widehat{M_{\Delta_{\lambda}}(\lambda)} \twoheadrightarrow \widehat{L(\lambda)}.
\end{equation*}

\begin{mythm}
\label{condition1application}
If the composition map $KG \otimes_{KP} V_{\lambda} \rightarrow \widehat{L(\lambda)}$ is injective and $p$ does not divide the determinant of the Cartan matrix of any root subsystem of $\Phi$, then $\widehat{L(\lambda)}$ is a faithful $KG$-module.
\end{mythm}

\begin{proof}
By \cite[Theorem 1]{me}, $\widehat{M_{\Delta_{\lambda}}(\lambda)}$ is faithful as a $KG$-module. Since $KG \otimes_{KP} V_{\lambda}$ is dense as a $KG$-submodule of $\widehat{M_{\Delta_{\lambda}}(\lambda)}$, this implies $KG \otimes_{KP} V_{\lambda}$ is faithful over $KG$.

So then the existence of an embedding $KG \otimes_{KP} V_{\lambda} \hookrightarrow \widehat{L(\lambda)}$ of $KG$-modules implies that $\widehat{L(\lambda)}$ is faithful over $KG$.
\end{proof}

\noindent We know that $KG \otimes_{KP} V_{\lambda}$ is  isomorphic as a $KN$-module to $KN \otimes_{K} V_{\lambda}$. We also know that the map $\widehat{M_{\Delta_{\lambda}}(\lambda)} \rightarrow \widehat{L(\lambda)}$ gives an embedding of $V_{\lambda}$ in $\widehat{L(\lambda)}$. So the statement that the composition $KG \otimes_{KP} V_{\lambda} \rightarrow \widehat{L(\lambda)}$ is injective is equivalent to the following.

\bigskip

\noindent \textbf{Condition (1)}: the multiplication map $KN_{\lambda} \otimes_{K} V_{\lambda} \rightarrow \widehat{L(\lambda)}$ is injective.

\bigskip

\noindent So given that $p$ satisfies the condition that $p$ does not divide the determinant of the Cartan matrix of any root subsystem of $\Phi$, we can see that condition (1) for $\lambda$ implies $\widehat{M_{\Delta_{\lambda}}(\lambda)}$ is faithful over $KG$. We also state another condition which is stronger than condition (1).

\bigskip

\noindent \textbf{Condition (2)}: the multiplication map $KN_{\lambda} \otimes_{K} L(\lambda) \rightarrow \widehat{L(\lambda)}$ is injective.

\bigskip

\noindent So now we want to prove faithfulness of $\widehat{L(\lambda)}$ by proving either condition (1) or condition (2). The approach ends up being very much on a case-by-case basis, but the rough approach is as follows.

\begin{itemize}
\item In \S \ref{simplerefsection}, we will see that if $\alpha \in \Delta$ is a simple root such that $\lambda(h_{\alpha}) \notin \mathbb{N}_{0}$, then $\Ann_{\Ug} \widehat{L(\lambda)} \subseteq \Ann_{\Ug} \widehat{L(s_{\alpha} \cdot \lambda)}$. This will be a helpful tool for reducing the number of cases we have to prove.

\item In \S \ref{transfunctorssection}, we will use translation functors (which involve taking the tensor product with finite-dimensional $\ug$-modules) to show that under certain circumstances, we can deduce condition (1) or (2) for $\lambda$ by proving them for $\lambda - \nu$ or $\lambda + \nu$, with $\nu$ dominant integral. This allows us to further reduce the number of cases to prove.

\item In \S \ref{inductionsection} and \S \ref{inductionsectiontypeD}, we will establish an inductive approach for proving condition (1). The outline is that given a simple root $\alpha \in \Delta$, we consider the subalgebra $ \mathfrak{g}_{\Delta \setminus \alpha}$ of $\mathfrak{g}$. We then take the restriction $\lambda'$ of $\lambda$ to $\mathfrak{h}_{\Delta \setminus \alpha}$. Under appropriate conditions, we can show that condition (1) for $\lambda'$ will imply condition (1) for $\lambda$. Then \S \ref{hwvectors} is dedicated to examining when the conditions for this inductive approach are met.

\item Section \ref{abeliansection} deals with a special case, where $\Delta_{\lambda} = \Delta \setminus \alpha$ for some $\alpha \in \Delta$. This is a particularly difficult case: speaking informally, $L(\lambda)$ will be "smaller" the closer $\lambda$ is to being dominant integral, which makes the faithfulness of $\widehat{L(\lambda)}$ over $KG$ harder to prove. Therefore some different tools are required to deal with this case.
\end{itemize}

\subsection{Use of simple reflections}
\label{simplerefsection}

For the moment, assume $\Phi$ is any indecomposable root system with a fixed set of simple roots $\Delta$. We will write $W$ for the Weyl group of $\Phi$. Recall that given a weight $\lambda \in \mathfrak{h}_{K}^{*}$, we denote $\Phi_{\lambda} = \{ \alpha \in \Phi \mid \lambda( h_{\alpha}) \in \mathbb{Z} \}$.

We can directly calculate that given $w \in W$, $\Phi_{w \cdot \lambda} = w^{-1}(\Phi_{\lambda})$.

\begin{myprop}
\label{simplerefsprop}
Suppose $\lambda$ is a weight such that $\lambda(p^{n}\mathfrak{h}_{R}) \subseteq R$, and $\alpha \in \Delta$ such that $\lambda(h_{\alpha}) \notin \mathbb{N}_{0}$. Then $\Ann_{\Ug} \widehat{L(\lambda)} \subseteq \Ann_{\Ug} \widehat{L(s_{\alpha} \cdot \lambda)}$.
\end{myprop}

\begin{proof}
\cite[Proposition 7.0.1]{gvmcase}
\end{proof}

\noindent This result means that in order to prove faithfulness for $\widehat{L(\lambda)}$ over $KG$, it is sufficient to prove for any $\widehat{L(s_{\alpha} \cdot \lambda)}$ with $\lambda(h_{\alpha}) \notin \mathbb{N}_{0}$. We will use this to restrict our attention to a certain subset of weights for which we can use an inductive approach.

\begin{mydef}
\label{connected}
 We say a subset $I \subseteq \Delta$ is \emph{connected} if the corresponding subgraph of the Dynkin diagram of $\Phi$ is a connected graph.

Moreover, given $\alpha \in I$, the \emph{connected component} of $\alpha$ in $I$ is the largest connected subset of $I$ containing $\alpha$.
\end{mydef}

\noindent We will want to restrict ourselves to proving faithfulness of $\widehat{L(\lambda)}$ only for weights $\lambda$ such that $\Delta \setminus \Phi_{\lambda} = \{ \alpha \in \Delta \mid \lambda(h_{\alpha}) \notin \mathbb{Z} \}$ is a connected subset of $\Delta$.

\begin{mylem}
\label{orbitlem}
Suppose $\Phi$ is simply laced , and $\alpha,\beta \in \Delta$ are adjacent in the Dynkin diagram of $\Phi$. Then $s_{\alpha}s_{\beta}(\alpha) = \beta$. In particular, all simple roots lie in the same orbit of $W$.
\end{mylem}

\begin{proof}
Since $\alpha, \beta$ are adjacent in the Dynkin diagram and $\Phi$ is simply laced, we can see that $\langle \alpha, \beta^{\lor} \rangle = \langle \beta, \alpha^{\lor} \rangle = -1$.

Therefore $s_{\beta}(\alpha) = \alpha+\beta = s_{\alpha}(\beta)$. So $s_{\alpha} s_{\beta}(\alpha) = s_{\alpha}(\alpha+\beta) = \beta$.

The fact that all simple roots lie in the same orbit of $W$ now follows from the fact that $\Phi$ is indecomposable, which means that the Dynkin diagram is a connected graph.
\end{proof}

\begin{mylem}
\label{connectedcompslem}
Suppose $\Phi' \subseteq \Phi$ is a root subsystem and $\Phi$ is simply laced, and suppose $C$ is a connected component of $\Delta \setminus \Phi'$. If $\alpha \in \Delta \setminus C$, then $C \subseteq \Delta \setminus s_{\alpha}(\Phi')$.
\end{mylem}

\begin{proof}
Suppose $\beta \in C \subseteq \Delta \setminus \Phi'$. We claim $\beta \in \Delta \setminus s_{\alpha}(\Phi')$.

\noindent \emph{Case 1}: $(\beta, \alpha) = 0$. Then $\beta = s_{\alpha}(\beta) \in s_{\alpha}(\Phi \setminus \Phi') = \Phi \setminus s_{\alpha}(\Phi')$.

\noindent \emph{Case 2}: $(\beta, \alpha) \neq 0$. Then we can see that $\alpha \in \Phi'$: if not, $C \cup \{\alpha \}$ would be a connected subset of $\Delta \setminus \Phi'$ strictly containing $C$, contradicting the fact that $C$ is a connected component.

Now since $\Phi$ is simply laced, we must have $\langle \alpha, \beta^{\lor} \rangle = -1 = \langle \beta, \alpha^{\lor} \rangle$. So we can calculate $\alpha + \beta = s_{\alpha}(\beta) \in \Phi$. We must have $\alpha + \beta \notin \Phi'$, as otherwise we would have $s_{\alpha}(\alpha + \beta) = \beta \in \Phi'$, a contradiction. 

Then $\beta = s_{\alpha}(\alpha + \beta) \in s_{\alpha}(\Phi \setminus \Phi') = \Phi \setminus s_{\alpha}(\Phi')$.
\end{proof}

\begin{myprop}
\label{subsystemprop}
Suppose $\Phi$ is a simply laced root system, and $\Psi$ is a proper root subsystem of $\Phi$. Suppose $\alpha_{1} \in \Delta$ is fixed. Then there exists $w \in W$ such that $\Delta \setminus w(\Psi)$ is a connected subset of $\Delta$ containing $\alpha_{1}$.
\end{myprop}

\begin{proof}
First, the fact that $\Psi$ is a proper root subsystem implies that $\Psi$ cannot contain all $\alpha \in \Delta$. If we choose $\alpha \in \Delta \setminus \Psi$, then by Lemma \ref{orbitlem} we can choose $w \in W$ such that $w(\alpha) = \alpha_{1}$. Then $\alpha_{1} \in \Delta \setminus w(\Psi)$. 

Choose $w_{0} \in W$ with $\alpha_{1} \in \Delta \setminus w_{0}(\Psi)$ to maximise the size of the connected component   of $\alpha_{1}$ in $\Delta \setminus w_{0}(\Psi)$. Write $C$ for the connected component of $\alpha_{1}$ in $\Delta \setminus w_{0}(\Psi)$.

We want to show that $\Delta \setminus w_{0}(\Psi) = C$.  Suppose for contradiction that $\alpha \in \Delta \setminus w_{0}(\Psi)$ with $\alpha \notin C$. Since Lemma \ref{orbitlem} implies all elements of $\Delta$ are in the same orbit under $W$, we can choose $w \in W$ such that $w(\alpha) \in \Delta$ is adjacent to $C$ in the Dynkin diagram of $\Phi$. Any such $w$ can be written as a product of simple reflections.

So let $\beta_{1},\dots,\beta_{r} \in \Delta$ be a sequence of minimal length such that \\ $(\gamma,(s_{\beta_{r}}  \dots s_{\beta_{1}})(\alpha)) \neq 0$ for some $\gamma \in C$. We must have $r \geq 1$ by choice of $\alpha$.

Fix $1 \leq i \leq r$. Suppose $\beta_{i} \in C$. By minimality of $r$, $(\gamma, s_{\beta_{i-1}} \dots s_{\beta_{1}}(\alpha)) = 0$ for all $\gamma \in C$. So $s_{\beta_{i}}( s_{\beta_{i-1}} \dots s_{\beta_{1}} (\alpha) ) = s_{\beta_{i-1}} \dots s_{\beta_{1}} (\alpha)$, but this contradicts minimality of $r$ because we could simply remove $\beta_{i}$ from the sequence. So we must have all $\beta_{i} \notin C$. 

\bigskip

\noindent \emph{Claim:} for each $0 \leq i \leq r$, the connected component of $\Delta \setminus s_{\beta_{i}} \dots s_{\beta_{1}}w_{0}(\Psi)$ containing $\alpha_{1}$ is $C$.
\bigskip

\noindent \emph{Proof of claim:} by induction on $i$. The $i = 0$ case is the base case. Now for $i > 0$, assume $C$ is the connected component of  $\Delta \setminus s_{\beta_{i-1}} \dots s_{\beta_{1}}w_{0}(\Psi)$ containing $\alpha_{1}$. Since $\beta_{i} \notin C$, Lemma \ref{connectedcompslem} implies that $C \subseteq \Delta \setminus s_{\beta_{i}} \dots s_{\beta_{1}}w_{0}(\Psi)$. Then the claim follows since $w_{0}$ was chosen to maximise the size of $C$.

\bigskip

Write $w = s_{\beta_{r}} \dots s_{\beta_{1}}$. Now the claim above shows that $C$ is the connected component of $\Delta \setminus w w_{0}(\Psi)$ containing $\alpha_{1}$. But $\alpha \in \Delta \setminus w_{0} (\Psi)$ implies $w(\alpha) \notin ww_{0}(\Psi)$, and $w$ was chosen such that $w(\alpha) \in \Delta \setminus C$  is adjacent to $C$ in the Dynkin diagram. So we have a contradiction.

We now see that $\Delta \setminus w_{0}(\Psi)$ is a connected subset of $\Delta$ containing $\alpha_{1}$.

\end{proof}

\begin{mycor}
\label{simplerefsapplication}
Suppose $\Phi$ is simply laced, and $\lambda \in \mathfrak{h}_{K}^{*}$ is a non-integral weight. Suppose $\alpha_{1} \in \Delta_{\lambda}$ is fixed.

 Then there exist $\beta_{1},\dots,\beta_{r} \in \Delta$ such that $(s_{\beta_{i-1}} \dots s_{\beta_{1}} \cdot \lambda)(h_{\beta_{i}}) \notin \mathbb{N}_{0}$ for all $i$, and $\lambda' = s_{\beta_{r}} \dots s_{\beta_{1}} \cdot \lambda$ satisfies the condition that $\Delta \setminus \Phi_{\lambda'}$ is a connected subset of $\Delta$ containing $\alpha_{1}$.

\end{mycor}

\begin{proof}
By Proposition \ref{subsystemprop}, we can choose $w \in W$ such that $\Delta \setminus w( \Phi_{\lambda})$ is a connected subset of $\Delta$ containing $\alpha_{1}$. We know $w(\Phi_{\lambda}) = \Phi_{w^{-1} \cdot \lambda}$, and we can write $w^{-1}$ as a product of simple reflections.

So choose $\beta_{1},\dots,\beta_{r} \in \Delta$ to be a sequence of minimal length such that $\Delta \setminus \Phi_{s_{\beta_{r}} \dots s_{\beta_{1}} \cdot \lambda}$ is a connected subset of $\Delta$ containing $\alpha_{1}$.

Fix $1 \leq i \leq r$. If $(s_{\beta_{i-1}} \dots s_{\beta_{1}} \cdot \lambda)(h_{\beta_{i}}) \in \mathbb{N}_{0}$, then $\beta_{i} \in \Phi_{s_{\beta_{i-1}} \dots s_{\beta_{1}} \cdot \lambda}$. Thus $\Phi_{s_{\beta_{i}} \dots s_{\beta_{1}} \cdot \lambda} = s_{\beta_{i}}(\Phi_{s_{\beta_{i-1}} \dots s_{\beta_{1}} \cdot \lambda}) = \Phi_{s_{\beta_{i-1}} \dots s_{\beta_{1}} \cdot \lambda}$. We then see that $\Phi_{s_{\beta_{r}} \dots s_{\beta_{1}} \cdot \lambda} = \Phi_{s_{\beta_{r}} \dots s_{\beta_{i+1}} s_{\beta_{i-1}} \dots s_{\beta_{1}} \cdot \lambda}$, which contradicts minimality of $r$.

So $(s_{\beta_{i-1}} \dots s_{\beta_{1}} \cdot \lambda)(h_{\beta_{i}}) \notin \mathbb{N}_{0}$ for all $i$, and the result follows.
\end{proof}

\subsection{Translation functors}
\label{transfunctorssection}

To motivate this section, suppose $\lambda, \mu \in \mathfrak{h}_{K}^{*}$ with $\lambda(p^{n}\mathfrak{h}_{R}) \subseteq R$, $\mu(p^{n}\mathfrak{h}_{R}) \subseteq R$ such that $\lambda - \mu$ is a dominant integral weight (that is, $(\lambda- \mu)(h_{\alpha}) \in \mathbb{N}_{0}$ for all $\alpha \in \Delta$). Then under certain circumstances, we wish to show that proving condition (1) (or (2)) for $\mu$ will imply that it also holds for $\lambda$, or alternatively that proving one of these conditions for $\lambda$ will also imply it holds for $\mu$.

 One particular use of this is that it will often be easier to prove condition (1) in the case where $M_{\Delta_{\lambda}}(\lambda)$ is a scalar generalised Verma module, that is, $\lambda(h_{\alpha}) = 0$ whenever $\lambda(h_{\alpha}) \in \mathbb{N}_{0}$. In this setting, the representation $V_{\lambda}$ is one-dimensional, and condition (1) is equivalent to the statement that $\Ann_{KN}v_{\lambda} = 0$.
 
The tool that we use is that of translation functors, which involve taking tensor products by finite-dimensional modules. Our main source for this section is \cite[\S 7]{Humphreys}.

We define subcategories $\mathcal{O}_{\lambda}$ of $\mathcal{O}$ for each weight $\lambda$ by $\mathcal{O}_{\lambda} = \mathcal{O}_{\chi_{\lambda}}$ in the sense of \cite[\S 1.12]{Humphreys}, which in particular contain all $\lambda$-highest-weight modules. Moreover, the distinct $\mathcal{O}_{\lambda}$ are the blocks of $\mathcal{O}$.

\begin{mydef}
Suppose $\lambda, \mu \in \mathfrak{h}_{K}^{*}$ such that $\lambda - \mu$ is a dominant integral weight.

The translation functor $T^{\lambda}_{\mu} : \mathcal{O}_{\mu} \rightarrow \mathcal{O}_{\lambda}$ is defined by taking $T_{\mu}^{\lambda}M$ to be the projection of $M \otimes_{K} L(\lambda - \mu)$ onto $\mathcal{O}_{\lambda}$.
\end{mydef}

\begin{mydef}
Given $\lambda \in \mathfrak{h}_{K}^{*}$, define:

\begin{itemize}
\item $\Phi_{\lambda}^{+} = \{ \alpha \in \Phi^{+} \mid \langle \lambda + \rho, \alpha^{\lor} \rangle \in \mathbb{N} \}$,

\item $\Phi_{\lambda}^{0} = \{ \alpha \in \Phi^{+} \mid \langle \lambda + \rho, \alpha^{\lor} \rangle = 0 \}$,

\item $\Phi_{\lambda}^{-} = \{ \alpha \in \Phi^{+} \mid \langle \lambda + \rho, \alpha^{\lor} \rangle \in  \mathbb{Z}_{<0} \}$.

\end{itemize}

\noindent Then a \emph{facet} is a set of weights $\mathcal{F} = \{ \lambda \in \mathfrak{h}_{K}^{*} \mid \Phi_{\lambda}^{+} = \Phi^{+}_{\mathcal{F}}, \Phi_{\lambda}^{0} = \Phi^{0}_{\mathcal{F}}, \Phi_{\lambda}^{-} = \Phi^{-}_{\mathcal{F}} \}$ for some choice of subsets $\Phi^{+}_{\mathcal{F}},\Phi^{0}_{\mathcal{F}},\Phi^{-}_{\mathcal{F}} \subseteq \Phi$. So any weight $\lambda$ is in a unique facet.

The \emph{upper closure} $\hat{\mathcal{F}}$ of a facet $\mathcal{F}$ is the set of weights $\hat{\mathcal{F}} = \{ \lambda \in \mathfrak{h}_{K}^{*} \mid \Phi_{\lambda}^{+} = \Phi_{\mathcal{F}}^{+},  \Phi_{\lambda}^{-} \subseteq \Phi_{\mathcal{F}}^{-} \}$.
\end{mydef}

\noindent The following result can be deduced by looking at the effect of translation functors on simple modules.

\begin{mylem}
\label{translationfunctionlem}
Suppose $\lambda, \mu \in \mathfrak{h}_{K}^{*}$ such that $\lambda - \mu$ is a dominant integral weight, and $\lambda$ is in the upper closure of the facet of $\mu$. 

Then $L(\lambda)$ can be considered as the $\ug$-submodule of $\widehat{L(\mu)}\otimes_{K} L(\lambda - \mu)$ generated by $v_{\mu} \otimes v_{\lambda - \mu}$.
\end{mylem}

\begin{proof}
\cite[Corollary 4.2.24]{me}
\end{proof}

\begin{myprop}
\label{translationfunctorsprop}
Suppose $\lambda, \mu \in \mathfrak{h}_{K}^{*}$ with $\lambda(p^{n}\mathfrak{h}_{R}) \subseteq R$, $\mu(p^{n}\mathfrak{h}_{R}) \subseteq R$ such that $\lambda - \mu$ is a dominant integral weight. Suppose $\lambda, \mu$ are in the same facet. Then: 
\begin{itemize}
\item If condition (1) holds for $\mu$, then condition (1) holds for $\lambda$.

\item If condition (2) holds for $\mu$, then condition (2) holds for $\lambda$.
\end{itemize}
\end{myprop}

\begin{proof}
\cite[Proposition 4.2.21]{me}
\end{proof}

One of the most difficult cases in which to prove Theorem \ref{maintheoremvariant} is that of regular integral weights. We will later see, using Proposition \ref{simplerefsprop}, that it is sufficient to consider only the regular integral weights of the form $s_{\alpha_{1}} \cdot \mu$ where $\mu$ is dominant integral and $\alpha_{1} \in \Delta$. So we want to consider integral weights $\lambda$ where $\Phi_{\lambda}^{+} = \Phi^{+} \setminus \alpha_{1}$ for some $\alpha_{1} \in \Delta$.

We will want to apply a further restriction on our simple root $\alpha_{1}$ that any $\alpha \in \Phi^{+}$ has $\alpha_{1}$-coefficient at most $1$. This implies that the corresponding nilpotent subalgebra $\mathfrak{n}_{ \Delta \setminus \alpha_{1} } = \Of \{ f_{\alpha} \mid \alpha_{1} \in \alpha \}$  is abelian (recall the notation $\alpha_{1} \in \alpha$ means $\alpha$ has a non-zero $\alpha_{1}$-coefficient). If $\alpha, \beta \in \mathfrak{n}_{ \Delta \setminus \alpha_{1}}$, then both $\alpha, \beta$ have $\alpha_{1}$-coefficient $1$. This means $\alpha + \beta \notin \Phi^{+}$ since it has $\alpha_{1}$-coefficient $2$, and hence $[f_{\alpha},f_{\beta}] = 0$. 

\begin{myprop}
\label{abeliantranslationfunctorprop}
Suppose $\alpha_{1} \in \Delta$ such that all $\alpha \in \Phi^{+}$ have $\alpha_{1}$-coefficient $0$ or $1$. 

Suppose $\lambda \in \mathfrak{h}_{K}^{*}$ is an integral weight such that $\Phi_{\lambda}^{+} = \Phi^{+} \setminus \alpha_{1}$. Suppose $\omega_{1} \in \mathfrak{h}_{K}^{*}$ is the fundamental weight corresponding to $\alpha_{1}$, so $\omega_{1}(h_{\alpha_{1}}) = 1$, $\omega_{1}(h_{\alpha}) = 0$ for $\alpha \in \Delta \setminus \alpha_{1}$.

Then if condition (2) holds for $- \omega_{1}$, condition (2) also holds for $\lambda$.
\end{myprop}

\noindent This is very helpful for proving that condition (2) holds in the setting $\Phi^{+}_{\lambda} = \Phi^{+} \setminus \alpha_{1}$, because it means we only need to prove it with the further restriction that $M_{\Delta_{\lambda}}(\lambda)$ is a scalar generalised Verma module. We will see later that this case is easier to work with.

\bigskip

\noindent
\textbf{For the rest of this section, fix a simple root $\alpha_{1} \in \Delta$, and make the assumption that all $\alpha \in \Phi^{+}$ have $\alpha_{1}$-coefficient $0$ or $1$. }

Also write $\mathfrak{n}$ for the corresponding nilpotent subalgebra (which is spanned by all $f_{\alpha}$ with $\alpha_{1} \in \alpha$). So $\mathfrak{n}$ is abelian. Also write $N$ for the corresponding subgroup of $G$.

Furthermore, write $\omega_{1}$ for the fundamental weight corresponding to $\alpha_{1}$ (that is, $\omega_{1} \in \mathfrak{h}_{K}^{*}$ with $\omega_{1}(h_{\alpha_{1}}) = 1$ and $\omega_{1}(h_{\alpha}) = 0$ for $\alpha \in \Delta \setminus \alpha_{1}$.

Our basic approach will be first to show that condition (2) for $- \omega_{1}$ implies condition (2) for weights of the form $s_{\alpha_{1}} \cdot a \omega_{1}$ for $a \in \mathbb{N}_{0}$, and then to show that this implies condition (2) in the general setting $\Phi^{+}_{\lambda} = \Phi^{+} \setminus \alpha_{1}$.

\begin{mylem}
\label{abeliantranslationlem}
Suppose $\Phi$ is any indecomposable root system, and $\lambda,\mu \in \mathfrak{h}_{K}^{*}$ such that $\lambda - \mu$ is a dominant integral weight, and $\lambda$ is in the upper closure of the facet of $\Phi$. Suppose $\mathfrak{n} \subseteq \mathfrak{n}^{-}$ is any abelian subalgebra. Then there exists $l \in \mathbb{N}$ such that if $x_{1},\dots,x_{l} \in \un$ with $x_{i} v_{\mu} =0$ for all $i$, then $x_{1} \dots x_{l} v_{\lambda} = 0$.

Moreover, if $\lambda(p^{n}\mathfrak{h}_{R}) \subseteq R$, $\mu(p^{n}\mathfrak{h}_{R}) \subseteq R$ and $x \in \Un$ such that $x v_{\mu} = 0$ in $\widehat{L(\mu)}$, then $x^{l} v_{\lambda} = 0$ in $\widehat{L(\lambda)}$.
\end{mylem}

\begin{proof}
By Lemma \ref{translationfunctionlem}, $L(\lambda)$ can be considered as the submodule of $L(\mu) \otimes_{K} L(\lambda - \mu)$ generated by $v_{\mu} \otimes v_{\lambda - \mu}$. We identify $v_{\lambda}$ with $v_{\mu} \otimes v_{\lambda - \mu}$.

Now define a descending $K$-linear filtration of $L(\lambda - \mu)$ by $L(\lambda - \mu)_{i} =  \{y_{1} \dots y_{j} v_{\lambda - \mu} \mid y_{1},\dots,y_{j} \in \mathfrak{n}^{-}, j \geq i \}$.

Since $\lambda - \mu$ is dominant integral, $L(\lambda - \mu)$ is finite-dimensional. So choose $l$ large enough that $L(\lambda - \mu)_{l} = 0$.

Now we will show that if $x \in \un$ with $x v_{\mu} = 0$, then $x \cdot ( \un v_{\mu} \otimes L(\lambda - \mu)_{i}) \subseteq \un v_{\mu} \otimes L(\lambda - \mu)_{i+1}$ for each $i$. 

Note that if $f \in \mathfrak{n}$, $a \in \un$ and $v \in L(\lambda - \mu)_{i}$, then $f \cdot (a v_{\mu} \otimes v) = (f \cdot a v_{\lambda}) \otimes v + a v_{\mu} \otimes f v \in a f v_{\mu} \otimes v + \un v_{\lambda} \otimes L(\lambda - \mu)_{i+1}$, using the definition of the action of $\ug$ on $L(\mu) \otimes L(\lambda - \mu)$ together with the fact that $\mathfrak{n}$ is abelian.

Applying this, we can see that if $a \in \un, v \in L(\lambda - \mu)$ and $x \in \un$ with $x v_{\mu} = 0$, then $x \cdot (a v_{\mu} \otimes v) \in a x v_{\mu} \otimes v + \un v_{\mu} \otimes L(\lambda - \mu)_{i+1} = \un v_{mu} \otimes L(\lambda - \mu)_{i+1}$.

Thus if $x_{1},\dots, x_{l} \in \un$ with $x_{i} v_{\mu} = 0$, then $x_{1} \dots x_{l} v_{\mu} \otimes v_{\lambda - \mu} \in \un v_{\mu} \otimes L(\lambda - \mu)_{l} = 0$.

Now suppose $x \in \Un$ such that $x v_{\mu} = 0$. Say $x = \sum_{\nu} x_{\nu}$ as a sum of weight components. So $x_{\nu} v_{\mu} = 0$ for each $\nu$. Then for any weights $\nu_{1},\dots,\nu_{l}$, we have $x_{\nu_{1}} \dots x_{\nu_{l}} v_{\lambda} = 0$. Since $x^{l}$ is a convergent sum of terms of the form $x_{\nu_{1}} \dots x_{\nu_{l}}$, it follows that $x^{l} v_{\lambda} = 0$.
\end{proof}

\noindent The following technical lemma motivates our choice to work with weights of the form $s_{\alpha_{1}} \cdot a \omega_{1}$, $a \in \mathbb{N}_{0}$.

\begin{mylem}
\label{makestuffcommutativelemma}
Suppose $\lambda = s_{\alpha_{1}} \cdot a \omega_{1}$ with $a \in \mathbb{N}_{0}$. Then for any $\beta \in \Phi^{+}$, $f_{\alpha_{1}} f_{\beta} \un v_{\lambda} \subseteq \un v_{\lambda}$ in $L(\lambda)$.
\end{mylem}

\begin{mycor}
\label{makestuffcommutativecor}
Suppose $\lambda = s_{\alpha_{1}} \cdot a \omega_{1}$ with $a \in \mathbb{N}_{0}$. Then for any $\beta_{1},\dots,\beta_{r} \in \Phi^{+}$, $f_{\alpha_{1}}^{r} f_{\beta_{1}} \dots f_{\beta_{r}} \un v_{\lambda} \subseteq \un v_{\lambda}$ in $L(\lambda)$.
\end{mycor}

\begin{proof}[Proof of Lemma \ref{makestuffcommutativelemma}]

Fix $\beta \in \Phi^{+}$. First, note that for any $f_{\alpha} \in \mathfrak{n}$ we have $[f_{\beta},f_{\alpha}] \in \mathfrak{n}$ : because $\alpha$ has a non-zero $\alpha_{1}$-component, and hence so does $\alpha + \beta$. So if $\alpha + \beta \in \Phi$ then $f_{\alpha + \beta} \in \mathfrak{n}$, and otherwise $[f_{\alpha},f_{\beta}] = 0 \in \mathfrak{n}$. This means $[f_{\beta},\--]$ defines a derivation $\un \rightarrow \un$. Also recall that our assumptions on $\alpha_{1}$ imply $\mathfrak{n}$ is abelian.

The case $f_{\beta} \in \mathfrak{n}$ is trivial, so assume not. Thus $\alpha_{1} \notin \beta$.

Now consider the scalar generalised Verma module $M_{\Delta \setminus \alpha_{1}}(a \omega_{1})$, and write $w_{a \omega_{1}}$ for the generating highest-weight vector. So $f_{\gamma} w_{a \omega_{1}} = 0$ whenever $\alpha_{1} \notin \gamma$. In particular, we have $f_{\beta} w_{a \omega_{1}} = 0$.

Now, $f_{\alpha_{1}}^{a+1} w_{a \omega_{1}}$ is a highest-weight vector of weight $a \omega_{1} - (a+1) \alpha_{1} = \lambda$. We can see this by looking at the Verma  module of highest weight $a \omega_{1}$ (see \cite[Proposition 1.4]{Humphreys}). Write $w_{\lambda} = f_{\alpha_{1}}^{a+1} v_{\lambda}$. So $w_{\lambda}$ generates a (not necessarily irreducible) highest-weight module with highest weight $\lambda$.

Now now show that $f_{\alpha_{1}} f_{\beta} w_{\lambda} \in \un w_{\lambda}$. Using the fact that $[f_{\beta},\--]$ is a derivation of $\un$ and $\un$ is commutative, we have:

\begin{equation*}
\begin{split}
f_{\alpha_{1}} f_{\beta} w_{\lambda} & = f_{\alpha_{1}} f_{\beta} f_{\alpha_{1}}^{a+1} w_{a \omega_{1}} \\
& = \overset{a}{\underset{i=0}{\sum}} f_{\alpha_{1}} \cdot f_{\alpha_{1}}^{i}[f_{\beta},f_{\alpha_{1}}] f_{\alpha_{1}}^{a-i} w_{a \omega_{1}} + f_{\alpha_{1}}^{a+2} f_{\beta} w_{a \omega_{1}} \\
& = (a+1)[f_{\beta},f_{\alpha_{1}}]f_{\alpha_{1}}^{a+1} w_{a \omega_{1}} \\
& = (a+1) [ f_{\beta},f_{\alpha_{1}}] w_{\lambda}.
\end{split}
\end{equation*}

So $(f_{\alpha_{1}} f_{\beta} - (a+1) [f_{\beta},f_{\alpha_{1}}]) w_{\lambda} = 0$, and hence $(f_{\alpha_{1}} f_{\beta} - (a+1) [f_{\beta},f_{\alpha_{1}}]) v_{\lambda} = 0$.

Now if $x \in \un$, then $f_{\alpha_{1}} f_{\beta} x v_{\lambda} = f_{\alpha_{1}} [ f_{\beta},x] v_{\lambda} + f_{\alpha_{1}} x f_{\beta} v_{\lambda}$. We have $[f_{\beta},x] \in \un$, and also $f_{\alpha_{1}} x f_{\beta} v_{\lambda} = x f_{\alpha_{1}} f_{\beta} v_{\lambda} \in \un v_{\lambda}$ by the previous result. So $f_{\alpha_{1}} f_{\beta} x v_{\lambda} \in \un v_{\lambda}$.

\end{proof}

\begin{proof}[Proof of Corollary \ref{makestuffcommutativecor}]
By induction on $r$. The case $r = 1$ is Lemma \ref{makestuffcommutativelemma}. So now let $r > 1$, and assume true for $r-1$. Let $x \in \un$. Then:

\begin{equation*}
\begin{split}
& f_{\alpha_{1}}^{r} f_{\beta_{1}} \dots f_{\beta_{r}} x v_{\lambda} = f_{\alpha_{1}} \cdot f_{\alpha_{1}}^{r-1} f_{\beta_{1}} \cdot f_{\beta_{2}} \dots f_{\beta_{r}} x v_{\lambda} \\
& = \overset{r-2}{\underset{i=0}{\sum}} f_{\alpha_{1}} \cdot f_{\alpha_{1}}^{i} [f_{\alpha_{1}},f_{\beta_{1}}] f_{\alpha_{1}}^{r-2-i} f_{\beta_{2}} \dots f_{\beta_{r}}x v_{\lambda} + f_{\alpha_{1}} f_{\beta_{1}} f_{\alpha_{1}}^{r-1} f_{\beta_{2}} \dots f_{\beta_{r}} x v_{\lambda}\\
& = (r-1) [f_{\alpha_{1}},f_{\beta_{1}}] f_{\alpha_{1}}^{r-1} f_{\beta_{2}} \dots f_{\beta_{r}} x v_{\lambda} + f_{\alpha_{1}} f_{\beta_{1}} f_{\alpha_{1}}^{r-1} f_{\beta_{2}} \dots f_{\beta_{r}} x v_{\lambda}.
\end{split}
\end{equation*}

\noindent Now, by induction we have $ f_{\alpha_{1}}^{r-1} f_{\beta_{2}} \dots f_{\beta_{r}} x v_{\lambda} \in \un v_{\lambda}$, and so also \\ $(r-1) [f_{\alpha_{1}},f_{\beta_{1}}] f_{\alpha_{1}}^{r-1} f_{\beta_{2}} \dots f_{\beta_{r}} x v_{\lambda} \in \un v_{\lambda}$. The $r = 1$ case then implies that $f_{\alpha_{1}} f_{\beta_{1}} f_{\alpha_{1}}^{r-1} f_{\beta_{2}} \dots f_{\beta_{r}} x v_{\lambda} \in \un v_{\lambda}$, and the corollary follows.

\end{proof}

\begin{proof}[Proof of Proposition \ref{abeliantranslationfunctorprop}]
Assume condition (2) holds for $ - \omega_{1}$.

First, consider the case where $\lambda(h_{\alpha_{1}}) = -1$. In this case, $\lambda -(- \omega_{1})$ is a dominant integral weight (since $\alpha \in \Phi_{\lambda}^{+}$ for all $\alpha \in \Delta \setminus \alpha_{1}$), and also $\lambda, - \omega_{1}$ are in the same facet. So by Proposition \ref{translationfunctorsprop}, condition (2) for $ - \omega_{1}$ implies condition (2) for $\lambda$.

So now assume $\lambda(h_{\alpha_{1}}) < -1$: that is, $\Phi^{+}_{\lambda} = \Phi^{+} \setminus \alpha_{1}$ and $\Phi^{-}_{\lambda} = \{ \alpha_{1} \}$.

We can now see that $s_{\alpha_{1}} \cdot \lambda$ is a dominant integral weight: for any $\alpha \in \Phi^{+}$, we have $\langle s_{\alpha_{1}} \cdot \lambda + \rho, \alpha^{\lor} \rangle = \langle \lambda + \rho, s_{\alpha_{1}}(\alpha)^{\lor} \rangle$. If $\alpha \in \Phi^{+} \setminus \alpha_{1}$ then $s_{\alpha_{1}}(\alpha) \in \Phi^{+} \setminus \alpha_{1}$ and so $\langle s_{\alpha_{1}} \cdot \lambda + \rho, \alpha^{\lor} \rangle > 0$. Meanwhile $\langle s_{\alpha_{1}} \cdot \lambda + \rho, \alpha_{1}^{\lor} \rangle = \langle \lambda + \rho, -\alpha_{1}^{\lor} \rangle = - \lambda(h_{\alpha_{1}}) - 1 > 0$.

So say $\mu = s_{\alpha_{1}} \cdot \lambda$, a dominant integral weight.

\bigskip

\noindent \emph{Case 1}: $\mu = a \omega_{1}$ for $a \in \mathbb{N}_{0}$.

\bigskip

\noindent Proof of case 1: observe that $\lambda(h_{\alpha_{1}}) = -a - 2$ using $\lambda = s_{\alpha_{1}} \cdot \mu$. Define $\lambda' = \lambda + (a+1) \omega_{1}$, so $\lambda'(h_{\alpha_{1}}) = -1$ and $\lambda'(h_{\alpha}) \geq 0$ for all $\alpha \in \Delta \setminus \alpha_{1}$. As we have already seen, condition (2) for $- \omega_{1}$ implies condition (2) for $\lambda'$.

Now, $\lambda' - \lambda$ is a dominant integral weight, and $\lambda'$ is in the upper closure of the facet of $\lambda$. Suppose condition (2) does not hold for $\lambda$.

Then by \cite[Proposition 4.2.13]{me}, we can find an expression $(x_{0} + \sum x_{i} a_{i})v_{\lambda} = 0$, where $x_{i} \in KN$, $a_{i} \in U(\mathfrak{n}^{-}_{K})$ are weight vectors with non-zero weight, the sum is finite, and $x_{0} \neq 0$.

By Corollary \ref{makestuffcommutativecor}, we can choose $r$ large enough that $f_{\alpha_{1}}^{r} a_{i} v_{\lambda} \in \un v_{\lambda}$ for each $i$. Say $f_{\alpha_{1}}^{r} a_{i} v_{\lambda} = b_{i} v_{\lambda}$, where $b_{i} \in \un$ is a weight vector with weight strictly less than $- r \alpha_{1}$ in the usual partial ordering on weights.

So $(x_{0}f_{\alpha_{1}}^{r}+ \sum x_{i} b_{i}) v_{\lambda} = 0$, and $x_{0} f_{\alpha_{1}}^{r} + \sum x_{i} b_{i} \in \Un$. By Lemma \ref{abeliantranslationlem}, we can choose $l \in \mathbb{N}$ such that $(x_{0}f_{\alpha_{1}}^{r}+ \sum x_{i} b_{i})^{l} v_{\lambda'} = 0$. This expression has the form $f_{\alpha_{1}}^{rl}x_{0}^{l} v_{\lambda'} + \sum y_{i} c_{i} v_{\lambda'} = 0$ for some $y_{i}\in KN$ and some $c_{i} \in \un$ weight vectors with weight strictly less than $- rl \alpha_{1}$ in the partial ordering on weights. In particular, the $c_{i} v_{\lambda'}$ are all linearly independent from $f_{\alpha_{1}}^{rl} v_{\lambda'}$ (which is non-zero since $\lambda'(h_{\alpha_{1}}) \notin \mathbb{N}_{0})$.

Condition (2) for $\lambda'$ then implies that $x_{0}^{l} = 0$. But $KN$ is a domain, so this means $x_{0} = 0$, contradicting our choice of $x_{0}$. This concludes the proof in this case.

\bigskip

\noindent Now consider the general case $\lambda = s_{\alpha_{1}} \cdot \mu$ for $\mu$ dominant integral. Then let $a = \mu(h_{\alpha_{1}})$. We then have $\lambda - s_{\alpha_{1}} \cdot a \omega_{1} = s_{\alpha_{1}} \cdot  \mu - s_{\alpha_{1}}\cdot a \omega_{1}) = s_{\alpha_{1}}(\mu - a \omega_{1}) = \mu - a \omega_{1}$ since $(\mu - a \omega_{1})(h_{\alpha_{1}}) = 0$. So $\lambda - s_{\alpha_{1}} \cdot a \omega_{1}$ is dominant integral.

Moreover, $s_{\alpha_{1}} \cdot a \omega_{1}$ and $\lambda = s_{\alpha_{1}} \cdot \mu$ lie in the same facet. So by Proposition \ref{translationfunctorsprop}, condition (2) for $s_{\alpha_{1}} \cdot a \omega_{1}$ implies condition (2) for $\lambda$.

\end{proof}

\subsection{Inductive approach}
\label{inductionsection}

For this section, allow $\Phi$ to be any indecomposable root system. Also fix a weight $\lambda \in \mathfrak{h}_{K}^{*}$ with $\lambda(p^{n}\mathfrak{h}_{R}) \subseteq R$, and let $\mathfrak{n} = \mathfrak{n}_{\Delta_{\lambda}}$, $\mathfrak{p} = \mathfrak{p}_{\Delta_{\lambda}}$ be the corresponding nilpotent and parabolic subalgebras. Let $N, P$ be the subgroups of $G$ with $L_{N} = p^{n+1} \mathfrak{n}$ and $L_{P} = p^{n+1} \mathfrak{p}$.

Suppose we fix a simple root $\alpha_{1} \in \Delta$, making the assumption that $\lambda |_{\mathfrak{h}_{\Delta \setminus \alpha_{1},K}}$ is not dominant integral. We have a decomposition $\mathfrak{n} = \mathfrak{n}_{1} \oplus \mathfrak{n}_{2}$ where:

\begin{itemize}
\item $\mathfrak{n}_{1} = \Of \{ f_{\alpha} \in \mathfrak{n} \mid \alpha \text{ has a non-zero } \alpha_{1} \text{-component} \}$,

\item $\mathfrak{n}_{2} = \Of \{ f_{\alpha} \in \mathfrak{n} \mid \alpha \text{ has no } \alpha_{1} \text{-component} \}$.
\end{itemize}

\noindent Write $N_{2}$ for the $F$-uniform group with $L_{N_{2}} = p^{n+1} \mathfrak{n}_{2}$.

\begin{mydef}
Suppose $M$ is a highest-weight module (or affinoid highest-weight module) with highest weight $\lambda$.

Suppose $v \in M$. Then the \emph{$\alpha_{1}$-free component} of $v$ is the sum of all weight components of $v$ with weight $\lambda - \underset{\alpha \in \Delta}{\sum} c_{\alpha} \alpha$ such that $c_{\alpha_{1}} = 0$. We say $v$ is \emph{$\alpha_{1}$-free} if it equals its $\alpha_{1}$-free component.
\end{mydef}

\noindent Define a vector subspace:

\begin{equation*}
 W_{\alpha_{1}} = \{ x \in \widehat{M_{\Delta_{\lambda}}(\lambda)} \mid e_{\alpha_{1}}^{r} x \text{ has zero } \alpha_{1} \text{-free component for all } r \in \mathbb{N}_{0} \}.
 \end{equation*}

\noindent Recall condition (1) from \S \ref{outline}. The following theorem is a crucial tool for proving Theorem \ref{maintheoremvariant}, since it allows an inductive approach for proving that weights satisfy condition (1).

\begin{mythm}
\label{inductionprop}
Suppose the following conditions are met:

\begin{itemize}
\item there is no non-zero highest-weight vector in the subset \\ $U(\mathfrak{n}_{2,K}) \left( W_{\alpha_{1}} \cap (U(\mathfrak{n}_{1,K})V_{\lambda}) \right) \subseteq M_{\Delta_{\lambda}}(\lambda)$,

\item the weight $\lambda |_{\mathfrak{h}_{\Delta \setminus \alpha_{1},K}}$ for $\mathfrak{g}_{\Delta \setminus \alpha_{1},K}$ satisfies condition (1).
\end{itemize}

Then $\lambda$ satisfies condition (1).
\end{mythm}

\noindent The technical condition that there is no highest-weight vector in the subset $U(\mathfrak{n}_{2,K}) \left( W_{\alpha_{1}} \cap (U(\mathfrak{n}_{1,K})V_{\lambda}) \right)$ of $M_{\Delta_{\lambda}}(\lambda)$ is likely to be difficult to prove in general, since there is no general classification of highest-weight vectors in generalised Verma modules. However, we can prove that this condition holds in certain specific cases, which is done in \S \ref{hwvectors}.

We now work towards proving this proposition. Write $\Delta' = \Delta \setminus \alpha_{1}$, $\mathfrak{h}' = \mathfrak{h}_{\Delta'}$, $\lambda' = \lambda |_{\mathfrak{h}'_{K}}$. Then write $\mathfrak{p}'$ and $\mathfrak{n}'$ for the parabolic and nilpotent subalgebras of $\mathfrak{g}'$ corresponding to $\lambda'$, and write $G',P',N'$ for the subgroups of $G$ corresponding to $\mathfrak{g}',\mathfrak{p}',\mathfrak{n}'$ respectively.

\begin{mylem}
The set of all $\alpha_{1}$-free elements of $L(\lambda)$ is isomorphic as a $\ugg$-module to $L(\lambda |_{\mathfrak{h}'_{K}})$.
\end{mylem}

\begin{proof}
Write $L$ for the set of all $\alpha_{1}$-free elements of $L(\lambda)$. It is clear that $L$ is then a $\ugg$-module. We know that $L(\lambda)$ is spanned by elements of the form $f_{\beta_{1}} \dots f_{\beta_{r}} v_{\lambda}$ with $\beta_{i} \in \Phi^{+}$. 

So $L$ is spanned by the elements $f_{\beta_{1}} \dots f_{\beta_{r}} v_{\lambda}$ with all $\beta_{i} \in \Phi^{+}_{\Delta \setminus \alpha_{1}}$, that is, with $f_{\beta_{i}} \in \mathfrak{g}'$. Thus $L$ is generated over $\ugg$ by $v_{\lambda}$, and it is clear that $v_{\lambda}$ is a highest-weight vector for $\ugg$ with weight $\lambda |_{\mathfrak{h}'_{K}}$. So it remains to show that $L$ is irreducible over $\ugg$.

Suppose $v \in L$ is a highest-weight vector for $\ugg$. Then $v$ is in fact a highest-weight vector for $\ug$ in $L(\lambda)$: because the fact that $v$ is $\alpha_{1}$-free implies $h_{\alpha_{1}} v = \lambda(h_{\alpha_{1}}) v$ and $e_{\alpha_{1}} v = 0$. So since $L(\lambda)$ is irreducible, this means $v$ is a scalar multiple of $v_{\lambda}$.

This now implies that $L$ is indeed irreducible over $\ugg$ and the lemma follows.
\end{proof}

\noindent Since the completion functor is flat (\cite[Proposition 2.2.12]{me}), we now see that $\widehat{L(\lambda')}$ embeds as a $\Ugg$-module in $\widehat{L(\lambda)}$, as the set of all $\alpha_{1}$-free elements.

Recall from Lemma \ref{Vembedding} that $V_{\lambda}$ embeds in $L(\lambda)$ as the $\up$-submodule generated by $v_{\lambda}$. Following our previous notation, we write $V_{\lambda'}$ for the irreducible highest-weight $U(\mathfrak{p}'_{K})$-module with highest weight $\lambda'$.

\begin{mylem}
The embedding $L(\lambda') \rightarrow L(\lambda)$ restricts to an inclusion $V_{\lambda'} \rightarrow V_{\lambda}$.
\end{mylem}

\begin{proof}
We know that $V_{\lambda'}$ embeds in $L(\lambda')$ as the $\upp$-submodule generated by $v_{\lambda'}$. So the image of $V_{\lambda'}$ in $L(\lambda)$ is precisely $\upp v_{\lambda}$, which is clearly contained in $\up v_{\lambda} = V_{\lambda}$.
\end{proof}

\noindent Recall that (by using the PBW theorem) $\widehat{M_{\Delta_{\lambda}}(\lambda)}$ is isomorphic as a $\Un$-module to $\Un \otimes_{K} V_{\lambda}$, and hence $KG \otimes_{KN} V_{\lambda}$ is isomorphic as a $KN$-module to $KN \otimes_{K} V_{\lambda}$.

\begin{mylem}
\label{Wcontainmentlemma}
Suppose the weight $\lambda |_{\mathfrak{h}'_{K}}$ for $\mathfrak{g}'_{K}$ satisfies condition (1). Then $\ker (KG \otimes_{KP} V_{\lambda} \rightarrow \widehat{L(\lambda)}) \subseteq W_{\alpha_{1}}$.
\end{mylem}

\begin{proof}

First, we claim that if $v \in \ker (KG \otimes_{KP} V_{\lambda} \rightarrow \widehat{L(\lambda)})$ then the $\alpha_{1}$-free component of $v$ is zero.

Say $V_{\lambda}$ has a $K$-basis of weight vectors $v_{1},\dots,v_{l}$, where $v_{1},\dots,v_{l'}$ are precisely the $v_{i}$ that are $\alpha_{1}$-free. Then $v_{1},\dots,v_{l'} \in V_{\lambda'} \subseteq V_{\lambda}$ by the previous lemma.

Suppose $v \in \ker (KG \otimes_{KP} V_{\lambda} \rightarrow \widehat{L(\lambda)})$. Say $v = \sum x_{i} \otimes v_{i}$, where $x_{i} \in KN$. Using the description of the embedding of $KG$ in $\Ug$ in \cite[Corollary 2.5.5]{me}, we can write $x_{i} = y_{i} + \text{(terms of weight with non-zero } \alpha_{1} \text{-component)}$, where $y_{i} \in KN'$.

Now the $\alpha_{1}$-free component of $v$ is $\underset{i \leq l'}{\sum} y_{i} \otimes v_{i}$, with $y_{i} \in KN'$ and each  $v_{i} \in V_{\lambda'}$ with $i \leq l'$. Since $v \in \ker (KG \otimes_{KP} V_{\lambda} \rightarrow \widehat{L(\lambda)})$, we must have $\underset{i \leq l'}{\sum} y_{i} v_{i} = 0$ in $\widehat{L(\lambda')} \subseteq \widehat{L(\lambda)}$. By condition (1) for $\lambda'$, each $y_{i} = 0$ with $i < l'$, and the $\alpha_{1}$-free component of $v$ is zero.

\bigskip

\noindent Now fix $v \in \ker (KG \otimes_{KP} V_{\lambda} \rightarrow \widehat{L(\lambda)})$, and we claim $v \in W_{\alpha_{1}}$.

Write $E_{\alpha_{1}} = \text{exp}(p^{n+1}e_{\alpha_{1}}-1) \in KN \subseteq \Un$. By using \cite[Lemma 2.5.6]{me}, we can see that $K \otimes_{R} R \llbracket E_{\alpha_{1}} \rrbracket \subseteq KN$ is dense in $\widehat{U(\Of e_{\alpha_{1}})}_{n,K} \subseteq \Un$.

For $r \in \mathbb{N}_{0}$, we have $E_{\alpha_{1}}^{r} v \in \ker (KG \otimes_{KN} V_{\lambda} \rightarrow \widehat{L(\lambda)})$, so the $\alpha_{1}$-free component of $E_{\alpha_{1}}^{r}v$ is $0$. 

Therefore for any $x \in K \otimes_{R} R \llbracket E_{\alpha_{1}} \rrbracket$, $x v$ has $\alpha_{1}$-free component $0$. Using the density of $K \otimes_{R} R \llbracket E_{\alpha_{1}} \rrbracket$ in $\widehat{U(\Of e_{\alpha_{1}})}_{n,K}$, we see that for any $r \in \mathbb{N}_{0}$, $e_{\alpha_{1}}^{r} v$ has $\alpha_{1}$-free component $0$.

Thus $v \in W_{\alpha_{1}}$.

\end{proof}

\noindent For $l \in \mathbb{N}_{0}$, define a $K$-vector space:

\begin{equation*}
\begin{split}
\Omega_{l}^{\alpha_{1}} \widehat{M_{\Delta_{\lambda}}(\lambda) } = & \{ v \in \widehat{M_{\Delta_{\lambda}}(\lambda)} \mid \text{ all non-zero weight components of } v \text{ have weight} \\ & \lambda - \underset{\alpha \in \Delta}{\sum} c_{\alpha} \alpha \text{ with } c_{\alpha_{1}} = l \}. 
\end{split}
\end{equation*}

\noindent So any $v \in \widehat{M_{\Delta_{\lambda}}(\lambda)}$ can be uniquely expressed as a convergent sum $v = \sum v_{l}$ with $v_{l} \in \Omega^{\alpha_{1}}_{l}$. 

For subspaces $V \subseteq \widehat{M}_{\Delta_{\lambda}}(\lambda)$ we will also write $\Omega_{l}^{\alpha_{1}} V = V \cap \Omega_{l}^{\alpha_{1}} \widehat{M_{\Delta_{\lambda}}(\lambda)}$.

\begin{mylem}
\label{reduceton1lemma}
Suppose $v \in KG \otimes_{KN}V_{\lambda} \subseteq \widehat{M_{\Delta_{\lambda}}(\lambda)}$, such that $v \in W_{\alpha_{1}}$. Then for any $l$, the $\Omega_{l}^{\alpha_{1}}$-component of $v$ lies in $KN_{2} \cdot \left(W_{\alpha_{1}} \cap \left( U(\mathfrak{n}_{1,K})V_{\lambda} \right) \right)$.
\end{mylem}

\noindent To prove this, we will first prove the following slightly more general lemma.

\begin{mylem}
\label{reduceton1lemmageneral}
Suppose $A$ is a $K$-subalgebra of $KN$, $\theta : \widehat{M_{\Delta_{\lambda}}(\lambda)} \rightarrow \widehat{M_{\Delta_{\lambda}}(\lambda)}$ is an $A$-module homomorphism, and $V \subseteq M_{\Delta_{\lambda}}(\lambda)$ is a $K$-subspace such that $\theta(V) \subseteq M_{\Delta_{\lambda}}(\lambda)$.

Then $\ker \theta \cap (A \cdot V) = A \cdot (\ker \theta \cap V)$.
\end{mylem}

\begin{proof}
First, we claim that the multiplication map $\mu : KN \otimes_{K} M_{\Delta_{\lambda}}(\lambda) \rightarrow \widehat{M_{\Delta_{\lambda}}(\lambda)}$ is injective. This holds because $\widehat{M_{\Delta_{\lambda}}(\lambda)}$ is isomorphic as a $\Un$-module to $\Un \otimes_{K} V_{\lambda}$, and the multiplication map $KN \otimes_{K} \un \rightarrow \Un$ is injective by the proof of \cite[Theorem 5.3]{verma}.

Now consider the following diagram of $A$-module homomorphisms:
\[\begin{tikzcd}
	{A \otimes_{K}V} & {A \otimes_{K}M_{\Delta_{\lambda}}(\lambda)} \\
	{A \cdot V} & {A \cdot M_{\Delta_{\lambda}}(\lambda)}
	\arrow["{1 \otimes \theta}", from=1-1, to=1-2]
	\arrow["\mu", from=1-1, to=2-1]
	\arrow["\mu", from=1-2, to=2-2]
	\arrow["\theta"', from=2-1, to=2-2]
\end{tikzcd}\]

Here the two vertical maps are given by the restriction of the multiplication map $\mu : KN \otimes_{K} M_{\Delta_{\lambda}}(\lambda) \rightarrow \widehat{M_{\Delta_{\lambda}}(\lambda)}$, and in particular are injective. They are clearly also surjective, so are in fact $A$-isomorphisms.

We claim this diagram is commutative. If $x \in A$ and $v \in V$, then $\mu \circ (1 \otimes \theta )(x \otimes v) = x \theta(v) = \theta(xv) = \theta \circ \mu (x \otimes v)$.

Therefore the kernels of $1 \otimes \theta: A \otimes_{K} V \rightarrow A \otimes_{K} M_{\Delta_{\lambda}}(\lambda)$ and $\theta : A \cdot V \rightarrow \theta A \cdot M_{\Delta_{\lambda}}(\lambda)$ are identified by $\mu$. We can see that the kernel of $1 \otimes \theta: A \otimes_{K} V \rightarrow A \otimes_{K} M_{\Delta_{\lambda}}(\lambda)$ is $A \otimes_{K} \ker \theta \cap V$, so this shows us that $A \cdot (\ker \theta \cap V) = \ker \theta \cap (A \cdot V)$.

\end{proof}

\begin{proof}[Proof of Lemma \ref{reduceton1lemma}]
In the statement of the previous lemma, take $A = KN_{2}$, $\theta$ the action of $e_{\alpha_{1}}^{l}$ on $\widehat{M_{\Delta_{\lambda}}(\lambda)}$ (which is an $A$-homomorphism because $e_{\alpha_{1}}$ commutes with $KN_{2}$ in $\Ug$), and $V = \Omega_{l}^{\alpha_{1}} ( U(\mathfrak{n}_{1,K})V_{\lambda})$. Note that $W_{\alpha_{1}} \cap \Omega_{l}^{\alpha_{1}} \widehat{M_{\Delta_{\lambda}}(\lambda)} = \ker \theta \cap \Omega_{l}^{\alpha_{1}} \widehat{M_{\Delta_{\lambda}}(\lambda)}$ by definition of $W_{\alpha_{1}}$.

We know that $KG \otimes_{KP} V_{\lambda}$ is generated as a $KN$-module by $V_{\lambda}$. By the description of Iwasawa algebras in \cite[Corollary 2.5.5]{me}, we can see that every element of $KG \otimes_{KP} V_{\lambda}$ is a convergent sum of elements of the form $x y w$ with $x \in KN_{2}, y \in KN_{1}, w \in V_{\lambda}$.

So we can see that the $\Omega_{l}^{\alpha_{1}}$-component of $v$ lies in $KN_{2} \Omega_{l}^{\alpha_{1}}\left( U(\mathfrak{n}_{1,K})V_{\lambda} \right)$.

Then Lemma \ref{reduceton1lemmageneral} implies that $W_{\alpha_{1}} \cap \left(KN_{2} \cdot \Omega_{l}^{\alpha_{1}}(U(\mathfrak{n}_{1,K})V_{\lambda})\right) = KN_{2} \cdot W_{\alpha_{1}} \cap \left(\Omega_{l}^{\alpha_{1}}(U(\mathfrak{n}_{1,K})V_{\lambda}) \right) $. The result follows.

\end{proof}

\noindent Now for convenience write $W = U(\mathfrak{n}_{2,K}) (W_{\alpha_{1}} \cap U(\mathfrak{n}_{1,K})V_{\lambda}) \subseteq M_{\Delta_{\lambda}}(\lambda)$. We can see from the above lemma that if $v \in KG \otimes_{KP} V_{\lambda}$ with $v \in W_{\alpha_{1}}$, then all the $\Omega_{l}^{\alpha_{1}}$-components of $v$ lie in the closure $\bar{W}$ of $W$ in $\widehat{M}_{\Delta_{\lambda}}(\lambda)$.  Therefore $v \in \bar{W}$.

We can see that $W$ is spanned by weight vectors: first, if $x \in W_{\alpha_{1}} \subseteq \widehat{M_{\Delta_{\lambda}}(\lambda)}$, then by definition of $W_{\alpha_{1}}$ all weight components of $x$ lie in $W_{\alpha_{1}}$. So if $x \in W_{\alpha_{1}} \cap ( U(\mathfrak{n}_{1,K})V_{\lambda})$ then all weight components of $x$ lie in $ W_{\alpha_{1}} \cap ( U(\mathfrak{n}_{1,K})V_{\lambda})$. Thus $W = U(\mathfrak{n}_{2,K}) (W_{\alpha_{1}} \cap U(\mathfrak{n}_{1,K})V_{\lambda})$ is spanned by weight vectors.

\begin{mylem}
\label{hwvectorexistencelem}
Suppose there exists $0 \neq w \in \bar{W}$ such that $KG w \subseteq \bar{W}$. Then there is a highest-weight vector in $W$.
\end{mylem}

\begin{proof}

Since $KG$ is dense in $\Ug$ and $\bar{W}$ is closed, we have $\Ug w \subseteq \bar{W}$.

Let $j \geq 0$ be minimal such that there exists $w' \in \Ug w$ with a non-zero component of weight $\mu = \lambda - \underset{\alpha \in \Delta}{\sum} t_{\alpha} \alpha$ with $|t| = j$. Then choose such $w'$ and $\mu$.

Now consider the weight component $w'_{\mu}$ of weight $\mu$. If $\alpha \in \Phi^{+}$, then $e_{\alpha} w'$ has $(\mu + \alpha)$-weight component $e_{\alpha} w'_{\mu}$. By the minimality of $j$, this means that $e_{\alpha} w'_{\mu} = 0$. Therefore $w'_{\mu}$ is a highest-weight vector. We know $w' \in \bar{W}$, and by Lemma \ref{closurelem} this implies that $w'_{\mu} \in W$.
\end{proof}

\noindent We are now ready to prove Theorem \ref{inductionprop} by combining the previous results.

\begin{proof}[Proof of Theorem \ref{inductionprop}]
We know condition (1) is equivalent to the statement that the map $KG \otimes_{KN} V_{\lambda} \rightarrow \widehat{L(\lambda)}$ is injective.

Suppose for contradiction that $0 \neq w \in \ker (KG \otimes_{KN} V_{\lambda} \rightarrow \widehat{L(\lambda)})$. Then $KG w \subseteq \ker (KG \otimes_{KN} V_{\lambda} \rightarrow \widehat{L(\lambda)})$.

By Lemma \ref{Wcontainmentlemma}, $KG w \subseteq W_{\alpha_{1}} \cap KG \otimes_{KN}V_{\lambda}$.

By Lemma \ref{reduceton1lemma}, this implies that $KG w \subseteq \bar{W}$. Finally, by Lemma \ref{hwvectorexistencelem}, this implies that there is a highest-weight vector in $W = U(\mathfrak{n}_{2,K}) (W_{\alpha_{1}} \cap U(\mathfrak{n}_{1,K})V_{\lambda}) \subseteq M_{\Delta_{\lambda}}(\lambda)$, a contradiction.

\end{proof}

\noindent Having established our general result Theorem \ref{inductionprop}, we now want to look at some specific cases in which we can find a useful description of the subspace $U(\mathfrak{n}_{2,K}) \left( W_{\alpha_{1}} \cap (U(\mathfrak{n}_{1,K})V_{\lambda}) \right)$.

\begin{myprop}
\label{Walphadescription}
Suppose $\lambda$ and $\alpha_{1} \in \Delta$ satisfy the following conditions.

\begin{itemize}
\item For all $\alpha \in \Phi^{+}$, the $\alpha_{1}$-coefficient of $\alpha$ is either $0$ or $1$.

\item $W_{\alpha_{1}} \cap \underset{i}{\sum}f_{\alpha_{1}}^{i}V_{\lambda} = 0$.

\item If $\alpha \in \Phi^{+} \setminus \Phi_{\Delta_{\lambda}}^{+}$ such that $\alpha - \alpha_{1} \in \Phi^{+}_{\Delta_{\lambda}}$, then $f_{\alpha - \alpha_{1}} \cdot V_{\lambda} = 0$.

\end{itemize}

\noindent Then $W_{\alpha_{1}} \cap (U(\mathfrak{n}_{1,K})V_{\lambda}) = I V_{\lambda}$, where $I$ is the left ideal of $U(\mathfrak{n}_{1,K})$ generated by all $f_{\alpha}$ with $\alpha \in \Phi^{+} \setminus \Phi_{\Delta_{\lambda}}^{+}$ such that $\alpha_{1} \in \alpha$, $\alpha \neq \alpha_{1}$, and $\alpha - \alpha_{1} \notin \Phi^{+} \setminus \Phi_{\Delta_{\lambda}}^{+}$.

\end{myprop}

\begin{proof}

Write $V'_{\lambda} = \sum f_{\alpha_{1}}^{i} V_{\lambda}$ for convenience. Now for each $l$, $\Omega_{l}^{\alpha_{1}}V'_{\lambda}$ is finite-dimensional, and our assumption that $W_{\alpha_{1}} \cap V'_{\lambda} = 0$ implies that $e_{\alpha_{1}}^{l} : \Omega_{l}^{\alpha_{1}} V'_{\lambda} \rightarrow \Omega_{0}^{\alpha_{1}} V'_{\lambda}$ is injective.

Recall that $\mathfrak{n}$ is spanned by all $f_{\alpha}$ with $\alpha \in \Phi^{+} \setminus \Phi_{\Delta_{\lambda}}^{+}$, and $\mathfrak{n}_{1}$ is spanned by all $f_{\alpha}$ with $\alpha \in \Phi^{+} \setminus \Phi_{\Delta_{\lambda}}^{+}$ and $\alpha_{1} \in \alpha$.

The condition that all $\alpha \in \Phi^{+}$ have $\alpha_{1}$-component $0$ or $1$ implies that $\mathfrak{n}_{1}$ is abelian, because it means if $\alpha, \beta \in \Phi^{+}$ with $\alpha_{1} \in \alpha$ and $\alpha_{1} \in \beta$ then we cannot have $\alpha + \beta \in \Phi^{+}$. So $[f_{\alpha},f_{\beta}] = 0$.

\begin{itemize}
\item Let $S_{1}$ be the set of all $\alpha \in \Phi^{+} \setminus \Phi_{\Delta_{\lambda}}^{+}$ with $\alpha \neq \alpha_{1}$, such that $\alpha - \alpha_{1} \notin \Phi^{+} \setminus \Phi_{\Delta_{\lambda}}^{+}$. So $I = U(\mathfrak{n}_{1,K}) \{f_{\alpha} \mid \alpha \in S_{1} \}$.

\item Let $S_{2}$ be the set of all $\alpha \in \Phi^{+} \setminus \Phi_{\Delta_{\lambda}}^{+}$ with $\alpha \neq \alpha_{1}$, such that $\alpha - \alpha_{1} \in \Phi^{+} \setminus \Phi_{\Delta_{\lambda}}^{+}$.
\end{itemize}

We fix orderings of $S_{1},S_{2}$ and assume products indexed by $S_{1},S_{2}$ are taken using these orders.

Now for each $l \geq 0$, write $v_{l,1},\dots, v_{l,r_{l}}$ for a $K$-basis of $\Omega_{l}^{\alpha_{1}} V'_{\lambda}$. Note that we could have $r_{l} = 0$.

Then $U(\mathfrak{n}_{1,K}) V_{\lambda}$ has a basis as follows:

\begin{equation*}
x_{t,u,i,j} = \underset{\alpha \in S_{2}}{\prod} f_{\alpha}^{t_{\alpha}} \underset{\alpha \in S_{1}}{\prod} f_{\alpha}^{u_{\alpha}} v_{i,j}
\end{equation*}

\noindent for $t \in \mathbb{N}_{0}^{S_{2}}, u \in \mathbb{N}_{0}^{S_{1}}, i \in \mathbb{N}_{0}, 1 \leq  j \leq r_{i}$. So since each $\alpha \in S_{1} \cup S_{2}$ has $\alpha_{1}$-coefficient $1$, we can see that $x_{t,u,i,j} \in \Omega_{i+|t|+|u|}^{\alpha_{1}} M_{\Delta_{\lambda}}(\lambda)$.

Now $[e_{\alpha_{1}},\--]$ defines a derivation on $\ug$. For $\alpha \in S_{1} \cup S_{2}$, we know from the definition of a Chevalley basis (Definition \ref{chevalley}) that $[e_{\alpha_{1}}, f_{\alpha}] = a_{\alpha} f_{\alpha - \alpha_{1}}$, where we say $f_{\alpha - \alpha_{1}} = 0$ if $\alpha - \alpha_{1} \notin \Phi^{+}$, and $a_{\alpha}$ is a scalar which is non-zero if $f_{\alpha - \alpha_{1}} \neq 0$. For convenience we will also take $a_{\alpha}$ to be some non-zero scalar when $f_{\alpha - \alpha_{1}} = 0$. We can see that $[e_{\alpha_{1}},f_{\alpha - \alpha_{1}}] = 0$ whenever $\alpha \in S_{1} \cup S_{2}$, since $\alpha$ must have $\alpha_{1}$-coefficient $1$. We can then calculate the following using the Leibniz identity:

\begin{equation}
e_{\alpha_{1}}^{i+|t|+|u|} x_{t,u,i,j} = c_{t,u,i,j} \underset{S_{2}}{\prod} f_{\alpha -\alpha_{1}}^{t_{\alpha}} \underset{S_{1}}{\prod} f_{\alpha - \alpha_{1}}^{u_{\alpha}} e_{\alpha_{1}}^{i}v_{i,j},
\end{equation}

\noindent where $c_{t,u,i,j} = \binom{i+|t|+|u|}{i,t,u} \underset{S_{2}}{\prod} a_{\alpha}^{t_{\alpha}} \underset{S_{1}}{\prod} a_{\alpha}^{u_{\alpha}}$ is a non-zero scalar.

\bigskip

\noindent Now consider the map $e_{\alpha_{1}}^{l} : \Omega_{l}^{\alpha_{1}} U(\mathfrak{n}_{1,K})V_{\lambda} \rightarrow \Omega_{0}^{\alpha_{1}} M_{\Delta_{\lambda}}(\lambda)$, which has kernel $W_{\alpha_{1}} \cap \Omega_{l}^{\alpha_{1}} U(\mathfrak{n}_{1,K})V_{\lambda}$ by definition of $W_{\alpha_{1}}$.

$\Omega_{l}^{\alpha_{1}} U(\mathfrak{n}_{1,K})V_{\lambda}$ has a $K$-basis consisting of all the $x_{t,u,i,j}$ with $i+|t|+|u| = l$.

\bigskip

\noindent  Suppose we choose $t,u,i,j$ such that $i+|t|+|u| = l$ and $u \neq 0$. Then note that $e_{\alpha_{1}}^{i}v_{i,j} \in \Omega_{0}^{\alpha_{1}} V'_{\lambda} \subseteq V_{\lambda}$. Now $f_{\alpha - \alpha_{1}} V_{\lambda} = 0$ if $\alpha \in S_{1}$: because either $\alpha - \alpha_{1} \notin \Phi^{+}$, in which case $f_{\alpha - \alpha_{1}} = 0$, or $\alpha - \alpha_{1} \in \Phi^{+}_{\Delta_{\lambda}}$, in which case this holds by assumption. So since $u \neq 0$, we can see that $e_{\alpha_{1}}^{i+|t|+|u|} x_{t,u,i,j} = 0$ using Equation (1).

\bigskip

\noindent Now instead consider the basis elements with $u = 0$, and $i + |t| = l$. We have $e_{\alpha_{1}}^{i+|t|} x_{t,0,i,j} = c_{t,0,i,j} \underset{S_{2}}{\prod} f_{\alpha - \alpha_{1}}^{t_{\alpha}} e_{\alpha_{1}}^{i}v_{i,j}$, with the scalar $c_{t,0,i,j} \neq 0$. Now the $\underset{S_{2}}{\prod} f_{\alpha - \alpha_{1}}^{t_{\alpha}}$  are elements of $\un$ by definition of $S_{2}$, and are linearly independent by the PBW theorem. Also if $t$ is fixed, and so $i = l - |t|$ is fixed, then we know the $e_{\alpha_{1}}^{i}v_{i,j}$ are all linearly independent since $W_{\alpha_{1}} \cap V'_{\lambda} = 0$.

It follows that the $e_{\alpha_{1}}^{l}x_{t,0,i,j}$ with $i+|t| = l$ are all linearly independent.

So $W_{\alpha_{1}} \cap \Omega_{l}^{\alpha_{1}} U(\mathfrak{n}_{1,K})V_{\lambda}$ is precisely the span of the $x_{t,u,i,j}$ with $i+|t|+|u| = l$, and $u \neq 0$. Therefore $W_{\alpha_{1}} \cap U(\mathfrak{n}_{1,K})V_{\lambda}$ is the span of all the $x_{t,u,i,j}$ with $u \neq 0$. 

From the definition of the basis $x_{t,u,i,j}$, we can now see that $W_{\alpha_{1}} \cap U(\mathfrak{n}_{1,K})V_{\lambda} = U(\mathfrak{n}_{1,K}) \{ f_{\alpha} \mid \alpha \in S_{1} \} V'_{\lambda} = I \cdot \underset{i \geq 0}{\sum} f_{\alpha_{1}}^{i}V_{\lambda}$.

\bigskip

\noindent If $\alpha_{1} \in \Delta_{\lambda}$, then $\sum f_{\alpha_{1}}^{i} V_{\lambda} = V_{\lambda}$, so $I \cdot \sum f_{\alpha_{1}}^{i}V_{\lambda} = I V_{\lambda}$ and we are done.

If $\alpha_{1} \notin \Delta_{\lambda}$, then $f_{\alpha_{1}} \in \mathfrak{n}_{1}$. Using the fact that $\mathfrak{n}_{1}$ is abelian, we can see that $I \cdot \sum f_{\alpha_{1}}^{i}V_{\lambda} = I V_{\lambda}$ and we are done.

\end{proof}

\subsection{Inductive approach in type D}
\label{inductionsectiontypeD}

For this section, assume $\Phi$ has type $D_{m}$, where $m \geq 4$. Recall the labelling of positive roots from \S \ref{Liealgsec} as follows:

\begin{itemize}
\item $\alpha_{i,j}:=\alpha_{i}+\dots+\alpha_{j}$ for $1 \leq i \leq j \leq m$,

\item $\beta_{i}:=\alpha_{i}+\dots+\alpha_{m-2}+\alpha_{m}$ for $1 \leq i \leq m-2$,

\item $\gamma_{i,j}:=\alpha_{i}+\dots+\alpha_{j-1}+2\alpha_{j}+\dots+2\alpha_{m-2}+\alpha_{m-1}+\alpha_{m}$ for $1 \leq i < j \leq m-2$,
\end{itemize}

\noindent where $\alpha_{1},\dots, \alpha_{m}$ are the simple roots.

 Note that $\alpha_{1}$ in our labelling satisfies the condition that all $\alpha \in \Phi^{+}$ have $\alpha_{1}$-coefficient $0$ or $1$ (by inspection of the roots), which is one of the conditions for Proposition \ref{Walphadescription}.

Also note by inspection of the roots that if $\alpha \in \Phi^{+} \setminus \alpha_{1}$ such that $\alpha_{1} \in \alpha$ and $\alpha - \alpha_{1} \notin \Phi^{+}$, then $\alpha = \gamma_{1,2}$.

\begin{mycor}
\label{inductioncor1}
Suppose $\lambda$ satisfies the following conditions:
\begin{itemize}
\item $M_{\Delta_{\lambda}}(\lambda)$ is scalar (that is, $\lambda(h_{\alpha}) = 0$ for all $\alpha \in \Delta_{\lambda}$),

\item $\alpha_{1} \notin \Delta_{\lambda}$, 

\item for some $1 < k \leq m$ such that $k \neq m-1$, we have $\alpha_{k} \notin \Delta_{\lambda}$.

\end{itemize}

Then $W_{\alpha_{1}} \cap (U(\mathfrak{n}_{1,K})V_{\lambda}) = U(\mathfrak{n}_{1,K})(\{f_{\gamma_{1,2}}\}\cup \{f_{\alpha_{1,j}} \mid 2 \leq j < k \}) V_{\lambda}$, where $k \in \{2,\dots,m \}$ is minimal such that $\alpha_{k} \notin \Delta_{\lambda}$.

\end{mycor}

\noindent The condition that $k \neq m-1$ is not particularly necessary and mostly for notational convenience. To prove this corollary, it is helpful to be able to carry out some explicit calculations, and so we use the following lemma.

\begin{mylem}
\label{calculationlem}
Suppose $\alpha \in \Delta$, $M$ is a $\ug$-module, and $v \in M$ is a weight vector of weight $\lambda$ such that $e_{\alpha}v = 0$. Then $e_{\alpha} f_{\alpha}^{i}v = i(1+\lambda(h_{\alpha}-i) f_{\alpha}^{i-1}v$ for each $i \in \mathbb{N}_{0}$.

Moreover, $e_{\alpha}^{i} f_{\alpha}^{i} v$ is a non-zero scalar multiple of $v$ if $\lambda(h_{\alpha}) \notin \{0,\dots,i-1 \}$, and zero otherwise.
\end{mylem}

\begin{proof}
The first part is \cite[Lemma 4.2.1]{me}. Given this, we can calculate $e_{\alpha}^{i} f_{\alpha}^{i} v = \overset{i}{\underset{j=1}{\prod}} \left( j(1+ \lambda(h_{\alpha})-j) \right) v$ and the second part follows.
\end{proof}

\begin{proof}[Proof of Corollary \ref{inductioncor1}]
Fix $2 \leq k \leq m$ to be minimal such that $\alpha_{k} \notin \Delta_{\lambda}$.

We apply Proposition \ref{Walphadescription}. Firstly, we need to check that $W_{\alpha_{1}}\cap \sum f_{\alpha_{1}}^{i}V_{\lambda} = 0$. Since $M_{\Delta_{\lambda}}(\lambda)$ is scalar, $V_{\alpha_{1}}$ is one-dimensional. So $\sum f_{\alpha_{1}}^{i} V_{\lambda}$ has $K$-basis $f_{\alpha_{1}}^{K} w_{\lambda}, i \in \mathbb{N}_{0}$, where $w_{\lambda}$ denotes the highest-weight vector generating $M_{\Delta_{\lambda}}(\lambda)$.

Using \ref{calculationlem}, $e_{\alpha_{1}}^{i} \cdot f_{\alpha_{1}}^{i} w_{\lambda}$ is a non-zero scalar multiple of $w_{\lambda}$ since $\alpha_{1} \notin \Delta_{\lambda}$ (so $\lambda(h_{\alpha_{1}}) \notin \mathbb{N}_{0}$). So indeed we see that $W_{\alpha_{1}} \cap \sum f_{\alpha_{1}}^{i}V_{\lambda} = 0$.

Moreover, $f_{\alpha}V_{\lambda} = 0$ for all $\alpha \in \Phi^{+}_{\Delta_{\lambda}}$ by the assumption that $M_{\Delta_{\lambda}}(\lambda)$ is scalar. So the third condition for Proposition \ref{Walphadescription} is met. 

\bigskip

\noindent Now suppose $\alpha \in \Phi^{+} \setminus \Phi_{\Delta_{\lambda}}^{+}$ such that $\alpha \neq \alpha_{1}$, $\alpha_{1} \in \alpha$, and $\alpha - \alpha_{1} \notin \Phi^{+} \setminus \Phi_{\Delta_{\lambda}}^{+}$.

If $\alpha - \alpha_{1} \notin \Phi^{+}$, then $\alpha = \gamma_{1,2}$. Now suppose instead that $\alpha - \alpha_{1} \in \Phi_{\Delta_{\lambda}}^{+}$. Then $\alpha_{k} \notin \alpha$, so $\alpha \neq \alpha_{1,i}$ for $i \geq k$, and also $\alpha \neq \gamma_{1,j}$ for all $j \geq 2$. Since some $\alpha_{l} \notin \Delta_{\lambda}$ with $l \neq m-1$, we must have $\alpha \neq \beta_{1}$.

So we can see that $\alpha$ must have the form $\alpha_{1,i}$ for some $1 < i < k$, and the result follows from Proposition \ref{Walphadescription}.
\end{proof}

\begin{mycor}
\label{inductioncor2}
Suppose $\alpha_{1},\alpha_{2} \notin \Delta_{\lambda}$. Then we have: \\ $W_{\alpha_{1}} \cap (U(\mathfrak{n}_{1,K})V_{\lambda}) = U(\mathfrak{n}_{1,K})f_{\gamma_{1,2}}V_{\lambda}$.
\end{mycor}

\begin{proof}
First, we show that $W_{\alpha_{1}} \cap \sum f_{\alpha_{1}}^{i}V_{\lambda} = 0$. Any weight vector in $V_{\lambda}$ must have a weight of the form $\mu = \lambda - \sum_{j \geq 3} c_{j} \alpha_{j}$, using the fact that $\alpha_{1},\alpha_{2} \notin \Delta_{\lambda}$. In particular, this weight satisfies $\mu(h_{\alpha_{1}}) = \lambda(h_{\alpha_{1}}) \notin \mathbb{N}_{0}$, since $(\alpha_{1},\alpha_{j}) = 0$ for $j \geq 3$.

So suppose $w \in \sum f_{\alpha_{1}}^{i} V_{\lambda}$ is a non-zero weight vector, then $w$ has the form $w = f_{\alpha_{1}}^{i} v$, where $v \in V_{\lambda}$ is a weight vector of weight $\mu$ satisfying $\mu(h_{\alpha_{1}}) = \lambda(h_{\alpha_{1}}) \notin \mathbb{N}_{0}$. By Lemma \ref{calculationlem}, $e_{\alpha_{1}}^{i} w$ is a non-zero multiple of $v$, which is $\alpha_{1}$-free. So $w \notin W_{\alpha_{1}}$.

Now if $\alpha \in \Phi^{+} \setminus \Phi_{\Delta_{\lambda}}^{+}$ such that $\alpha_{1} \neq \alpha$, $\alpha_{1} \in \alpha$ and $\alpha - \alpha_{1} \notin \Phi^{+} \setminus \Phi_{\Delta_{\lambda}}^{+}$, then $\alpha_{2} \in \alpha$ (by inspection of the roots in $\Phi^{+}$). This means that $\alpha - \alpha_{1} \notin \Phi_{\Delta_{\lambda}}^{+}$, so $\alpha = \gamma_{1,2}$. This also means that the third condition for Proposition \ref{Walphadescription} holds trivially.

The result now follows from Proposition \ref{Walphadescription}.
\end{proof}

\begin{mycor}
\label{inductioncorscalar}
Suppose $\alpha_{1} \in \Delta_{\lambda}$, and $M_{\Delta_{\lambda}}(\lambda)$ is a scalar generalised Verma module. Then $W_{\alpha_{1}} \cap (U(\mathfrak{n}_{1,K})V_{\lambda}) = U(\mathfrak{n}_{1,K})f_{\gamma_{1,2}}V_{\lambda}$.
\end{mycor}

\begin{proof}
Firstly, since $\alpha_{1} \in \Delta_{\lambda}$, we have $\sum f_{\alpha_{1}}^{i} V_{\lambda} = V_{\lambda}$, and $V_{\lambda}$ is one-dimensional since $M_{\Delta_{\lambda}}(\lambda)$ is scalar. Thus any element of $V_{\lambda}$ is $\alpha_{1}$-free, so $W_{\alpha_{1}} \cap V_{\lambda} = 0$.

If $\alpha \in \Phi^{+} \setminus \Phi_{\Delta_{\lambda}}^{+}$ with $\alpha_{1} \in \alpha$, then $\alpha - \alpha_{1}$ has a non-zero coefficient of some element of $\Delta \setminus\Delta_{\lambda}$ since $\alpha$ does. So $\alpha - \alpha_{1} \notin \Phi^{+} \setminus \Phi^{+}_{\Delta_{\lambda}}$. If $\alpha - \alpha_{1} \notin  \Phi^{+}$, then $\alpha = \gamma_{1,2}$ as seen previously.

The result follows from Proposition \ref{Walphadescription}.
\end{proof}

\noindent In most cases, if $\alpha_{1} \in \Delta_{\lambda}$ and $M_{\Delta_{\lambda}}(\lambda)$ is not scalar, the condition $W_{\alpha_{1}} \cap V_{\lambda} = 0$ will in fact not be met. However, we will need to consider a specific situation where this condition is satisfied.

\begin{mycor}
\label{inductioncor3}
Suppose $1 \leq k < m-2$ such that $\lambda(h_{\alpha_{i}}) = 0$ for $i < k$, $\lambda(h_{\alpha_{k}}) \in \mathbb{N}_{0}$, and $\lambda(h_{\alpha_{k+1}}) \notin \mathbb{N}_{0}$. Then $W_{\alpha_{1}} \cap U(\mathfrak{n}_{1,K})V_{\lambda} = U(\mathfrak{n}_{1,K}) f_{\gamma_{1,2}}V_{\lambda}$.
\end{mycor}

\noindent The main difficulty in this case is proving that $W_{\alpha_{1}} \cap V_{\lambda} = 0$. To this end, we want to establish an explicit description of $V_{\lambda}$.

\begin{mylem}
Let $1 \leq k \leq m-2$. Write $\Delta_{k} = \{ \alpha_{1},\dots,\alpha_{k} \}$, so the root subsystem $\Phi_{\Delta_{k}}$ spanned by $\Delta_{k}$ has type $A_{k}$. Let $\mu : \mathfrak{h}_{\Delta_{k},K} \rightarrow K $ be defined by $\mu(h_{\alpha_{i}}) = 0$ for $i < k$, and $\mu( h_{\alpha_{k}}) = a$ where $a \in \mathbb{N}_{0}$.

Then the $U(\mathfrak{g}_{\Delta_{k},K})$-module $L(\mu)$ has $K$-basis $\{ \overset{k}{\underset{i=1}{\prod}}f_{\alpha_{i,k}}^{t_{i}}v_{\mu} \mid t \in \mathbb{N}_{0}^{k}, |t| \leq a \}$, where $v_{\mu}$ is the highest-weight vector.
\end{mylem}

\begin{proof}
This is a standard result in the representation theory of $\mathfrak{sl}_{2}$ if $k = 1$, so assume $k > 1$.

The weights $\alpha_{i,k} = \alpha_{i} + \dots + \alpha_{k}$ for $1 \leq i \leq k$ are linearly independent, so the $\prod f_{\alpha_{i,k}}^{t_{i}}v_{\mu}$ have distinct weights. Therefore as long as they are non-zero, they must be linearly independent.

By explicit calculation, $e_{\alpha_{k-1,k-1}}^{t_{k}} \dots e_{\alpha_{1,k-1}}^{t_{1}} \cdot \prod f_{\alpha_{i,k}}^{t_{i}}v_{\mu}$ is a non-zero scalar multiple of $f_{\alpha_{k}}^{|t|}v_{\mu}$, hence $\prod f_{\alpha_{i,k}}^{t_{i}}v_{\mu}$ is non-zero as long as $|t| \leq a$ (by applying Lemma \ref{calculationlem}).

Also by explicit calculation, for any $t$, $f_{\alpha_{k-1,k-1}}^{t_{k-1}} \dots f_{\alpha_{1,k-1}}^{t_{1}} \cdot f_{\alpha_{k}}^{|t|} v_{\mu}$ is a non-zero scalar multiple of $\prod f_{\alpha_{i,k}}^{t_{i}}v_{\mu}$. 

Now $f_{\alpha_{k}}^{a+1}v_{\mu} = 0$, because using Lemma \ref{calculationlem} implies that it is a highest-weight vector. Therefore for $|t| > a$, $f_{\alpha_{k}}^{|t|}v_{\mu} = 0$ in $L(\mu)$, which implies that $\prod f_{\alpha_{i,k}}^{t_{i}}v_{\mu} = 0$ if $|t| > a$.

Now we have $f_{\alpha_{i,j}}v_{\mu} = 0$ whenever $1 \leq i \leq j < k$. By applying the PBW Theorem, $L(\mu)$ is spanned by elements of the form $\prod f_{\alpha_{i,k}}^{t_{i}} \underset{1 \leq i \leq j < k}{\prod} f_{\alpha_{i,j}}^{s_{i,j}}v_{\mu}$ with $s_{i,j} \in \mathbb{N}_{0}$. The only non-zero elements of this form are those with $s_{i,j} = 0$ and $|t| \leq a$, which we have established are linearly independent. The result follows.

\end{proof}

\begin{proof}[Proof of Corollary \ref{inductioncor3}]
In this case, we have $\sum f_{\alpha_{1}}^{i}V_{\lambda} = V_{\lambda}$ since $\alpha_{1} \in \Delta_{\lambda}$. 

If $\alpha \in \Phi^{+} \setminus \Phi_{\Delta_{\lambda}}^{+}$ and $\alpha - \alpha_{1} \notin \Phi^{+} \setminus \Phi^{+}_{\Delta_{\lambda}}$, then $\alpha - \alpha_{1} \notin \Phi^{+}$ (since $\alpha$ has a non-zero component of some $\beta \in \Delta \setminus \Delta_{\lambda}$, which cannot be $\alpha_{1}$). Therefore $\alpha = \gamma_{1,2}$. So it is sufficient to prove $W_{\alpha_{1}} \cap V_{\lambda} = 0$, then the result follows from Proposition \ref{Walphadescription}.

Say $\Delta_{\lambda} = \Delta_{\lambda,1} \sqcup \Delta_{\lambda,2}$, where $\Delta_{\lambda,1}$ is the connected component of $\Delta_{\lambda}$ containing $\alpha_{1}$ (in the sense of Definition \ref{connected}). Then $\mathfrak{g}_{\Delta_{\lambda}} = \mathfrak{g}_{\Delta_{\lambda,1}} \oplus \mathfrak{g}_{\Delta_{\lambda,2}}$.

 Write $\lambda_{i} = \lambda |_{\mathfrak{h}_{\Delta_{\lambda,i}}}$ for $i=1,2$. We can then describe  $V_{\lambda}$ as a $\mathfrak{g}_{\Delta_{\lambda},K}$-module via $V_{\lambda} = L(\lambda_{1}) \otimes L(\lambda_{2})$, where each  $L(\lambda_{i})$ is a $U(\mathfrak{g}_{\Delta_{\lambda,i},K})$-module.
 
\bigskip

\noindent \emph{Case 1}: $k > 1$.

\noindent In this case, $L(\lambda_{1})$ has a $K$-basis $\overset{k}{\underset{i=1}{\prod}} f_{\alpha_{i,k}}^{t_{i}} v_{\lambda_{1}}$ for $|t| \leq \lambda(h_{\alpha_{k}})$, which all have distinct weights.

Now suppose $v \in V_{\lambda}$ is a non-zero weight vector. Then $v$ has the form $v = \prod f_{\alpha_{i,k}}^{t_{i}} v_{\lambda_{1}} \otimes w$ for some non-zero $w \in L(\lambda_{2})$. We calculate that $e_{\alpha_{1}}^{|t|} v = C f_{\alpha_{2,k}}^{t_{1}+t_{2}} \overset{k}{\underset{i=3}{\prod}}f_{\alpha_{i,k}}^{t_{i}} v_{\lambda_{1}} \otimes w$ with $C$ a non-zero scalar. This element is non-zero, so $v \notin W_{\alpha_{1}}$.

\bigskip

\noindent \emph{Case 2}: $k = 1$.

\noindent In this case, $L(\lambda_{1})$ has a $K$-basis $\prod f_{\alpha_{1}}^{i} v_{\lambda_{1}}$ for $i \leq \lambda(h_{\alpha_{1}})$. If $v \in V_{\lambda}$ is a non-zero weight vector, $v$ has the form $v = f_{\alpha_{1}}^{i}v_{\lambda_{1}} \otimes w$ for some non-zero $w \in L(\lambda_{2})$. Then $e_{\alpha_{1}}^{i} v$ is a non-zero scalar multiple of $v_{\lambda_{1}} \otimes w$, which is non-zero. So $w \notin W_{\alpha_{1}}$.

\end{proof}

\subsection{Highest-weight vectors in generalised Verma modules}
\label{hwvectors}
Motivated by the conditions of Theorem \ref{inductionprop} (specifically that $M_{\Delta_{\lambda}}(\lambda)$ does not have a highest-weight vector in the set $U(\mathfrak{n}_{2,K}) (W_{\alpha_{1}} \cap U(\mathfrak{n}_{1,K})V_{\lambda})$, we want to study highest-weight vectors in a generalised Verma module $M_{\Delta_{\lambda}}(\lambda)$. We will first work in a general setting, before restricting to the particular cases in type D that we want to deal with.

So for now we allow $\Phi$ to be any root system, and we fix a weight $\lambda$. Write $\mathfrak{n} = \mathfrak{n}_{\lambda}$ for the corresponding nilpotent subalgebra, so $\mathfrak{n} = \Of \{ f_{\alpha} \mid \alpha \in \Phi^{+} \setminus \Phi_{\Delta_{\lambda}}^{+} \}$. Write $\mathfrak{p} = \mathfrak{p}_{\lambda}$ for the corresponding parabolic subalgebra, so $\mathfrak{g} = \mathfrak{p} \oplus \mathfrak{n}$. 
So the generalised Verma module $M_{\Delta_{\lambda}}(\lambda)$ has the form $\ug \otimes_{\up}V_{\lambda}$ as before.

\bigskip

\noindent For convenience, write $M = M_{\Delta_{\lambda}}(\lambda)$.

\bigskip

Let $\Omega_{i} \un$ be the PBW filtration on $\un$, defined by $\Omega_{i} \un = K \{x_{1} \dots x_{j} \mid 0 \leq j \leq i, x_{1},\dots,x_{j} \in \mathfrak{n} \}$. Then by the PBW theorem, the associated graded algebra $gr \un = \bigoplus \Omega_{i+1} \un / \Omega_{i} \un$ is a polynomial ring generated by the images of the $f_{\alpha} \in \mathfrak{n}$ in $\Omega_{1} \un / \Omega_{0} \un$.

Let $\sigma_{i} : \Omega_{i} \un \rightarrow \Omega_{i} \un / \Omega_{i-1} \un $ be the projection map. Then write $x_{\alpha} = \sigma_{1}(f_{\alpha})$ for each $f_{\alpha} \in \mathfrak{n}$. So we can write $gr \un = K[\{x_{\alpha} \}]$.

Now consider the filtration $\Omega_{i} M = (\Omega_{i} \un) \cdot V_{\lambda}$, which makes $M$ a filtered $\un$-module. We can then describe the corresponding associated graded module of $M$ as $gr M = K[\{ x_{\alpha}  \}] \otimes V_{\lambda}$, a $K[ \{ x_{\alpha} \}]$-module. We can extend $\sigma_{i}$ to a map $\sigma_{i}: \Omega_{i} M \rightarrow gr M$ via $\sigma_{i}(xv) = \sigma_{i}(x)v$ for $x \in \Omega_{i} \un$, $v \in V_{\lambda}$.

\begin{mylem}
For any $i$, $\up \cdot \Omega_{i}M \subseteq \Omega_{i}M$.
\end{mylem}

\begin{proof}
By induction on $i$. Firstly, $\Omega_{0}M = V_{\lambda}$, so the result is true for $i=0$.

Now let $i>0$, and assume true for $i-1$. Let $x \in \mathfrak{p}_{K}$, and we aim to show $x \Omega_{i}M \subseteq \Omega_{i}M$.

Let $x_{1},\dots,x_{j} \in \mathfrak{n}_{K}, v \in V_{\lambda}$, where $j \leq i$. Then we have $x x_{1}\dots x_{j} v = [x,x_{1}] x_{2} \dots x_{j}v + x_{1}x x_{2} \dots x_{j} v$. Say $[x,x_{1}] = y_{1}+y_{2}$, where $y_{1} \in \mathfrak{p}_{K},y_{2} \in \mathfrak{n}_{K}$.

So $x x_{1}\dots x_{j} v = y_{1} x_{2} \dots x_{j}v + y_{2} x_{2} \dots x_{j}v+ x_{1}x x_{2} \dots x_{j} v$. By the induction hypothesis,  $y_{1} x_{2} \dots x_{j} v \in \Omega_{i-1}M \subseteq \Omega_{i}M $. Also $y_{2} x_{2} \dots x_{j}v \in \Omega_{i}M$ by definition of $\Omega_{i}M$. Finally, $x x_{2} \dots x_{j}v \in \Omega_{i-1}M$ by the induction hypothesis, so $x_{1} x x_{2} \dots x_{j} b \in \Omega_{i}M$ and the result follows.
\end{proof}

\noindent This means that for any $x \in \up$, $x$ induces a $K$-linear map $\Omega_{i}M / \Omega_{i-1}M \rightarrow \Omega_{i}M / \Omega_{i-1}M$ for each $i$, and hence induces a $K$-linear map $gr M \rightarrow gr M$. Write $\bar{x}$ for the map $gr M \rightarrow gr M$ induced by $x$.

Also write $\phi : \up \rightarrow \End_{K}V_{\lambda}$ for the map corresponding to the action of $\up$ on $V_{\lambda}$.

Let $A \subset \End_{K} K[ \{x_{\alpha} \}]$ be the Weyl algebra of $K[ \{x_{\alpha} \}]$, generated by the multiplication maps $\bar{x}_{\alpha}$ (given by multiplication by $x_{\alpha}$) and the partial derivatives $\partial_{\alpha} = \partial_{x_{\alpha}}$. Then the following result will show that $\bar{e}_{\beta} \in A \otimes \End_{K} V$ for each $\beta$.

\begin{mylem}
\label{eactionlem}
For any $\beta \in \Phi^{+}$, $\bar{e}_{\beta} = \sum c_{\alpha} \bar{x}_{\alpha-\beta} \partial_{\alpha} \otimes 1 + 1 \otimes \phi(e_{\beta})$, where the sum runs over all $\alpha \in \Phi^{+} \setminus \Phi_{\Delta_{\lambda}}^{+}$ such that $\alpha - \beta \in \Phi^{+} \setminus \Phi_{\Delta_{\lambda}}^{+}$, and the $c_{\alpha}$ are non-zero scalars.
\end{mylem}

\begin{proof}
Fix $\beta \in \Phi^{+}$. By the definition of a Chevalley basis, for each $\alpha$ such that $\alpha - \beta \in \Phi^{+} \setminus \Phi_{\Delta_{\lambda}}^{+}$, say $[e_{\beta},f_{\alpha}] = c_{\alpha} f_{\alpha - \beta}$ where $c_{\alpha}$ is a non-zero scalar. Also note that if $\alpha - \beta \notin \Phi^{+} \setminus \Phi_{\Delta_{\lambda}}^{+}$, then $[e_{\beta},f_{\alpha}] \in \mathfrak{p}$ (potentially equal to $0$), since all weight vectors in $\mathfrak{g}$ either lie in $\mathfrak{n}$ or $\mathfrak{p}$.

\noindent Define a derivation $D \in A$ by $D = \underset{\alpha}{\sum} c_{\alpha} \bar{x}_{\alpha-\beta} \partial_{\alpha}$.

\bigskip

\noindent \emph{Claim:} for any $\alpha_{1},\dots,\alpha_{r} \in \Phi^{+} \setminus \Phi_{\Delta_{\lambda}}^{+}$ and $v \in V_{\lambda}$, we have: \\ $\bar{e}_{\beta}(x_{\alpha_{1}} \dots x_{\alpha_{r}}\otimes v) = D(x_{\alpha_{1}}\dots x_{\alpha_{r}}) \otimes v + x_{\alpha_{1}}\dots x_{\alpha_{r}} \otimes e_{\beta}v$.

\bigskip

\noindent \emph{Proof of claim:} induction on $r$. The case $r = 0$ is trivial. So now let $r > 0$, and assume true for $r-1$. 

Then:
\begin{equation*}
\begin{split}
 \bar{e_{\beta}}(x_{\alpha_{1}} \dots x_{\alpha_{r}} \otimes v) &= \sigma_{r}(e_{\beta} f_{\alpha_{1}} \dots f_{\alpha_{r}} v) \\
 &= \sigma_{r}([e_{\beta},f_{\alpha_{1}}] f_{\alpha_{2}} \dots f_{\alpha_{r}} v + f_{\alpha_{1}} e_{\beta} f_{\alpha_{2}} \dots f_{\alpha_{r}} v) \\
 &= \sigma_{r}([e_{\beta},f_{\alpha_{1}}] f_{\alpha_{2}} \dots f_{\alpha_{r}} v) + x_{\alpha_{1}} \bar{e}_{\beta}(x_{\alpha_{2}} \dots x_{\alpha_{r}} \otimes v).
 \end{split}
 \end{equation*}
 
 By induction, $x_{\alpha_{1}} \bar{e}_{\beta}(x_{\alpha_{2}} \dots x_{\alpha_{r}} \otimes v) = x_{\alpha_{1}}D(x_{\alpha_{2}} \dots x_{\alpha_{r}}) \otimes v + x_{\alpha_{1}} \dots x_{\alpha_{r}} \otimes e_{\beta} v$.
 
If $\beta - \alpha_{1} \notin \Phi^{+} \setminus \Phi_{\Delta_{\lambda}}^{+}$, then $[e_{\beta},f_{\alpha_{1}}] \in \mathfrak{p}_{K}$. So $[e_{\beta},f_{\alpha_{1}}] f_{\alpha_{2}} \dots f_{\alpha_{r}} v \in \Omega_{i-1}M$, and thus $\sigma_{r}([e_{\beta},f_{\alpha_{1}}] f_{\alpha_{2}} \dots f_{\alpha_{r}} v) = 0$. But also in this case $D(x_{\alpha_{1}}) = 0$, so $\sigma_{r}([e_{\beta},f_{\alpha_{1}}] f_{\alpha_{2}} \dots f_{\alpha_{r}} v) = D(x_{\alpha_{1}}) x_{\alpha_{2}} \dots x_{\alpha_{r}} \otimes v$.

If $\beta - \alpha_{1} \in \Phi^{+} \setminus \Phi_{\Delta_{\lambda}}^{+}$, then $[e_{\beta},f_{\alpha_{1}}] = c_{\alpha_{1}}f_{\alpha_{1}-\beta}$, so $\sigma_{r}([e_{\beta},f_{\alpha_{1}}] f_{\alpha_{2}} \dots f_{\alpha_{r}} v) = c_{\alpha_{1}} x_{\alpha_{1} - \beta} x_{\alpha_{2}} \dots x_{\alpha_{r}} \otimes v = D(x_{\alpha_{1}}) x_{\alpha_{2}} \dots x_{\alpha_{r}} v$.

Therefore $\bar{e_{\beta}}(x_{\alpha_{1}} \dots x_{\alpha_{r}} \otimes v) = D(x_{\alpha_{1}}) x_{\alpha_{2}} \dots x_{\alpha_{r}} \otimes v + x_{\alpha_{1}} D(x_{\alpha_{2}} \dots x_{\alpha_{r}}) \otimes v + x_{\alpha_{1}} \dots x_{\alpha_{r}} \otimes e_{\beta} v = D(x_{\alpha_{1}} \dots x_{\alpha_{r}} \otimes v) + x_{\alpha_{1}} \dots x_{\alpha_{r}} \otimes v$, using the fact that $D$ is a derivation.

The claim now follows, and so does the lemma.

\end{proof}

\noindent We are now able to prove the following general result, which we will then apply to certain cases in type D.

\begin{myprop}
\label{hwvectorprop}
Suppose $\gamma_{1},\dots,\gamma_{r} \in \Phi^{+}\setminus \Phi_{\Delta_{\lambda}}^{+}$ and $\beta_{1},\dots,\beta_{r} \in \Phi^{+}$ satisfy the following conditions:

\begin{itemize}
\item $\gamma_{i} - \beta_{i} \in \Phi^{+} \setminus \Phi_{\Delta_{\lambda}}^{+}$ for all $i$, and $\gamma_{i} - \beta_{i} \notin \{ \gamma_{1},\dots,\gamma_{r} \}$,

\item $\gamma_{i} - \beta_{j} \notin \Phi^{+} \setminus \Phi_{\Delta_{\lambda}}^{+}$ whenever $i>j$.
\end{itemize}

\noindent Then $M$ has no non-zero highest-weight vector in $\un \{f_{\gamma_{1}},\dots,f_{\gamma_{r}} \} V_{\lambda} \subseteq M$.

\end{myprop}

\begin{proof}
Define a grading $A_{j}$ of the Weyl algebra $A$ by $A_{j} = K \{ x^{t} \partial^{t'} \mid \sum (t_{\gamma_{i}} -  t'_{\gamma_{i}}) = j \}$ (i.e. we take all $x_{\gamma_{i}}$ to have degree $1$ and all $\partial_{\gamma_{i}}$ to have degree $-1$). If $D \in A_{j}$, we will say that $D$ has degree $j$ in the $\gamma_{i}$. 

Then $K[\{x_{\alpha} \}]$ is a graded $A$-module via $K[\{x_{\alpha} \}]_{d} = K \{x^{t} \mid \sum t_{\gamma_{i}} = d \}$.

Now extend  the grading of $A$ to a grading of $A \otimes \End_{K}V_{\lambda}$ via $(A \otimes \End_{K}V_{\lambda})_{d} = A_{d} \otimes \End_{K}V_{\lambda}$. Thus $\text{gr}M = K[\{ x_{\alpha} \}] \otimes  V_{\lambda}$ is a graded $A \otimes \End_{K}V_{\lambda}$-module via $(K[\{ x_{\alpha} \}] \otimes  V_{\lambda})_{d} = K[\{ x_{\alpha} \}]_{d} \otimes V_{\lambda}$.

Also write $(A \otimes \End_{K}V_{\lambda})_{\geq d} = \underset{d' \geq d}{\bigoplus}(A \otimes \End_{K}V_{\lambda})_{d'}$ and similarly $(K[\{ x_{\alpha} \}] \otimes  V_{\lambda})_{\geq d} = \underset{d' \geq d}{\bigoplus} (K[\{ x_{\alpha} \}] \otimes  V_{\lambda})_{d'}$ for each $d'$.

Now fix some $1 \leq j \leq r$, and consider the map $\bar{e}_{\beta_{j}} : K[\{x_{\alpha} \}] \otimes V_{\lambda} \rightarrow K[ \{ x_{\alpha} \}] \otimes V_{\lambda}$ given by the action of $e_{\beta_{j}}$ on $M$. By Lemma \ref{eactionlem}, we can describe $\bar{e}_{\beta_{j}}$ as follows:

\begin{equation*}
\bar{e}_{\beta_{j}} = \sum c_{\alpha} \bar{x}_{\alpha-\beta_{j}} \partial_{\alpha} \otimes 1 + 1 \otimes \phi(e_{\beta_{j}}),
\end{equation*}
where the sum runs over all $\alpha \in \Phi^{+} \setminus \Phi_{\Delta_{\lambda}}^{+}$ such that $\alpha - \beta_{j} \in \Phi^{+} \setminus \Phi_{\Delta_{\lambda}}^{+}$.

Now for any $\alpha \notin \{ \gamma_{1},\dots,\gamma_{r} \}$, we can see that $\bar{x}_{\alpha - \beta_{j}} \partial_{\alpha} \in (A \otimes \End_{K}V_{\lambda})_{\geq 0}$. Moreover, if $i > j$ we have $\gamma_{i} - \beta_{j} \notin \Phi^{+} \setminus \Phi_{\Delta_{\lambda}}^{+}$ by assumption, so there is no $\gamma_{i}$-term in the sum. Therefore we can write:

\begin{equation}
\bar{e}_{\beta_{j}} \equiv \overset{j}{\underset{i=1}{\sum}} c_{j,i} \bar{x}_{\gamma_{i} - \beta_{j}} \partial_{\gamma_{i}} \otimes 1 \: \left( \text{mod }(A \otimes \End_{K}V_{\lambda})_{\geq 0} \right)
\end{equation}

\noindent where $c_{j,i} = 0$ if and only if $\gamma_{i} - \beta_{j} \notin \Phi^{+} \setminus \Phi_{\Delta_{\lambda}}^{+}$. By assumption $\gamma_{j} - \beta_{j} \in \Phi^{+} \setminus \Phi_{\Delta_{\lambda}}^{+}$ with $\gamma_{j} - \beta_{j} \notin \{ \gamma_{1}, \dots, \gamma_{r} \}$, so in fact $c_{j,j} x_{\gamma_{j} - \beta_{j}} \partial_{\gamma_{j}} \otimes 1$ is a non-zero term of degree $-1$ in the $\gamma_{i}$.

\bigskip

\noindent Now suppose for contradiction that $v \in M$ is a non-zero highest-weight vector in the set $\un \{f_{\gamma_{1}},\dots,f_{\gamma_{r}} \}V_{\lambda}$. Choose $k$ to be minimal such that $v \in \Omega_{k}M$, then $0 \neq \sigma_{k}(v) \in K[ \{x_{\alpha} \}] \otimes V_{\lambda} $ with $\bar{e_{\beta}}(v) = \sigma_{k}(e_{\beta}v) = 0$ for all $\beta \in \Phi^{+}$.

Moreover, $\sigma_{k}(v) \in K[\{x_{\alpha} \}]\{x_{\gamma_{1}},\dots,x_{\gamma_{r}} \} \otimes V_{\lambda}$ by choice of $v$. Thus $\sigma_{k}(v) \in (K[\{x_{\alpha} \}]\otimes V_{\lambda})_{\geq 1}$.

Let $d$ be minimal such that $\sigma_{k}(v)$ has a non-zero component of degree $d$ in the $\gamma_{i}$. Then $d \geq 1$. Write $\sigma_{k}(v)_{d}$ for the  component of degree $d$ in the $\gamma_{i}$. Then for each $1 \leq j \leq r$, using (2), we have:
\begin{equation}
0 = \bar{e}_{\beta_{j}}(\sigma_{k}(v)) \equiv \underset{i \leq j}{\sum} c_{j,i} (\bar{x}_{\gamma_{i}-\beta_{j}} \partial_{\gamma_{i}} \otimes 1)(\sigma_{k}(v)_{d}) \: \left( \text{mod } (K[\{x_{\alpha} \}]\otimes V_{\lambda})_{\geq d} \right)
\end{equation}

\noindent \emph{Claim:} for each $1 \leq j \leq r$, $(\partial_{\gamma_{j}} \otimes 1)(\sigma_{k}(v)_{d}) = 0$.

\noindent \emph{Proof of claim:} induction on $j$. Let $1 \leq j \leq r$, and assume true for all $i < j$. Then (3) shows that:

\begin{equation*}
0 = \bar{e}_{\beta_{j}}(\sigma_{k}(v)) \equiv  c_{j,j} (\bar{x}_{\gamma_{j}-\beta_{j}} \partial_{\gamma_{j}} \otimes 1)(\sigma_{k}(v)_{d}) \: \left( \text{mod } (K[\{x_{\alpha} \}]\otimes V_{\lambda})_{\geq d} \right)
\end{equation*}

\noindent As seen previously $c_{j,j} (x_{\gamma_{j}-\beta_{j}} \partial_{\gamma_{j}} \otimes 1)$ is non-zero of degree $-1$ in the $\gamma_{i}$. So by looking at degree $d-1$ terms, we see that $c_{j,j} (\bar{x}_{\gamma_{j}-\beta_{j}} \partial_{\gamma_{j}} \otimes 1)(\sigma_{k}(v)_{d}) = 0$, and hence $(\partial_{\gamma_{j}} \otimes 1)(\sigma_{k}(v_{d})) = 0$. The claim is now proved.

\bigskip

\noindent Finally, say $\sigma_{k}(v)_{d} = \sum f_{i} \otimes v_{i}$, where the $f_{i} \in K[\{x_{\alpha} \}]$ and the $v_{i} \in V_{\lambda}$ are linearly independent. From the claim, for each $j$ we have $\sum \partial_{\gamma_{j}}(f_{i}) \otimes v_{i} = 0$, and hence $\partial_{\gamma_{j}}(f_{i}) = 0$ for each $i$. Therefore the polynomials $f_{i}$ are constant in the $x_{\gamma_{j}}$, implying that the degree $d$ is zero, a contradiction.

\end{proof}

\noindent \textbf{We will now apply this in type D. Assume for the rest of the section that $\Phi$ has type $D_{m}$ where $m \geq 4$, and use the labelling of positive roots from \S \ref{Liealgsec}.}

\begin{mycor}
\label{hwcor1}
Suppose $\Phi$ has type $D_{m}$ with $m \geq 4$, and suppose $\lambda$ is not a weight that is not dominant integral. Then there is no highest-weight vector in $M_{\Delta_{\lambda}}(\lambda)$ in the set $\un f_{\gamma_{1,2}}V_{\lambda}$.

\end{mycor}

\begin{proof}
We apply Proposition \ref{hwvectorprop} with the roots $\gamma_{1} = \gamma_{1,2} $ and $\beta_{1} = \alpha_{2}$. 

If $m = 4$, then $\gamma_{1,2} - \alpha_{2} = \alpha_{1} + \dots + \alpha_{4} = \alpha_{1,4}$, and if $m > 4$ then $\gamma_{1,2} - \alpha_{2} = \gamma_{1,3}$.

In either case, $\alpha_{k} \in (\gamma_{1,2} - \alpha_{2})$ for all $1 \leq k \leq m$. Since $\lambda$ is not dominant integral, $\Delta_{\lambda} \neq \Delta$, so this means $(\gamma_{1,2} - \alpha_{2}) \in \Phi^{+} \setminus \Phi_{\Delta_{\lambda}}^{+}$.

The conditions of Proposition \ref{hwvectorprop} are met, so the result follows.
\end{proof}

\noindent We now return to the case of a scalar generalised Verma module.

\begin{mycor}
\label{hwcor2}
Suppose $\Phi$ has type $D_{m}$. Suppose $M_{\Delta_{\lambda}}(\lambda)$ is a scalar generalised Verma module, such that $\alpha_{1} \notin \Phi^{+}$ and also $\alpha_{k} \notin \Delta_{\lambda}$ for some $k > 1$.

Now fix $k>1$ to be minimal such that $\alpha_{k} \notin \Delta_{\lambda}$

Then $M_{\Delta_{\lambda}}(\lambda)$ has no highest-weight vector in the set $\un (\{f_{\alpha_{1,i}} \mid 2 \leq i \leq k-1 \} \cup \{f_{\gamma_{1,2}} \}) V_{\lambda}$.
\end{mycor}

\begin{proof}

Apply Proposition \ref{hwvectorprop} with the choice of roots $\gamma_{1} = \gamma_{1,2},\beta_{1} = \alpha_{1,2}$ and $\gamma_{i} = \alpha_{1,i}, \beta_{i} = \alpha_{2,i}$ for $2 \leq i \leq k$. We now check the conditions are met.

\noindent Firstly, $\gamma_{1}-\beta_{1} = 
\begin{cases}
\gamma_{2,3}, & m>4,\\
\alpha_{2,m}, & m=4.
\end{cases}
$

\noindent So $\gamma_{1} - \beta_{1}$ has a non-zero $\alpha_{k}$-component, and hence lies in $\Phi^{+} \setminus \Phi_{\Delta_{\lambda}}^{+}$. Meanwhile $\gamma_{i} - \beta_{i} = \alpha_{1}$ for $2 \leq i \leq k$, which lies in $\Phi^{+} \setminus \Phi_{\Delta_{\lambda}}^{+}$ also.

Now $\gamma_{2} - \beta_{1} = 0$, while $\gamma_{i} - \beta_{1} = \alpha_{3,i}\in \Phi_{\Delta_{\lambda}}^{+}$ for $3 \leq i \leq k$. So $\gamma_{i} - \beta_{1} \notin \Phi^{+} \setminus \Phi_{\Delta_{\lambda}}^{+}$ for each $i > 1$.

Finally, if $k \geq i > j \geq 2$, we have $\gamma_{i} - \beta_{j} = \alpha_{1}+\alpha_{j+1}+\alpha_{j+2}+\dots+\alpha_{i} \notin \Phi^{+}$. So the conditions for Proposition \ref{hwvectorprop} are met and the result follows.

\end{proof}

\noindent We now combine these results with our results from \S \ref{inductionsectiontypeD}.

\begin{mycor}
\label{inductioncase1}
Suppose $\lambda$ satisfies the following conditions:
\begin{itemize}
\item $M_{\Delta_{\lambda}}(\lambda)$ is scalar,

\item $\alpha_{k} \notin \Delta_{\lambda}$ for some $k > 1$.

\end{itemize}

\noindent If the weight $\lambda |_{\mathfrak{h}_{\Delta \setminus \alpha_{1},K}}$ for $\mathfrak{g}_{\Delta \setminus \alpha_{1}}$ satisfies condition (1), then $\lambda$ satisfies condition (1).

\end{mycor}

\begin{proof}
We consider separately the cases $\alpha_{1} \in \Delta_{\lambda}$ and $\alpha_{1} \notin \Delta_{\lambda}$. First, assume $\alpha_{1} \notin \Delta_{\lambda}$.

We wish to apply Corollary \ref{inductioncor1}. We can assume $\Delta_{\lambda} \neq \{ \alpha_{1}, \alpha_{m-1} \}$, as this case will follow from the $\Delta_{\lambda} = \{ \alpha_{1}, \alpha_{m} \}$ case by symmetry between $\alpha_{m-1},\alpha_{m}$.

Let $k > 1$ be minimal such that $\alpha_{k} \notin \Delta_{\lambda}$. Corollary \ref{inductioncor1} implies that $W_{\alpha_{1}} \cap (U(\mathfrak{n}_{1,K})V_{\lambda}) = U(\mathfrak{n}_{1,K})(\{f_{\gamma_{1,2}}\}\cup \{f_{\alpha_{1,j}} \mid 2 \leq j < k \}) V_{\lambda}$. Then by Corollary \ref{hwcor2}, there is no highest-weight vector in the set $\un \left( W_{\alpha_{1}} \cap (U(\mathfrak{n}_{1,K})V_{\lambda}) \right)$.

The result then follows from Theorem \ref{inductionprop}.

\bigskip

\noindent Now instead assume $\alpha_{1} \in \Delta_{\lambda}$. Then Corollary \ref{inductioncorscalar} implies that $W_{\alpha_{1}} \cap (U(\mathfrak{n}_{1,K})V_{\lambda}) = U(\mathfrak{n}_{1,K})f_{\gamma_{1,2}}V_{\lambda}$. Corollary \ref{hwcor1} implies that there is no highest-weight vector in the set $\un \left( W_{\alpha_{1}} \cap (U(\mathfrak{n}_{1,K})V_{\lambda}) \right)$.

The result then follows from Theorem \ref{inductionprop}.
\end{proof}

\begin{mycor}
\label{inductioncase2}
Suppose $\alpha_{1},\alpha_{2} \notin \Delta_{\lambda}$. If the weight $\lambda |_{\mathfrak{h}_{\Delta \setminus \alpha_{1},K}}$ for $\mathfrak{g}_{\Delta \setminus \alpha_{1}}$ satisfies condition (1), then $\lambda$ satisfies condition (1).

\end{mycor}

\begin{proof}
By Corollary \ref{inductioncor2}, $W_{\alpha_{1}} \cap (U(\mathfrak{n}_{1,K})V_{\lambda}) = U(\mathfrak{n}_{1,K})f_{\gamma_{1,2}}V_{\lambda}$. Then Corollary \ref{hwcor1} implies that there is no highest-weight vector in the subset \\ $\un \left( W_{\alpha_{1}} \cap (U(\mathfrak{n}_{1,K})V_{\lambda}) \right) \subseteq M_{\Delta_{\lambda}}(\lambda)$.

The result then follows from Theorem \ref{inductionprop}.
\end{proof}

\begin{mycor}
\label{inductioncase3}
Suppose $1 \leq k < m-2$ such that $\lambda(h_{\alpha_{i}}) = 0$ for $i < k$, $\lambda(h_{\alpha_{k}}) \in \mathbb{N}_{0}$, and $\lambda(h_{\alpha_{k+1}}) \notin \mathbb{N}_{0}$. 

If the weight $\lambda |_{\mathfrak{h}_{\Delta \setminus \alpha_{1},K}}$ for $\mathfrak{g}_{\Delta \setminus \alpha_{1}}$ satisfies condition (1), then $\lambda$ satisfies condition (1).
\end{mycor}

\begin{proof}
By Corollary \ref{inductioncor3}, $W_{\alpha_{1}} \cap (U(\mathfrak{n}_{1,K})V_{\lambda}) = U(\mathfrak{n}_{1,K})f_{\gamma_{1,2}}V_{\lambda}$. Then Corollary \ref{hwcor1} implies that there is no highest-weight vector in the subset \\ $\un \left( W_{\alpha_{1}} \cap (U(\mathfrak{n}_{1,K})V_{\lambda}) \right) \subseteq M_{\Delta_{\lambda}}(\lambda)$.

The result then follows from Theorem \ref{inductionprop}.
\end{proof}

\subsection{The case with an abelian nilpotent subalgebra}
\label{abeliansection}

The main result of this section is as follows.

\begin{myprop}
\label{abeliancase}
Suppose $\Phi$ is any indecomposable root system. Suppose $\alpha_{1} \in \Delta$ such that any $\beta \in \Phi^{+}$ has $\alpha_{1}$-component 0 or 1. Suppose $\lambda \in \mathfrak{h}_{K}^{*}$ with $\lambda(p^{n}\mathfrak{h}_{R}) \subseteq R$ such that $\lambda(h_{\alpha_{1}}) \notin \mathbb{N}_{0}$, and $\lambda(h_{\alpha}) = 0$ for $\alpha \in \Delta \setminus \alpha_{1}$.

Then $\lambda$ satisfies condition (2).
\end{myprop}

\noindent Fix $\Phi$ to be any indecomposable root system, and fix a choice of $\alpha_{1} \in \Delta$ such that all $\beta \in \Phi^{+}$ have $\alpha_{1}$-coefficient either $0$ or $1$. Let $\mathfrak{p} = \mathfrak{p}_{\Delta \setminus \alpha_{1}},\mathfrak{n} = \mathfrak{n}_{\Delta \setminus \alpha_{1}}$ be the corresponding parabolic and nilpotent subalgebras, so $\mathfrak{n} = \Of \{f_{\alpha} \mid \alpha_{1} \in \alpha \}$ and $\mathfrak{g} = \mathfrak{n} \oplus \mathfrak{p}$. Let $P, N$ be the subgroups of $G$ corresponding to $\mathfrak{p}, \mathfrak{n}$.

Also define $\mathfrak{p}^{-} = \mathfrak{n}^{-} \oplus \mathfrak{h} \oplus \Of \{ e_{\alpha} \mid \alpha_{1} \notin \alpha \}$, which is the transpose of $\mathfrak{p}$. Write $P^{-}$ for the $F$-uniform group corresponding to $\mathfrak{p}^{-}$.

We wish to apply \cite[Proposition 5.4]{metaplectic} with the abelian subgroup $N \subseteq P^{-}$. 

\begin{myrem}
If \cite[Proposition 5.4]{metaplectic} could be shown to be true for $K$ any completely discretely valued field extension of $F$ (not necessarily finite), then it would be possible to prove  Theorem \ref{maintheorem} without the assumption that $K / F$ is finite.
\end{myrem}

\noindent In order to apply \cite[Proposition 5.4]{metaplectic}, it will first be necessary to prove the following lemma.

\begin{mylem}
$\mathfrak{n}$ is an ideal in $\mathfrak{p}^{-}$.
\end{mylem}

\begin{proof}
Suppose $f_{\beta} \in \mathfrak{n}$, where $\beta \in \Phi^{+}$ with $\alpha_{1} \in \beta$. 
\begin{itemize}
\item If $h \in \mathfrak{h}$, $[h,f_{\beta}]$ is a scalar multiple of $f_{\beta}$. 

\item Suppose $\alpha \in \Phi^{+}$. If $\alpha + \beta \notin \Phi^{+}$, then $[f_{\alpha},f_{\beta}] = 0$. If $\alpha + \beta \in \Phi^{+}$, then $[f_{\alpha},f_{\beta}]$ is a scalar multiple of $f_{\alpha + \beta}$, which lies in $\mathfrak{n}$ since $\alpha_{1} \in \alpha + \beta$.

\item Suppose $\alpha \in \Phi^{+}$ such that $e_{\alpha} \in \mathfrak{p}^{-}$, that is, $\alpha_{1} \notin \alpha$. Then $[e_{\alpha},f_{\beta}] = 0$ if $\beta - \alpha \notin \Phi$. If $\beta - \alpha \in \Phi$, then $\beta - \alpha$ must have positive $\alpha_{1}$-coefficient, so $f_{\beta - \alpha} \in \mathfrak{n}$ and $[e_{\alpha},f_{\beta}]$ is a scalar multiple of $f_{\beta - \alpha}$.
\end{itemize}
\end{proof}

\noindent This in particular means $\mathfrak{n}$ can be considered as a $\mathfrak{p}^{-}$-representation, using the adjoint action of $\mathfrak{p}^{-}$.

\begin{mylem}
\label{nirreduciblelem}
$\mathbb{Q}_{p}\otimes_{\mathbb{Z}_{p}}\mathfrak{n}$ is an irreducible $\mathbb{Q}_{p} \otimes_{\mathbb{Z}_{p}} \mathfrak{p}^{-}$-representation (with the adjoint action).
\end{mylem}

\begin{proof}
Denote $\mathfrak{n}_{\mathbb{Q}_{p}} = \mathbb{Q}_{p}\otimes_{\mathbb{Z}_{p}}\mathfrak{n}$ and $\mathfrak{p}_{\mathbb{Q}_{p}}^{-} = \mathbb{Q}_{p} \otimes_{\mathbb{Z}_{p}} \mathfrak{p}^{-}$. Suppose $\mathfrak{n}' \subseteq \mathfrak{n}_{\mathbb{Q}_{p}}$ is an ideal in $\mathfrak{p}_{\mathbb{Q}_{p}}^{-}$. First, we show that $\mathfrak{n}'$ is spanned by elements of the form $f_{\beta}$. 

\noindent \emph{Claim: } if $x=\overset{r}{\underset{i=1}{\sum}} a_{i} f_{\beta_{i}} \in \mathfrak{n}'$ with $a_{i} \neq 0$ and the $\beta_{i}$ distinct, then all $f_{\beta_{i}} \in \mathfrak{n}'$.

\noindent Proof of claim: induction on $r$. The case $r=1$ is trivial. If $r > 1$, then choose $h \in \mathfrak{h}_{\mathbb{Q}_{p}}$ such that $\beta_{1}(h) \neq \beta_{2}(h)$. Then $(h - \beta_{1}(h))x \in \mathfrak{n}'$ has the form $\overset{r}{\underset{i=2}{\sum}} b_{i} f_{\beta_{i}}$, with $b_{2} \neq 0$. By induction, we have $f_{\beta_{2}} \in \mathfrak{n}'$. Then $\underset{i \neq 2}{\sum} a_{i} f_{\beta_{i}} \in \mathfrak{n}'$, so by induction all $f_{\beta_{i}} \in \mathfrak{n}'$.

\bigskip

\noindent \emph{Claim: } for $f_{\alpha} \in \mathfrak{n}$,  $f_{\alpha} \in \mathfrak{n}'$ if and only if $f_{\alpha_{1}} \in \mathfrak{n}'$.

\noindent Proof of claim: induction on the length of $\alpha$ (that is, the sum of the coefficients of the simple roots in $\alpha$). Trivial for length 1, so assume $\alpha$ has length greater than 1. Then there exists $\beta \in \Delta$ such that $\alpha - \beta \in \Phi^{+}$, by \cite[\S 10.2 Lemma A]{otherHumphreys}.

If $\beta = \alpha_{1}$, then $e_{\alpha - \alpha_{1}} \in \mathfrak{p}^{-}$ (since $\alpha$ has $\alpha_{1}$-coefficient 1, and so $\alpha_{1} \notin \alpha - \alpha_{1}$). Then $[e_{\alpha - \alpha_{1}}, f_{\alpha}]$ is a non-zero integer multiple of $f_{\alpha_{1}}$, and $[f_{\alpha - \alpha_{1}},f_{\alpha_{1}}]$ is a non-zero integer multiple of $f_{\alpha}$. The claim follows for $\alpha$.

If instead $\beta \neq \alpha_{1}$, then $e_{\beta} \in \mathfrak{p}^{-}$. Then $[e_{\beta},f_{\alpha}]$ is a non-zero integer multiple of $f_{\alpha - \beta}$, and $[f_{\beta},f_{\alpha - \beta}]$ is a non-zero integer multiple of $f_{\alpha}$. So $f_{\alpha} \in \mathfrak{n}'$ if and only if $f_{\alpha - \beta} \in \mathfrak{n}'$, and the claim follows by applying the induction hypothesis to $f_{\alpha - \beta}$.

\bigskip

\noindent We now know that $\mathfrak{n}'$ is spanned by elements $f_{\alpha}$, and that if any $f_{\alpha} \in \mathfrak{n}'$ then $\mathfrak{n}' = \mathfrak{n}_{\mathbb{Q}_{p}}$. The lemma follows.

\end{proof}

Fix a weight $\lambda \in \mathfrak{h}_{K}^{*}$ such that $\lambda(p^{n} \mathfrak{h}_{R}) \subseteq R$, with $\lambda(h_{\alpha_{1}}) \notin \mathbb{N}_{0}$ and $\lambda(h_{\alpha}) = 0 $ for $\alpha \in \Delta \setminus \alpha_{1}$.

\noindent Recall that condition (2) on $\lambda$ is the statement that the multiplication map $KN \otimes_{K} L(\lambda) \rightarrow \widehat{L(\lambda)}$ is injective. We are now interested in studying the $KN$-submodule $KN L(\lambda)$ of $\widehat{L(\lambda)}$, which is the image of the above multiplication map.

\begin{mylem}
\label{abeliansubmodlem}
\begin{itemize}
\item $KN L(\lambda)$ is a $KP^{-}$-submodule of $\widehat{L(\lambda)}$,
\item $\underset{v \text{ weight } < \lambda}{\sum} KNv$ is a $KP^{-}$-submodule of $\widehat{L(\lambda)}$ (here the sum runs over all weight vectors $v \in L(\lambda)$ with weight not equal to $\lambda$).

\end{itemize}
\end{mylem}

\begin{proof}
We decompose $\mathfrak{p}^{-} = \mathfrak{n} \oplus \mathfrak{l}$, where $\mathfrak{l}$ is the Levi subalgebra $\mathfrak{l} = \mathfrak{h} \oplus \Of \{e_{\alpha},f_{\alpha} \mid \alpha_{1} \notin \alpha \}$.

Now, $L(\lambda)$ is a quotient of $M_{\Delta \setminus \alpha_{1}}(\lambda)$ which is a locally finite $\up$-module by \cite[Proposition 2.3.1]{me}. Therefore $L(\lambda)$ is a locally finite $\up$-module, and hence a locally finite $\ul$-module since $\mathfrak{l} \subseteq \mathfrak{p}$. By \cite[Proposition 2.2.18]{me}, this means $L(\lambda)$  is a $\Ul$-module and hence a $KL$-module.

Let $v \in L(\lambda)$. By the local finiteness of $L(\lambda)$, say $KLv \subseteq L(\lambda)$ is spanned by $v_{1},\dots,v_{r} \in L(\lambda)$. Now using \cite[Corollary 2.5.5]{me}, we can see that any element of $KP^{-}$ is a convergent sum of elements of the form $xy$ with $x \in KN, y \in KL$. So $KP^{-}v = \sum KN v_{i} \subseteq KN L(\lambda)$. 

Thus if $x \in KP^{-}, y \in KN$, then $xyv \in KN L(\lambda)$. So $KN L(\lambda)$ is indeed a $KP^{-}$-module.

Now suppose $v \in L(\lambda)$ is a weight vector, with weight not equal to $\lambda$. Since $\lambda(h_{\alpha}) = 0$ for all $\alpha \in \Delta \setminus \alpha_{1}$, $M_{\Delta_{\lambda}}(\lambda)$ is a scalar generalised Verma module and so is generated over $\un$ by the highest-weight vector $w_{\lambda}$. Thus $L(\lambda)$ is generated over $\un$ by $v_{\lambda}$.

This means the weight of $v$ must be of the form $\lambda - \underset{\alpha \in \Delta}{\sum} c_{\alpha} \alpha$ with $c_{\alpha_{1}} \neq 0$. Then since every weight vector in $\mathfrak{l}$ has zero $\alpha_{1}$-coefficient in its weight, we can see that any element of $KL v$ cannot have a non-zero $\lambda$-weight coefficient. So again say $KLv \subseteq L(\lambda)$ has $K$-basis $v_{1},\dots,v_{r}$, where the $v_{i}$ all have $\lambda$-weight component $0$.

We can then see that $KP^{-}v = \sum KN v_{i} \subseteq \underset{v \text{ weight } < \lambda}{\sum} KNv$, and thus $\underset{v \text{ weight } < \lambda}{\sum} KNv$ is a $KP^{-}$-module as required.
\end{proof}

\noindent In \cite[\S 5.1]{metaplectic}, a filtration on Iwasawa algebras is used as follows. Choose a $\mathbb{Q}_{p}$-basis $x_{1},\dots,x_{r}$ for $\mathfrak{g}$. Then by \cite[Corollary 2.5.5]{me}, we can describe $KG$ as $\{ \underset{t \in \mathbb{N}_{0}^{r}}{\sum} c_{t} (\exp(p^{n+1}x_{1})-1)^{t_{1}} \dots ( \exp(p^{n+1}x_{r}) - 1)^{t_{r}} \mid c_{t} \in K \text{ bounded} \} \subseteq \Ug$.

Now write $v_{p}$ for the valuation of $K$ given by extending the $p$-adic valuation on $\mathbb{Q}_{p}$. We can then define a filtration of $KG$ using:

\begin{equation*}
w_{p}(\underset{t \in \mathbb{N}_{0}^{r}}{\sum} c_{t} (\exp(p^{n+1}x_{1})-1)^{t_{1}} \dots ( \exp(p^{n+1}x_{r}) - 1)^{t_{r}}) = \underset{t}{\inf} v_{p} (c_{t}).
\end{equation*}

\noindent This filtration takes values in $\frac{1}{e} \mathbb{Z}$, where $e$ is the ramification index of $K / \mathbb{Q}_{p}$. The filtration $w_{p}$ is then equivalent to the $\mathbb{Z}$-filtration $e \cdot w_{p}$.

We want to check compatibility between this filtration, and the one given by the restriction of our filtration $\Gamma$ from $\Ug$ to $KG$. To this end, let $\Gamma'_{i}KG = \{ x \in KG \mid e \cdot w_{p}(x) \geq i\}$ for each $i \in \mathbb{Z}$.

\begin{mylem}
For each $i \in \mathbb{Z}$, $\Gamma'_{i}KG \subseteq \Gamma_{i} KG = KG \cap \Gamma_{i}\Ug$.
\end{mylem}

\begin{proof}

Let $x \in \mathfrak{g}$. Then consider $\exp(p^{n+1}x) - 1 = \underset{j\geq 1}{\sum} \frac{p^{j}}{j!} (p^{n}x)^{j}$.

We can calculate that for $j \geq 0$, $v_{p}\left(\frac{p^{j}}{j!} \right) = j - v_{p}(j!) = j - \underset{k \geq 1}{\sum} \lfloor j / p^{k} \rfloor \geq j - \underset{k \geq 1}{\sum}  j / p^{k} = j(1 - \frac{1}{p-1}) $. So $\frac{p^{j}}{j!} \in R$ for each $j$, and $\frac{p^{j}}{j!} \rightarrow 0$ as $j \rightarrow \infty$. This implies that $\exp(p^{n+1}x) - 1 \in \widehat{U(\mathfrak{g})_{n,R}} = \Gamma_{0} \Ug$.

Now suppose $y = \underset{t \in \mathbb{N}_{0}^{r}}{\sum} c_{t} (\exp(p^{n+1}x_{1})-1)^{t_{1}} \dots ( \exp(p^{n+1}x_{r}) - 1)^{t_{r}} \in \Gamma'_{i}KG$. So $e v_{p}(c_{t}) \geq i$ for each $t$. 

We therefore have $c_{t} \in \pi^{i}R$ for each $t$. We also know $(\exp(p^{n+1}x_{1})-1)^{t_{1}} \dots ( \exp(p^{n+1}x_{r}) - 1)^{t_{r}} \in \Gamma_{0}\Ug$ for each $t$. It now follows that $y \in \Gamma_{i} \Ug$.

\end{proof}

\noindent This lemma implies that if $M$ is a $KG$-module with a $\Gamma$-filtration $\Gamma_{i}M$, we can also consider $\Gamma_{i}M$ as a $\Gamma'$-filtration since $\Gamma'_{i}KG \cdot \Gamma_{j}M \subseteq \Gamma_{i}KG \cdot \Gamma_{j}M \subseteq \Gamma_{i+j}M$ for each $i,j$. Since $\Gamma'$ is equivalent to the filtration $w_{p}$, it follows that any $KG$-module with a $\Gamma$-filtration can also be considered as a filtered module with respect to $w_{p}$.

\begin{proof}[Proof of Proposition \ref{abeliancase}]
Define $V = KN L(\lambda) / \underset{v \text{ weight } < \lambda}{\sum} KNv$. Then by Lemma \ref{abeliansubmodlem}, $V$ is a $KP^{-}$-module.

We can consider $KP^{-}$ as a filtered $K$-algebra by the restriction of $\Gamma$. Since $L(\lambda) \in \mathcal{O}_{n}(\mathfrak{g})$, we can choose a filtration $\Gamma$ on $L(\lambda)$, and extend this to a filtration on $\widehat{L(\lambda)}$. This then induces a filtration on $V$, say $F_{i}V$.

Now define $\End^{d}_{K}V = \{ \phi \in \End_{K}V \mid \phi(F_{i}V) \subseteq F_{i+d}V \text{ for all } i \}$, and define $T = \underset{i \in \mathbb{Z}}{\bigcup} \End^{d}_{K}V \subseteq \End_{K}V$. Then $T$ is a filtered $K$-algebra with filtration given by the $\End^{d}_{K}V$.

Then define $\Psi : KP^{-} \rightarrow T$ to be the map given by the action of $KP^{-}$ on $V$, which is a filtered $K$-algebra homomorphism with respect to the filtrations $\Gamma$ on $KP^{-}$ and $\End^{d}_{K}V$ on $T$. We then wish to show that $\Psi |_{KN}$ is injective, by using \cite[Proposition 5.4]{metaplectic}. We already know by \ref{nirreduciblelem} that $\mathfrak{n}_{\mathbb{Q}_{p}}$ is irreducible as a $\mathfrak{p}^{-}_{\mathbb{Q}_{p}}$-representation with the adjoint action. So now for this result to apply, we need to show that $\Psi(KN)$ is infinite-dimensional.

\bigskip

\noindent Define $\mathfrak{a} = \Of f_{\alpha_{1}}$, and $A \subseteq N$ to be the corresponding $F$-uniform group with $L_{A} = p^{n+1} \mathfrak{a}$. We will claim that $\Psi |_{KA}$ is injective, and hence $\Psi(KA) \subseteq \Psi(KN)$ is infinite-dimensional. 

The map $\widehat{U(\mathfrak{a})_{n,K}} \rightarrow \widehat{L(\lambda)}$ given by $a \mapsto a v_{\lambda}$ is injective. We can see this by using Lemma \ref{calculationlem} together with the fact that $\lambda(h_{\alpha_{1}}) \notin \mathbb{N}_{0}$ to see that all $f_{\alpha_{1}}^{i} v_{\lambda}$ are non-zero in $L(\lambda)$. Then they are linearly independent as their weights are different.

Now suppose $a \in KA$ such that $\Psi(a) = 0$. Then $a v_{\lambda} \in \underset{v \text{ wt} < \lambda}{\sum}KNv$. So say $a v_{\lambda} = \sum x_{i} y_{i} v_{\lambda}$, where $x_{i} \in KN$, the $y_{i} \in U(\mathfrak{n}^{-}_{K})$ are weight vectors with non-zero weight, and the sum is finite. 

Since $a v_{\lambda}$ has only weight components with weight in $\lambda - \mathbb{N}_{0}\alpha_{1}$, we can assume $y_{i} \in U(\mathfrak{a}_{K})$ and $x_{i} \in KA$. But by the proof of \cite[Theorem 5.3]{verma}, the multiplication map $KA \otimes_{K} U(\mathfrak{a}_{K}) \rightarrow \widehat{U(\mathfrak{a})_{n,K}}$ is injective, so $a = 0$.

This establishes that $\Psi |_{KA}$ is injective, so $\Psi(KN)$ is infinite-dimensional.

\bigskip

It now follows from \cite[Proposition 5.4]{metaplectic} that $\Psi |_{KN}$ is injective.

Suppose for contradiction that condition (2) does not hold for $\lambda$. By \cite[Proposition 4.2.13]{me}, we can find an expression $(x_{0} + \sum x_{i} y_{i})v_{\lambda} = 0$, where $x_{i} \in KN$, $x_{0} \neq 0$, the $y_{i} \in U(\mathfrak{n}_{K})$ are weight vectors with non-zero weight, and the sum is finite.

Now given any $y \in KN$, we have $(x_{0} + \sum x_{i} y_{i}) y v_{\lambda} = 0$ (using the fact that $\Un$ is commutative), so $x_{0} y v_{\lambda} = 0$ in $V$. Thus $x_{0} V = 0$. So $\Psi(x_{0}) = 0$. As $\Psi |_{KN}$ is injective, $x_{0} = 0$, a contradiction.

\end{proof}

\noindent We now slightly generalise Proposition \ref{abeliancase} by combining it with our results from \S \ref{transfunctorssection}. We continue to assume that $\alpha_{1} \in \Delta$ satisfies the condition that all $\beta \in \Phi^{+}$ have $\alpha_{1}$-component $0$ or $1$.

\begin{mycor}
\label{abelianintegral}
Suppose $\lambda$ is an integral weight such that $\lambda(h_{\alpha_{1}}) \in \mathbb{Z}_{< 0}$ and $\langle \lambda + \rho, \beta^{\lor} \rangle > 0$ for all $\beta \in \Phi^{+} \setminus \alpha_{1}$. Then $\lambda$ satisfies condition (2).
\end{mycor}

\begin{proof}
By Proposition \ref{abeliancase}, the weight $- \omega_{1}$ satisfies condition (2) (with $\omega_{1}$ the fundamental weight corresponding to $\alpha_{1}$). Then by Proposition \ref{abeliantranslationfunctorprop}, condition (2) holds for $\lambda$.
\end{proof}

\begin{mycor}
\label{abeliannonintegral}
Suppose $\Phi$ is a simply laced root system, and $\lambda \in \mathfrak{h}_{K}^{*}$ with $\lambda(p^{n}\mathfrak{h}_{R}) \subseteq R$ such that $\lambda(h_{\alpha_{1}}) \notin \mathbb{Z}$, $\lambda(h_{\alpha}) \in \mathbb{N}_{0}$ for all $\alpha \in \Delta \setminus \alpha_{1}$. Then $\lambda$ satisfies condition (2).
\end{mycor}

\begin{proof}
Suppose $\beta \in \Phi^{+}$, and say $\beta = \underset{\alpha \in \Delta}{\sum} c_{\alpha} \alpha$. By assumption, $c_{\alpha_{1}} \in \{ 0, 1 \}$.

By the fact $\Phi$ is simply laced, $\langle \lambda + \rho, \beta^{\lor} \rangle = \sum c_{\alpha} \langle \lambda + \rho, \alpha^{\lor} \rangle$. We know $\langle \lambda + \rho, \alpha^{\lor} \rangle \in \mathbb{Z}$ for all $\alpha \in \Delta \setminus \alpha_{1}$, and $\langle \lambda + \rho, \alpha_{1}^{\lor} \rangle \notin \mathbb{Z}$. So $\langle \lambda + \rho, \beta^{\lor} \rangle \notin \mathbb{Z}$ if $c_{\alpha} =  1$, and $\langle \lambda + \rho, \beta^{\lor} \rangle \in \mathbb{N}$ if $c_{\alpha} = 0$.

So we can describe the facet of $\lambda$ as $\Phi_{\lambda}^{+} = \Phi_{\Delta \setminus \alpha_{1}}^{+}, \Phi_{\lambda}^{0} = \Phi_{\lambda}^{-} = \emptyset$.

Now define $\mu \in \mathfrak{h}_{K}^{*}$ by $\mu(h_{\alpha_{1}}) = \lambda(h_{\alpha_{1}})$ and $\mu(h_{\alpha}) = 0$ for $\alpha \in \Delta \setminus \alpha_{1}$. Then $\mu$ is in the same facet as $\lambda$, and $\lambda - \mu$ is a dominant integral weight.

By Proposition \ref{abeliancase}, condition (2) holds for $\mu$. Then by Proposition \ref{translationfunctorsprop}, condition (2) holds for $\lambda$.
\end{proof}

\subsection{The case of type $A_{3}$}

Although our main concern in this paper is the case of type D, we will need to look at the case of type $A_{3}$ for the purposes of our proof by induction. For this section, assume $\Phi$ has type $A_{m}$ for some $m$. We use a labelling of the set $\Delta$ of simple roots as follows.

\begin{dynkinDiagram}[
edge length=1cm,
labels={1,2,m-1,m},
label macro/.code={\alpha_{\drlap{#1}}}
]A{}
\end{dynkinDiagram}

\noindent We can then describe the set of positive roots as follows: $\Phi^{+} = \{ \alpha_{i} + \dots + \alpha_{j} \mid 1 \leq i \leq j \leq m \}$. In particular, every $\alpha_{i}$ satisfies the condition that for all $\alpha \in \Phi^{+}$, $\alpha$ has $\alpha_{i}$-coefficient $0$ or $1$.

We note that \cite[Theorem 2]{me} proves the faithfulness of all infinite-dimensional $\widehat{L(\lambda)}$ for type $A_{m}$ as long as $p > m+1$, with an approach involving proving condition (1) for certain $\lambda$. However, there are some additional cases we will need to prove condition (1) for now.

As before, assume $\lambda \in \mathfrak{h}_{K}^{*}$ is a weight with $\lambda(p^{n}\mathfrak{h}_{R}) \subseteq R$, and $\mathfrak{n}, \mathfrak{p}$ are the corresponding nilpotent and parabolic subalgebras.

\begin{mylem}
\label{ainductionlem}
Suppose $\Phi$ has type $A_{m}$ where $m \geq 2$, and suppose $\lambda \in \mathfrak{h}_{K}^{*}$ with $\lambda(p^{n}\mathfrak{h}_{R}) \subseteq R$ such that $\alpha_{2} \notin \Delta_{\lambda}$. If $\lambda |_{\mathfrak{h}_{\Delta \setminus \alpha_{1},K}}$ satisfies condition (1), then so does $\lambda$.
\end{mylem}

\begin{proof}
Recall the subspace $W_{\alpha_{1}} \subseteq \widehat{M_{\Delta_{\lambda}}(\lambda)}$ and the subalgebra $\mathfrak{n}_{1}$ of $\mathfrak{n}$ from \S \ref{inductionsection}.

We aim to apply Theorem \ref{inductionprop}. We will use Proposition \ref{Walphadescription} to show that $W_{\alpha_{1}} \cap (U(\mathfrak{n}_{1,K})V_{\lambda}) = 0$.

By our description of the positive roots, if $\alpha \in \Phi^{+}$ with $\alpha_{1} \in \alpha$ and $\alpha_{1} \neq \alpha$, then $\alpha_{2} \in \alpha$. Therefore $\alpha - \alpha_{1} \notin \Phi_{\Delta_{\lambda}}^{+}$. Thus using Proposition \ref{Walphadescription}, if we can show that $W_{\alpha_{1}} \cap \sum f_{\alpha_{1}}^{i} V_{\lambda} = 0$, then it follows that $W_{\alpha_{1}} \cap (U(\mathfrak{n}_{1,K})V_{\lambda}) = 0$.

In the case $\alpha_{1} \notin \Delta_{\lambda}$, we can use a similar proof to Corollary \ref{inductioncor2}. In the case $\alpha_{1} \in \Delta_{\lambda}$, we can use a similar proof to Corollary \ref{inductioncor3} (in the $k=1$ case).

So now Proposition \ref{Walphadescription} implies $W_{\alpha_{1}} \cap (U(\mathfrak{n}_{1,K})V_{\lambda}) = 0$. The result now follows from Theorem \ref{inductionprop}.

\end{proof}

\begin{mylem}
\label{A2case}
Suppose $\Phi$ has type $A_{1}$ or $A_{2}$ and $\lambda \in \mathfrak{h}_{K}^{*}$ with $\lambda(p^{n}\mathfrak{h}_{R}) \subseteq R$ is not dominant integral. Then $\lambda$ satisfies condition (1).
\end{mylem}

\begin{proof}
First, suppose $\Phi$ has type $A_{1}$. So $\lambda(h_{\alpha_{1}}) \notin \mathbb{N}_{0}$ since $\lambda$ is not dominant integral. Then $M_{\Delta_{\lambda}}(\lambda) = M(\lambda)$ is irreducible by standard results about the representation theory of $\mathfrak{sl}_{2}$. So $\widehat{M_{\Delta_{\lambda}}(\lambda)} = \widehat{L(\lambda)}$, and condition (1) holds for $\lambda$.

Now suppose $\Phi$ has type $A_{2}$. Since $\lambda$ is not dominant integral, either $\alpha_{1} \notin \Delta_{\lambda}$ or $\alpha_{2} \notin \Delta_{\lambda}$. If $\alpha_{2} \notin \Delta_{\lambda}$, then the result follows from Lemma \ref{ainductionlem} together with the $A_{1}$ case. The case $\alpha_{1} \notin \Delta_{\lambda}$ then follows by symmetry between $\alpha_{1}, \alpha_{2}$.
\end{proof}

\noindent We can now state the main result of this section, in type $A_{3}$.

\begin{myprop}
\label{A3case}
Suppose $\Phi$ has type $A_{3}$ and $\lambda \in \mathfrak{h}_{K}^{*}$ with $\lambda(p^{n}\mathfrak{h}_{R}) \subseteq R$ satisfies the following conditions:

\begin{itemize}
\item $\lambda$ is not dominant integral,

\item $\lambda(h_{\alpha}) \notin \mathbb{Z}_{<-1}$ for all $\alpha \in \Delta$,

\item $\Delta \setminus \Phi_{\lambda} \neq \{ \alpha_{1},\alpha_{3} \}$ (i.e. $\Delta \setminus \Phi_{\lambda}$ is a connected subset of $\Delta$).
\end{itemize}

\noindent Then $\lambda$ satisfies condition (1).
\end{myprop}

\begin{proof}

In \cite[\S 4.3.1]{me} a set of weights for $\mathfrak{g}$ is defined:

\begin{equation*}
\begin{split}
T_{\mathfrak{g}} = \{ & \lambda \in \mathfrak{h}_{K}^{*} \mid \lambda(p^{n} \mathfrak{h}_{R}) \subseteq R, \Delta_{\lambda} \subset \Delta \text{ is proper }, \Phi_{\lambda} = \Phi_{I} \text{ for some } \\ & I \subseteq \Delta , \langle \lambda + \rho, \alpha^{\lor} \rangle \geq 0 \text{ for all } \alpha \in I \}.
\end{split}
\end{equation*}

\noindent In \cite[Proposition 4.3.15]{me}, condition (1) is proved for all weights in $T_{\mathfrak{g}}$ (in a more general setting where $\mathfrak{g}$ is a direct sum of components of the form $\mathfrak{sl}_{m+1}(\Of)$). 

So by our assumption that $\lambda(h_{\alpha}) \notin \mathbb{Z}_{< -1}$ for all $\alpha \in \Delta$, condition (1) will hold for $\lambda$ as long as $\Phi_{\lambda}$ is the subalgebra spanned by some subset of $\Delta$. We now proceed in a case-by-case fashion.

\begin{itemize}
\item If $\lambda$ is integral: $\Phi_{\lambda} = \Phi = \Phi_{\Delta}$, so condition (1) holds for $\lambda$ by \cite[Proposition 4.3.15]{me}.

\item If $\alpha_{2} \notin \Delta_{\lambda}$, then condition (1) holds for $\lambda$ by Lemma \ref{ainductionlem} together with Lemma \ref{A2case}.

\item If $\lambda$ is not integral and $\alpha_{2} \in \Delta_{\lambda}$: either $\alpha_{1} \notin \Phi_{\lambda}$ or $\alpha_{3} \notin \Phi_{\lambda}$. By symmetry, we can assume without loss of generality that $\alpha_{1} \notin \Phi_{\lambda}$. Then by assumption $\alpha_{3} \in \Phi_{\lambda}$.

Now if $\alpha = \alpha_{i}+\dots+\alpha_{j} \in \Phi^{+}$ with $1 \leq i \leq j \leq 3$, then $\langle \lambda + \rho, \alpha^{\lor} \rangle = \overset{j}{\underset{k=i}{\sum}} \langle \lambda + \rho, \alpha_{k}^{\lor} \rangle$ since $A_{3}$ is a simply laced root system. Thus $\langle \lambda + \rho, \alpha^{\lor} \rangle \in \mathbb{Z}$ if and only if $i > 1$: that is, $\Phi_{\lambda} = \Phi_{ \{\alpha_{2},\alpha_{3} \}}$. Condition (1) for $\lambda$ now holds by \cite[Proposition 4.3.15]{me}.

\end{itemize}

\end{proof}

\subsection{Proof of main theorems}

In this section we take $\Phi$ to be type $D_{m}$, and use the labelling of simple roots from \S \ref{Liealgsec}. We write $W$ for the Weyl group of $\Phi$.

We will assume that $p \geq 5$, which in this case implies that $p$ does not divide the determinant of the Cartan matrix of any root subsystem of $\Phi$. Thus Theorem \ref{condition1application} shows us that if $\lambda \in \mathfrak{h}_{K}^{*}$ with $\lambda(p^{n} \mathfrak{h}_{R}) \subseteq R$ satisfies condition (1) (or condition (2)), then $\widehat{L(\lambda)}$ is faithful over $KG$.

Also in this section we take type $D_{3}$ to be equal to type $A_{3}$, with the following ordering of simple roots:

\begin{dynkinDiagram}[
edge length=1cm,
labels={2,1,3},
label directions={,,,right,,},
label macro/.code={\alpha_{\drlap{#1}}}
]A3
\end{dynkinDiagram}

\noindent as this will be helpful for our induction.

We will break down Theorem \ref{maintheoremvariant} into two main cases: the case where the highest weight is regular integral, and the case where it is not. 

Recall our definition of a connected subset of $\Delta$ from Definition \ref{connected}.

\begin{mydef}
If $\mathfrak{g}$ has type $D_{m}$ with $m \geq 3$ and we use our usual labelling of simple roots, denote $S_{\mathfrak{g}}$ for the set of weights $\lambda \in \mathfrak{h}_{K}^{*}$ with $\lambda(p^{n}\mathfrak{h}_{R}) \subseteq R$ such that:

\begin{itemize}
\item $\lambda$ is not dominant integral,

\item $\Delta \setminus \Phi_{\lambda}$ is either empty, or a connected subset of $\Delta$ containing $\alpha_{1}$,

\item $\langle \lambda + \rho, \alpha^{\lor} \rangle \notin \mathbb{Z}_{<0}$ for all $\alpha \in \Delta$. 
\end{itemize}
\end{mydef}

\noindent The following proposition together with Proposition \ref{simplerefsprop} then shows that in order to prove Theorem \ref{maintheoremvariant} for all $\lambda$ not regular integral, it is sufficient to prove it for all weights in $S_{\mathfrak{g}}$.

\begin{myprop}
\label{reducetoSg}
Suppose $\Phi$ has type $D_{m}$ for $m \geq 3$, and suppose $\lambda \in \mathfrak{h}_{K}^{*}$ with $\lambda(p^{n}\mathfrak{h}_{R}) \subseteq R$ such that $\lambda$  is not a regular integral weight.

Then there exist $\beta_{1},\dots,\beta_{r} \in \Delta$ such that $(s_{\beta_{i-1}} \dots s_{\beta_{1}} \cdot \lambda)(h_{\beta_{i}}) \notin \mathbb{N}_{0}$ for all $1 \leq i \leq r$, and $s_{\beta_{r}} \dots s_{\beta_{1}} \cdot \lambda \in S_{\mathfrak{g}}$.
\end{myprop}

\begin{proof}
The condition that $\lambda$ is not regular integral implies that $w \cdot \lambda$ is not dominant integral for all $w \in W$.

First, if $\lambda$ is not integral (or equivalently $\Delta \setminus \Phi_{\lambda}$ is non-empty) then Corollary \ref{simplerefsapplication} implies we can choose $\beta_{1},\dots,\beta_{r}$ such that $(s_{\beta_{i-1}} \dots s_{\beta_{1}} \cdot \lambda)(h_{\beta_{i}}) \notin \mathbb{N}_{0}$ for all $1 \leq i \leq r$, and $\Delta \setminus \Phi_{s_{\beta_{r}} \dots s_{\beta_{1}} \cdot \lambda}$ is a connected subset of $\Delta$ containing $\alpha_{1}$. 

So it is sufficient to prove the result for all $\lambda$ such that $\Delta \setminus \Phi_{\lambda}$ is either empty or a connected subset of $\Delta$ containing $\alpha_{1}$. Assume now $\lambda$ satisfies this condition.

Then choose $\beta_{1},\beta_{2}, \dots \in \Delta$ inductively as follows. If $(s_{\beta_{i}} \dots s_{\beta_{1}} \cdot \lambda)(h_{\alpha}) \notin \mathbb{Z}_{<-1}$ for all $\alpha \in \Delta$, stop. Otherwise, choose $\beta_{i+1} \in \Delta$ with $(s_{\beta_{i}} \dots s_{\beta_{1}} \cdot \lambda)(h_{\beta_{i+1}}) \in \mathbb{Z}_{<-1}$.

Then $\beta_{i+1} \in \Phi_{s_{\beta_{i}} \dots s_{\beta_{1}} \cdot \lambda}$, so $\Phi_{s_{\beta_{i+1}} \dots s_{\beta_{1}} \cdot \lambda} = s_{\beta_{i+1}}(\Phi_{s_{\beta_{i}} \dots s_{\beta_{1}} \cdot \lambda}) = \Phi_{s_{\beta_{i}} \dots s_{\beta_{1}} \cdot \lambda}$. Thus $\Phi_{s_{\beta_{i+1}} \dots s_{\beta_{1}} \cdot \lambda} = \Phi_{\lambda}$.

Moreover, $s_{\beta_{i+1}} \dots s_{\beta_{1}} \cdot \lambda > s_{\beta_{i}} \dots s_{\beta_{1}} \cdot \lambda$ in the partial ordering on weights. Since the orbit of $\lambda$ under the dot action of $W$ is finite, this process must terminate.

So say we obtain the finite sequence $\beta_{1}, \dots, \beta_{r}$. Then $(s_{\beta_{r}} \dots s_{\beta_{1}} \cdot \lambda)(h_{\alpha}) \notin \mathbb{Z}_{<-1}$ for all $\alpha \in \Delta$, and $\Phi_{s_{\beta_{r}} \dots s_{\beta_{1}} \cdot \lambda} = \Phi_{\lambda}$. So $s_{\beta_{r}} \dots s_{\beta_{1}} \cdot \lambda \in S_{\mathfrak{g}}$.
\end{proof}

\begin{myprop}
\label{Sgprop}
Suppose $\Phi$ has type $D_{m}$ with $m \geq 3$ and $\lambda \in S_{\mathfrak{g}}$. Then $\lambda$ satisfies condition (1).
\end{myprop}

\noindent We will proceed case-by-case. The general approach will be by induction on $m$: note that $\Delta \setminus \alpha_{1}$ spans a root system of type $D_{m-1}$ (with the labelling of the simple roots being shifted by 1 compared with our usual labelling). We can see that if $\lambda \in S_{\mathfrak{g}}$ such that $\lambda |_{\mathfrak{h}_{\Delta \setminus \alpha_{1},K}}$ is not dominant integral, then $\lambda |_{\mathfrak{h}_{\Delta \setminus \alpha_{1},K}} \in S_{\mathfrak{g}_{\Delta \setminus \alpha_{1}}}$.

\begin{mylem}
Suppose $\Phi$ has type $D_{m}$ with $m \geq 3$ and $\lambda \in S_{\mathfrak{g}}$ such that $\lambda |_{\mathfrak{h}_{\Delta \setminus \alpha_{1},K}}$ is dominant integral. Then $\lambda$ satisfies condition (1).
\end{mylem}

\begin{proof}
In this case, we must have $\alpha_{1} \notin \Delta_{\lambda}$ (since $\lambda$ is not dominant integral by the definition of $S_{\mathfrak{g}}$). So $\Delta_{\lambda} = \Delta \setminus \alpha_{1}$.

If $\lambda(h_{\alpha_{1}}) \notin \mathbb{Z}$: then Corollary \ref{abeliannonintegral} implies that $\lambda$ satisfies condition (2) and hence also condition (1).

If $\lambda(h_{\alpha_{1}}) \in \mathbb{Z}$: by definition of $S_{\mathfrak{g}}$, we must have $\lambda(h_{\alpha_{1}}) = -1$. Then $\langle \lambda + \rho, \alpha^{\lor} \rangle > 0$ for all $\alpha \in \Phi^{+} \setminus \alpha_{1}$, and Corollary \ref{abelianintegral} implies that $\lambda$ satisfies condition (2) and hence also condition (1).
\end{proof}

\begin{mylem}
\label{Sgscalar}
Suppose $\Phi$ has type $D_{m}$ with $m \geq 3$ and $\lambda \in S_{\mathfrak{g}}$ such that $M_{\Delta_{\lambda}}(\lambda)$ is a scalar generalised Verma module. Then $\lambda$ satisfies condition (1).
\end{mylem}

\begin{proof}
By induction on $m$. The base case $m = 3$ follows from Proposition \ref{A3case}. So now assume $m > 3$, and assume true for $m-1$.

We have seen that if $\lambda' = \lambda |_{\mathfrak{h}_{\Delta \setminus \alpha_{1},K}} $ is dominant integral, then $\lambda$ satisfies condition (1). So assume not.

Then $\lambda' \in S_{\mathfrak{g}_{\Delta \setminus \alpha_{1}}}$, and also $\lambda'(h_{\alpha}) = 0$ whenever $\alpha \in \Delta \setminus \alpha_{1}$ with $\lambda'(h_{\alpha}) \in \mathbb{N}_{0}$ (i.e. $\lambda'$ also satisfies the scalar condition). So by the induction hypothesis, $\lambda'$ satisfies condition (1). By Corollary \ref{inductioncase1}, $\lambda$ satisfies condition (1).
\end{proof}

\noindent The approach now will be to show that if $|\Delta \setminus \Phi_{\lambda}| \leq 1$ we can use translation functors to reduce to the scalar case, and otherwise we can apply Corollary \ref{inductioncase2} to reduce to the $m-1$ case and use induction.

\begin{myprop}
\label{Sgcase1}
Suppose $\Phi$ has type $D_{m}$ with $m \geq 3$ and $\lambda \in S_{\mathfrak{g}}$ such that $|\Delta \setminus \Phi_{\lambda}| \leq 1$. Then $\lambda$ satisfies condition (1).
\end{myprop}

\begin{proof}

First, note that the $m = 3 $ case follows from Proposition \ref{A3case}, so assume $m \geq 4$.

\bigskip

\noindent  If $|\Delta \setminus \Phi_{\lambda}| = 1$: this means $\langle \lambda + \rho, \alpha_{1}^{\lor} \rangle \notin \mathbb{Z}$, and $\langle \lambda + \rho, \alpha_{i}^{\lor} \rangle \in \mathbb{N}_{0}$ for $i > 1$. 

Now suppose $\alpha \in \Phi^{+}$, and say $\alpha = \sum c_{i} \alpha_{i}$. We know that $c_{i} \in \{0,1 \}$. Then since $\Phi$ is simply laced, we have $\langle \lambda + \rho, \alpha^{\lor} \rangle = \sum c_{i} \langle \lambda + \rho, \alpha_{i}^{\lor} \rangle$. We can see that if $c_{1} = 1$, then $\langle \lambda + \rho, \alpha^{\lor} \rangle \notin \mathbb{Z}$.

Let $I = \Delta \setminus ( \{ \alpha_{1} \} \cup \Delta_{\lambda})$. So $\langle \lambda + \rho, \alpha_{i}^{\lor} \rangle = 0$ whenever $\alpha_{i} \in I$. We can calculate the facet of $\lambda$ as $\Phi_{\lambda}^{+} = \Phi^{+}_{\Delta \setminus \alpha_{1}} \setminus \Phi_{I}^{+}$, $\Phi_{\lambda}^{0} = \Phi_{I}$, $\Phi_{\lambda}^{-} = \emptyset$.

Define $\mu \in \mathfrak{h}_{K}^{*}$ by $\mu(h_{\alpha_{i}}) = 0$ if $\alpha_{i} \in \Delta_{\lambda}$, and $\mu(h_{\alpha_{i}}) = \lambda(h_{\alpha_{i}}) $ otherwise. Then $\lambda - \mu$ is a dominant integral weight, $\Phi_{\lambda} = \Phi_{\mu}$,
 and $\lambda, \mu$ are in the same facet.
 
 Now $\mu \in S_{\mathfrak{g}}$ and $M_{\Delta_{\mu}}(\mu)$ is scalar. So $\mu$ satisfies condition (1) by Lemma \ref{Sgscalar}. Then $\lambda$ satisfies condition (1) by Proposition \ref{translationfunctorsprop}.
 
 \bigskip
 
\noindent If $|\Delta \setminus \Phi_{\lambda}| = 0$: this means $\lambda$ is integral. We take a similar approach to the $|\Delta \setminus \Phi_{\lambda}| = 1$ case. By definition of $S_{\mathfrak{g}}$ we have $\langle \lambda + \rho, \alpha_{i}^{\lor} \rangle = 0$ whenever $\alpha_{i} \notin \Delta_{\lambda}$. We can then describe the facet of $\lambda$ by $\Phi_{\lambda}^{+} = \Phi^{+} \setminus \Phi_{\Delta \setminus \Delta_{\lambda}}^{+}$, $\Phi_{\lambda}^{0} = \Phi_{\Delta \setminus \Delta_{\lambda}}^{+}$, and $\Phi_{\lambda}^{-} = 0$.

Define $\mu \in \mathfrak{h}_{K}^{*}$ by $\mu(h_{\alpha_{i}}) = 0$ if $\alpha_{i} \in \Delta_{\lambda}$, and $\mu(h_{\alpha_{i}}) = \lambda(h_{\alpha_{i}}) $ otherwise. Then $\lambda - \mu$ is a dominant integral weight, $\mu$ is integral, and $\lambda, \mu$ are in the same facet.

Again, $\mu \in S_{\mathfrak{g}}$ and $M_{\Delta_{\mu}}(\mu)$ is scalar. So $\mu$ satisfies condition (1) by Lemma \ref{Sgscalar}. Then $\lambda$ satisfies condition (1) by Proposition \ref{translationfunctorsprop}.

\end{proof}

\noindent We can now finish the proof of Proposition \ref{Sgprop}.

\begin{proof}[Proof of Proposition \ref{Sgprop}]
By induction on $m$. The case $m = 3$ follows from Proposition \ref{A3case}.

So now let $m \geq 4$, and assume true for $m-1$. If $|\Delta \setminus \Phi_{\lambda} | \leq 1$, then true by Proposition \ref{Sgcase1}. So assume $|\Delta \setminus \Phi_{\lambda} | \geq 2$.

By definition of $S_{\mathfrak{g}}$ we then have $\alpha_{1},\alpha_{2} \notin \Delta_{\lambda}$. Note that $\lambda |_{\mathfrak{h}_{\Delta \setminus \alpha_{1},K}} \in S_{\mathfrak{g}_{\Delta \setminus \alpha_{1}}}$, so by the induction hypothesis $\lambda |_{\mathfrak{h}_{\Delta \setminus \alpha_{1},K}}$ satisfies condition (1). Then by Corollary \ref{inductioncase1}, $\lambda$ satisfies condition (1).
\end{proof}

\begin{mycor}
\label{nonregintcor}
Suppose $\Phi$ has type $D_{m}$ with $m \geq 3$, and $\lambda \in \mathfrak{h}_{K}^{*}$ such that $\lambda(p^{n} \mathfrak{h}_{R}) \subseteq R$. If $\lambda$ is not a regular integral weight, then $\widehat{L(\lambda)}$ is faithful as a $KG$-module.
\end{mycor}

\begin{proof}
By Proposition \ref{reducetoSg}, there exist $\beta_{1},\dots,\beta_{r} \in \Delta$ such that $s_{\beta_{r}} \dots s_{\beta_{1}} \cdot \lambda \in S_{\mathfrak{g}}$, and such that $(s_{\beta_{i-1}} \dots s_{\beta_{1}} \cdot \lambda)(h_{\beta_{i}}) \notin \mathbb{N}_{0}$ for each $1 \leq i \leq r$. Write $\lambda' = s_{\beta_{r}} \dots s_{\beta_{1}} \cdot \lambda$.

By repeated application of Proposition \ref{simplerefsprop}, we see that $\Ann_{\Ug} \widehat{L(\lambda)} \subseteq \Ann_{\Ug} \widehat{L(\lambda')}$. By Proposition \ref{Sgprop}, $\lambda'$ satisfies condition (1). This means $\widehat{L(\lambda')}$ is faithful over $KG$ (as seen in \S \ref{outline}), which implies $\widehat{L(\lambda)}$ is faithful over $KG$.
\end{proof}

\noindent So now we need to consider the case of regular integral weights.

\begin{mylem}
\label{integralreductionlem}
If $\Phi$ is any root system and $\lambda \in \mathfrak{h}_{K}^{*}$ is a regular integral weight that is not dominant integral, then there exist $\beta_{1},\dots,\beta_{r} \in \Delta$ such that $(s_{\beta_{i-1}} \dots s_{\beta_{1}} \cdot \lambda)(h_{\beta_{i}}) < -1$ for all $i$, and $ s_{\beta_{r}} \dots s_{\beta_{1}} \cdot \lambda$ is dominant integral.

In particular, there exists a weight $\lambda'$ of the form $\lambda' = s_{\alpha} \cdot \mu$ with $\alpha \in \Delta$, $\mu$ dominant integral, such that $\Ann_{\Ug} \widehat{L(\lambda)} \subseteq \Ann_{\Ug} \widehat{L(\lambda')}$.
\end{mylem}

\begin{proof}
Choose $\beta_{1},\beta_{2},\dots$ inductively as follows. Assume we have already chosen $\beta_{1},\dots,\beta_{i-1}$, where $i \geq 1$.

If $s_{\beta_{i-1}} \dots s_{\beta_{1}} \cdot \lambda$ is dominant integral, stop. Otherwise, choose $\beta_{i} \in \Delta$ such that $(s_{\beta_{i-1}} \dots s_{\beta_{1}} \cdot \lambda)(h_{\beta_{i}}) < 0$, so then $(s_{\beta_{i-1}} \dots s_{\beta_{1}} \cdot \lambda)(h_{\beta_{i}}) < -1$ since $\lambda$ is regular.

Then $s_{\beta_{i}} \dots s_{\beta_{1}} \cdot \lambda > s_{\beta_{i-1}} \dots s_{\beta_{1}} \cdot \lambda$ in the partial ordering on weights. As the orbit of $\lambda$ under the dot action of $W$ is finite, this process must terminate. This proves the first part.

Now if we have $\beta_{1},\dots,\beta_{r}$ as above such that $s_{\beta_{r}} \dots s_{\beta_{1}} \cdot \lambda$ is dominant integral: take $\lambda' = s_{\beta_{r-1}} \dots s_{\beta_{1}} \cdot \lambda$. Then $\Ann_{\Ug} \widehat{L(\lambda)} \subseteq \Ann_{\Ug} \widehat{L(\lambda')}$ by repeated application of Proposition \ref{simplerefsprop}, and $s_{\beta_{r}} \cdot \lambda'$ is dominant integral.
\end{proof}

\noindent In light of this Lemma, we only need to consider regular integral weights of the form $s_{\alpha_{i}} \cdot \mu$ with $\mu$ dominant integral.

\begin{myprop}
\label{regintprop}
Suppose $\Phi$ has type $D_{m}$ with $m \geq 3$, and suppose $\lambda = s_{\alpha_{i}} \cdot \mu$ where $1 \leq i \leq m$ and $\mu$ is a dominant integral weight. Then $\lambda$ satisfies condition (1).
\end{myprop}

\begin{proof}

By induction on $m$. First, consider the base case $m = 3$. In this case, any simple root $\alpha_{i} \in \Delta$ has the condition that all $\alpha \in \Phi^{+}$ have $\alpha_{i}$-coefficient $0$ or $1$. So Corollary \ref{abelianintegral} implies that $\lambda$ satisfies condition (2), and hence condition (1). So now let $m \geq 4$, and assume true for $m-1$.

If $i \in \{ 1, m-1,m \}$ then all $\alpha \in \Phi^{+}$ have $\alpha_{i}$-coefficient $0$ or $1$ (by inspection of the positive roots). So in these cases, Corollary \ref{abelianintegral} implies that $\lambda$ satisfies condition (2) and hence condition (1).

So now assume $2 \leq i \leq m-2$.

\bigskip

\noindent Define $\mu' \in \mathfrak{h}_{K}^{*}$ by $\mu'(h_{\alpha_{i}}) = \mu(h_{\alpha_{i}}), \mu'(h_{\alpha_{j}}) = 0$ for $j \neq i$, so $\mu'$ is a dominant integral weight. Moreover $\mu - \mu'$ is dominant integral. Define $\lambda' = s_{\alpha_{i}} \cdot \mu'$. Then $\lambda, \lambda'$ are in the same facet.

Moreover, $\lambda - \lambda' = s_{\alpha_{i}} \cdot \mu - s_{\alpha_{i}} \cdot \mu' = s_{\alpha_{i}}(\mu -\mu') = \mu - \mu'$ since $\mu(h_{\alpha_{i}}) = \mu'(h_{\alpha_{i}})$, and so $\lambda - \lambda'$ is dominant integral.

Write $a = \mu(h_{\alpha_{i}}) \in \mathbb{N}_{0}$, then we can write $\mu' = a \omega_{i}$ where $\omega_{i}$ is the fundamental weight corresponding to $\alpha_{i}$. We can then see that $\lambda' = a \omega_{i} - (a+1) \alpha_{i}$, and calculate that $\lambda'(h_{\alpha_{j}}) =
\begin{cases}
0, &  j < i-1, \\
a+1, &  j = i-1, \\
-a-2, &  j = i.
\end{cases}
$

\noindent In particular, $\lambda'$ meets the conditions for Corollary \ref{inductioncase3} (with $k = i-1$). Now $\lambda'|_{\mathfrak{h}_{\Delta \setminus \alpha_{1},K}} = s_{\alpha_{i}} \cdot \mu'|_{\mathfrak{h}_{\Delta \setminus \alpha_{1},K}}$, so the induction hypothesis implies that $\lambda'|_{\mathfrak{h}_{\Delta \setminus \alpha_{1},K}}$ satisfies condition (1). Then Corollary \ref{inductioncase3} implies that $\lambda'$ satisfies condition (1).

Finally, Proposition \ref{translationfunctorsprop} implies that $\lambda$ satisfies condition (1).
\end{proof}

\begin{mycor}
\label{regintcor}
If $\Phi$ has type $D_{m}$ with $m \geq 3$ and $\lambda \in \mathfrak{h}_{K}^{*}$ is a regular integral weight that is not dominant integral, then $\widehat{L(\lambda)}$ is a faithful $KG$-module.
\end{mycor}

\begin{proof}
By Lemma \ref{integralreductionlem}, there exists $\lambda'$ of the form $\lambda' = s_{\alpha_{i}} \cdot \mu$ for some $i$ and for $\mu$ dominant integral, such that $\Ann_{\Ug} \widehat{L(\lambda)} \subseteq \Ann_{\Ug} \widehat{L(\lambda')}$.

By Proposition \ref{regintprop}, $\lambda'$ satisfies condition (1). So $\widehat{L(\lambda')}$ is faithful over $KG$, hence so is $\widehat{L(\lambda)}$.
\end{proof}

\begin{proof}[Proof of Theorem \ref{maintheorem}]
By Corollary \ref{nonregintcor} and Corollary \ref{regintcor}, for any $\lambda \in \mathfrak{h}_{K}^{*}$ such that $\lambda(p^{n} \mathfrak{h}_{R}) \subseteq R$ and $\lambda$ is not dominant integral, $\widehat{L(\lambda)}$ is a faithful $KG$-module.

By Lemma \ref{irreducibleslem}, any infinite-dimensional affinoid highest-weight module has a subquotient isomorphic to some such $\widehat{L(\lambda)}$, and therefore must be a faithful $KG$-module.

\end{proof}

\noindent Theorem \ref{maintheoremapplication} now follows from Theorem \ref{maintheorem}, using the proof of \cite[Theorem 5.0.1]{me}. The statement of \cite[Theorem 5.0.1]{me} uses the assumption that faithfulness of $\widehat{L(\lambda)}$ holds for any $K$ which is a completely discretely valued field extension of $F$, not necessarily finite. However, the proof only uses the faithfulness result for a finite field extension $K' / K$, and so this proof still works in our setting.

\end{document}